\documentclass[final,leqno]{amsart}

\usepackage{graphicx}
\usepackage{amssymb}
\usepackage{amsmath}

\usepackage{hyperref}
\usepackage{mathrsfs}
\usepackage{float}
\usepackage{mathtools}

\usepackage[usenames]{xcolor}

\newtheorem{teo}{Theorem}[section]

\theoremstyle{definition}

\newtheorem{cor}[teo]{Corollary}
\newtheorem{prop}[teo]{Proposition}

\newtheorem{assunp}[teo]{Assumption}
\newtheorem{hypo}{Hypothesis}
\newtheorem{problem}{\bfseries Problem\rmfamily}
\theoremstyle{remark}

\numberwithin{equation}{section}

\newcommand{\mV}{\mathcal{V}}
\newcommand{\mQ}{\mathcal{Q}}
\newcommand{\mK}{\mathcal{K}}
\newcommand{\mH}{\mathcal{H}}

\newcommand{\mF}{\mathcal{F}}
\newcommand{\mG}{\mathcal{G}}
\newcommand{\eu}{\texttt{e}_u}
\newcommand{\ep}{\texttt{e}_p}
\newcommand{\ew}{\texttt{e}_w}
\newcommand{\emm}{\texttt{e}_m}

\DeclareMathOperator{\Ima}{Im}
\newcommand\mA{\mathbb{A}}

\newcommand\mB{\mathbb{B}}
\newcommand\mC{\mathbb{C}}
\newcommand\mL{\mathbb{L}}

\usepackage{array}   
\newcolumntype{C}{>{$}c<{$}} 
\usepackage{threeparttable}
\newcolumntype{A}{>{\centering}p{0.1\textwidth}}
\newcolumntype{B}{>{\centering}p{1cm}}

\begin{document}

\title[A mixed formulation with applications to linear viscoelasticity]{A mixed parameter formulation with applications to linear viscoelasticity}


\author{Erwin Hern\'andez}
\address{Departamento de Matem\'atica, Universidad T\'ecnica Federico Santa Mar\'ia, Casilla 110-V Valparaiso, Chile}
\curraddr{}
\email{erwin.hernandez@usm.cl}
\thanks{The first author has been partially supported by FONDECYT project No.1181098, Chile}

\author{Felipe Lepe}
\address{Departamento de Matem\'atica, Universidad del B\'io-B\'io, Casilla 5-C, Concepci\'on, Chile}
\curraddr{}
\email{flepe@ubiobio.cl}
\thanks{The second author has been partially supported by FONDECYT project No. 11200529, Chile}

\author{Jesus Vellojin}
\address{Departamento de Matem\'atica, Universidad T\'ecnica Federico Santa Mar\'ia, Casilla 110-V Valparaiso, Chile}
\curraddr{}
\email{jesus.vellojinm@usm.cl}
\thanks{The third author has been partially supported by FONDECYT  project No.1181098, Chile}


\date{}

\dedicatory{}


\begin{abstract}
In this work we propose and analyze an abstract parameter dependent model written as a mixed variational formulation based on Volterra integrals of second kind.  For the analysis, we consider a suitable adaptation to the classic mixed theory in the Volterra equations setting, and prove the well posedness of the resulting mixed viscoelastic formulation. Error estimates are derived, using the available results for Volterra equations, where all the estimates are independent of the perturbation parameter. We consider an application of the developed theory in a viscoelastic Timoshenko beam,  and report numerical experiments in order to assess the independence of the perturbation parameter.
\end{abstract}

\maketitle
\section{Introduction}
\label{intro}

Viscoelasticity is a physical property present in a wide variety of structures and became important after the popularization of polymers \cite{meyers132mechanical}. The study of viscoelastic materials, their damping capabilities and behavior in time due to induced stress or temperature changes, is well established and we refer to the studies of Flugge \cite{flugge1975viscoelasticity}, Christensen \cite{christensen2012theory} and Reddy \cite{reddy2007introduction} for a rigorous theoretical development. 

There are several mathematical models and numerical methods to analyze viscoelastic problems. The most common numerical approach is through the finite difference method and the finite element method in order to predict, as accurately, the response of viscoelastic components which depend on their history loading and its service conditions \cite{argyris1991135,chen200051,janovsky199591,van2009properties}. As is stated in \cite{christensen2012theory,shaw1997comparison}, there are two main approaches that relate the strains and stress in viscoelasticity: a constitutive equation in hereditary integral form, or constitutive equations in differential operator form. Shaw et. al \cite{shaw1997comparison} shows that although the two approaches are equivalent at the continuous level, the numerical approximations arising from them will, in general, be different. 
	
	On the other hand, in engineering, slender structures are modeled by considering their thickness behavior (see, for instance, \cite{chapelle2010finite}). This parameter causes difficulties in the convergence of numerical methods, leading to the so-called locking phenomenon. Is this fact that motivates both communities, engineering and mathematician,  to design and analyze locking-free numerical methods in order to approximate correctly the solutions. This is also the case of viscoelastic slender structures.

	For example, an approach dealing with viscoelastic structures where the thickness plays an important role can be found in \cite{castineira2017justification, castineira2019justification}. Here, the authors make use of a Kelvin-Voigt constitutive equation \cite{shillor2004models} along with asymptotic analysis in order to obtain a two dimensional viscoelastic shell as a limit of a quasi-static three dimensional shell model, which includes a long-term memory that takes into account previous deformations. Furthermore, they provide convergence results which justify those equations. Other numerical approach, where a constitutive equation in hereditary integral form and the Laplace transform are used, can be found in \cite{martin2014quasi,martin2016modified,martin2017nonlinear,2017viscoelastic,ayyad2009frictionless}. Also, there are several finite element studies for viscoelastic thin structures where the hereditary integral is approximated using the trapezoidal rule \cite{hernandez2019,payette2010nonlinear,payette2013nonlinear}.

When parameter-dependent formulation appears, mixed finite element methods can be considered to solve it. The fundamental of mixed methods can be found in Boffi et. al \cite{boffi2013mixed}, It is well known that mixed finite element methods can be applied to a wide variety of applications including elastic slender structures (\cite{arnold1981discretization,hernandez2008approximation}). Nevertheless, in the case of integro-differential equations, as appear in viscoelastic problems, the studies working with mixed formulations is scarce.

The research of \cite{sinha2009mixed} analyze a semi-discrete mixed finite element methods for parabolic integro-differential equations that arise in the modeling of nonlocal reactive flows in porous media, deriving error estimates for the pressure and velocity for smooth and non-smooth data using a mixed Ritz-Volterra projection, introduced earlier by Ewing et. al \cite{ewing2002sharp}. Based in this work, Karaa and Pani \cite{karaa2015optimal} proposed a mixed finite element formulation for a class of second order hyperbolic integro-differential equations using a modification of the nonstandard energy formulation of Baker in order to obtain optimal error estimates under minimal smoothness assumptions in the initial data. 

For linear viscoelastic materials, Rognes and Winther \cite{rognes2010}, propose a mixed finite element method enforcing the symmetry of the stress tensor weakly. They obtain a priori error estimates which are tested in several numerical examples. We observe that, in this work,  the thickness of the structure is not considered as an element of analysis in the approximation. Instead, the authors are focused on giving a mixed formulation involving the stress tensor in addition to the displacement to address the convergence issues when working with incompressible or nearly incompressible materials.

To the best of the authors knowledge,  the works reported here, cannot be generalized to cover a wide variety of cases.  In this paper, we propose an abstract setting to analyze numerical and mathematically mixed methods with an integro-differential term. We consider an abstract framework which involves a mixed formulation along with memory terms as hereditary integrals, where the kernel is considered to be bounded. This type of kernel is found, for instance, in viscoelastic material of bounded creep. We adapt the general framework presented in \cite{boffi2013mixed}, by including the integral term, and using the classical Gronwall inequality. We obtain estimates with explicit constants that do not depend on the perturbation parameters and, then,  can be used in different problems where this is an important fact; we show the use of this framework in the viscoelastic Timoshenko beam formulation.

The paper is organized as follows: In Section \ref{section2}  the abstract setting with a parameter perturbation is proposed. This can be considered as an extension to viscoelasticity of the well-known regular and penalty type cases in elasticity. For the later, we provide stability bounds that do not deteriorate when considering the a vanishing penalty parameter. This approach is accompanied with a semi-discretization analysis of the continuous problem that takes into account several common discrete spatial assumptions such as conforming methods or semi-discrete inf-sup conditions, and then we use the stability results to obtain semi-discrete error bounds that can be used to obtain locking-free finite element formulations on bounded creep linear viscoelastic beams. In Section \ref{section:TIMO} the abstract analysis is applied to a Timoshenko beam model in order to obtain an equivalent well-posed mixed formulation, which is semi-discretized and the theoretical convergence are obtained. We end this section with numerical tests in order to assess the performance fo the mixed method for the viscoelastic Timoshenko beam. Finally in Section \ref{section:conclusions} we present some conclusions of our work.

\section{The abstract setting}
\label{section2}

Let $\mV, \mQ$ be two real Hilbert spaces endowed with norms $\Vert \cdot\Vert_{\mV}$ and $\Vert \cdot\Vert_{\mQ}$, respectively. We denote by $\mathcal{L}(\mV ;\mQ)$ the space of continuous linear mappings from $\mV$ to $\mQ$. Also, we denote by $\mV'$ and $\mQ'$ their corresponding dual spaces endowed with norms $\Vert\cdot\Vert_{\mV'}$ and $\Vert \cdot\Vert_{\mQ'}$; $\langle \cdot,\cdot\rangle_{\mV}$  and $\langle \cdot,\cdot\rangle_{\mQ}$ the duality pairing between the spaces $\mV$ and $\mV'$ and $\mQ$ and $\mQ'$; and,  $\mathcal{J}:=[0,T]$ the observation period, with $T\in[0,\infty)$. 

For every Banach space $\mathcal{B}$ and every time interval $[0,t]$, we denote by $L_{[0,t]}^\ell(\mathcal{B})$ the space of maps $ \mathfrak{w}:[0,t]\rightarrow \mathcal{B}$ with norm, for $1\leq \ell <\infty$,
$$
\Vert \mathfrak{w}\Vert_{L_{[0,t]}^\ell(\mathcal{B})}:=\left(\int_{0}^t\Vert \mathfrak{w}\Vert_{\mathcal{B}}^\ell\right)^{1/\ell},
$$
 with the usual modification for $\ell=\infty$. If $\mathcal{S}:=[0,t]$, we will simply write $L_{\mathcal{S}}^\ell(\mathcal{B})$.

The aim of the present section is to establish necessary and sufficient conditions to guarantee the stability of a system of Volterra's integral equations formulated in such a way that the spatial components are analyzed using results of mixed formulations. To accomplish this task,  first we will consider the following general assumptions  (see \cite{boffi2013mixed} for details).
\begin{assunp}
	\label{assumption-1}
	Let $\mV$ and $\mQ$ be two Hilbert spaces.  Let $a(\cdot,\cdot):\mV\times \mV\rightarrow \mathbb{R}$, $b(\cdot,\cdot):\mV\times \mQ\rightarrow\mathbb{R}$ and $c(\cdot,\cdot):\mQ\times \mQ\rightarrow\mathbb{R}$ be three given bilinear forms, and denote by $\mA:\mV\rightarrow\mV$, $\mB:\mV\rightarrow\mQ$ and $\mC:\mQ\rightarrow\mQ$ their corresponding induced linear operators. Assume that the following properties are satisfied:
\begin{itemize}
		\item [i.] The bilinear forms $a(\cdot,\cdot)$ and $c(\cdot,\cdot)$ are symmetric, positive semi-definite and continuous on $\mV$ and $\mQ$, respectively, i.e., 
		\begin{equation*}
			\begin{aligned}
				&\vert a(v,w)\vert \leq \Vert \mA \Vert_{\mathcal{L}(\mV;\mV')} \Vert v\Vert_\mV\Vert w\Vert_\mV \equiv\Vert a \Vert_{\mathcal{L(V\times\mV;\mathbb{R})}}\Vert v\Vert_\mV\Vert w\Vert_\mV, \\
				&\vert c(p,q)\vert \leq \Vert \mC \Vert_{\mathcal{L}(\mQ;\mQ')} \Vert p\Vert_\mQ\Vert q\Vert_\mQ \equiv\Vert c \Vert_{\mathcal{L(Q\times\mQ;\mathbb{R})}}\Vert p\Vert_\mQ\Vert q\Vert_\mQ,
			\end{aligned}
		\end{equation*}
		for all $w,v\in\mV$ and for all $p,q\in\mQ$. We will also require that $a(\cdot,\cdot)$ and $c(\cdot,\cdot)$ are strongly coercive in $\mathcal{K}$ and $\mathcal{H}$, respectively, i.e., there exist $\alpha_0,\gamma_0>0$ such that
		\begin{equation*}
			\begin{aligned}
				&a(v_0,v_0)\geq \alpha_0\Vert v_0\Vert_\mV^2 \hspace{0.4cm} \forall v_0\in \mK,\\
				&c(q_0,q_0)\geq \gamma_0\Vert q_0\Vert_\mQ^2 \hspace{0.4cm} \forall q_0\in \mH.
			\end{aligned}
		\end{equation*}
		The norms of the induced operators and bilinear forms will be denoted simply by $\Vert a \Vert$ and $\Vert c \Vert$. Also,  we have that
		\begin{equation*}
			\begin{aligned}
				&\langle \mA w,v\rangle_{\mV}=\langle w,\mA v\rangle_{\mV}=a(w,v)\hspace{0.4cm}\forall w,v\in\mV,\\
				&\langle \mC p,q\rangle_{\mQ}=\langle p,\mC q\rangle_{\mQ}=c(p,q)\hspace{0.4cm}\forall p,q\in\mQ.
			\end{aligned}
		\end{equation*}
		\item [ii.]  Similarly, the bilinear form $b(\cdot, \cdot)$ is continuous, i.e.,  
		\begin{equation*}
			\vert b(v,q)\vert\leq \Vert \mB\Vert_{\mathcal{L}(\mV;\mQ')} \Vert v\Vert_\mV\Vert q\Vert_\mQ= \Vert b\Vert_{\mathcal{L(V\times}\mQ;\mathbb{R})}\Vert v\Vert_\mV\Vert q\Vert_\mQ.
		\end{equation*}
		Moreover, the  operator $\mB$ satisfies
		\begin{equation*}
			\langle \mB v,q\rangle_{\mQ}=\langle v, \mB^*q\rangle_{\mV}=b(v,q) \hspace{0.4cm}\forall v\in \mV, \forall q\in\mQ.
		\end{equation*}
		\item [iii. ] We consider functions $f$ and $g$ that are continuous on $\mV$ and $\mQ$, respectively, i.e.,
		$$
		\langle f,v\rangle_{\mV}\leq \Vert f \Vert_{\mV'}\Vert v\Vert_\mV \hspace{0.2cm}\forall v\in \mV,
		$$		and 
		$$
		\langle g,q\rangle_{\mQ}\leq \Vert g \Vert_{\mQ'}\Vert q\Vert_\mQ\hspace{0.2cm}\forall q\in \mQ.
		$$
		a.e. in $\mathcal{J}$.
	\end{itemize}
\end{assunp}

We recall the following result based on the Banach closed range Theorem, related to the inf-sup condition and stability constants (See \cite[Chapter 4]{boffi2013mixed}).

\begin{teo}[Banach Closed Range Theorem]
	\label{teo-1}
	Let $\mV$ and $\mQ$ be two real Hilbert spaces and $\mB$ a linear continuous operator from $\mV$ to $\mQ'$. Set
	$$
	\mathcal{K}:=\ker \mB\subset \mV \hspace{0.4cm}\mbox{ and }\hspace{0.4cm}\mathcal{H}:=\ker \mB^*\subset \mQ.
	$$
	Then, the following statements are equivalent:
	\begin{enumerate}
		\item $\Ima \mB$ is closed in $\mQ'$,
		\item $\Ima \mB^*$ is closed in $\mV'$,
		\item There exists a lifting operator $\mL_{\mB}\in \mathcal{L}(\Ima \mB;\mK^{\perp})$ and there exists $\beta>0$ such that $\mB(\mL_{\mB}(g))=g$ and moreover $\beta\Vert\mL_{\mB}g\Vert_{\mV}\leq \Vert g\Vert_{\mQ'}$ for all $g\in\Ima\mB$,
		\item There exists a lifting operator $\mL_{\mB^*}\in \mathcal{L}(\Ima \mB^*;\mH^{\perp})$ and there exists $\beta>0$ such that $\mB^*(\mL_{\mB^*}(f))=f$ and moreover $\beta\Vert\mL_{\mB^*}f\Vert_{\mQ}\leq \Vert f\Vert_{\mV'}$ for all $f\in\Ima\mB^*$.
	\end{enumerate}
\end{teo}
If $\mB$ is surjective, then the Closed Range Theorem has a direct consequence.
\begin{cor}
	\label{cor-1}
	Let $\mV$ and $\mQ$ be two real Hilbert spaces and $\mB$ a linear continuous operator from $\mV$ to $\mQ'$.
	Then, the following statements are equivalent:
	\begin{enumerate}
		\item $\Ima \mB=\mQ'$.
		\item  $ \Ima \mB^*$ is closed and $\mB^*$ is injective.
		\item $\mB^*$ is bounding, i.e., there exists $\beta>0$ such that $\Vert \mB^*q\Vert_{\mV'}\geq \beta \Vert q\Vert_\mQ, \forall q\in \mQ$.
		\item There exists a lifting operator $\mL_{\mB}\in \mathcal{L}(\mQ';\mV)$ such that $\mB(\mL_{\mB}(g))=g$, for all $g\in \mQ'$, with $\Vert \mL_{\mB}\Vert\leq1/\beta$.
	\end{enumerate}
\end{cor} 
We note that Theorem \ref{teo-1} and Corollary \ref{cor-1} are valid a.e. in $\mathcal{J}$.

\subsection{Mixed formulation model}
\label{subsection2-1}
Here, and in the forthcoming analysis, we will omit the time dependence of the solutions and test functions outside the time integral unless necessary in the arguments. 

Let us consider the following problem:

\begin{problem}
	\label{goal-prob-regular}
	Given $f\in L_\mathcal{J}^1(\mV')$ and $g\in L_\mathcal{J}^1(\mQ')$, find $(u,p)\in L_\mathcal{J}^1(\mV\times\mQ)$ such that
	\begin{equation*}
		\left\{
		\begin{aligned}
			a(u,v) + b(v,p) &= \langle f,v\rangle_{\mV} + \int_{0}^{t}\bigg[k_1(t,s) a(u(s),v) + k_2(t,s)b(v,p(s))\bigg]ds,\\
			b(u,q) - c(p,q)&=\langle g,q\rangle_{\mQ} + \int_{0}^{t}\bigg[k_3(t,s) b(u(s),v) - k_4(t,s)c(v,p(s))\bigg]ds
		\end{aligned}\right.
	\end{equation*}
	for all $(v,q)\in\mV\times\mQ.$
\end{problem}

Here, $k_i(\cdot,\cdot),i=1,2,3,4,$ are continuous bounded kernels, i.e., there exists a non-negative constant $C_{k_i}$ such that
$$
\vert k_i(t,s)\vert\leq C_{k_i},\quad i=1,2,3,4.
$$
Also, we assume that $k_i$ is almost everywhere in the triangle (see for instance \cite{gutierrez2014engineering})
\begin{equation}
\label{eq:triangle}
\mathcal{T}:= \bigg\{ \tau \in \mathcal{J} \;\;\vert\;\; 0\leq \tau\leq t,\quad t\in \mathcal{J} \bigg\}.
\end{equation}
Note that when $k_i(t,s)=0 $ a.e $\in\mathcal{T}$, we have a well-known mixed formulation which has been analyzed in \cite[Chapter 4]{boffi2013mixed} in the context of elliptic problems. 

The corresponding operator form of the proposed mixed problem reads as follows:
\begin{problem}
	\label{po-global-prob-regular}
	Given $f\in L_\mathcal{J}^1(\mV')$ and $g\in L_\mathcal{J}^1(\mQ')$, find $(u,p)\in L_\mathcal{J}^1(\mV\times\mQ)$ such that
	\begin{equation*}
		\left\{\begin{aligned}
			\mA u(t) +\mB^*p(t) &=f(t) + \int_{0}^{t}\bigg[k_1(t,s)\mA u(s)+ k_2(t,s)\mB^*p(s)\bigg]ds,\\
			\mB u(t) -\mC p(t)&= g(t)+ \int_{0}^{t}\bigg[k_3(t,s)\mB u(s) -k_4(t,s)\mC p(s)\bigg]ds.
		\end{aligned}\right.
	\end{equation*}
\end{problem}

We now introduce some additional notations and properties derived from Assumption \ref{assumption-1}. 

Let us define the semi-norms
\begin{equation}
	\label{seminorms}
	\vert v\vert_a^2:=a(v,v),\hspace{1cm}\vert q\vert_c^2:=c(q,q).
\end{equation}
Due to the continuity of $a(\cdot,\cdot)$ and $c(\cdot,\cdot)$, it is clear that
\begin{equation}
	\label{continuity-seminorms}
	\vert v\vert_a^2\leq\Vert a\Vert\,\Vert v\Vert_\mV^2,\;\;\;\forall v\in \mV\hspace{1cm}\vert q\vert_c^2\leq\Vert c\Vert\,\Vert q\Vert_\mQ^2,\;\;\;\forall q\in \mQ.
\end{equation}
Also, from \cite[Lemma 4.2.1]{boffi2013mixed} and \eqref{seminorms} we have
\begin{equation}
	\label{continuity-seminorms-bilinear}
	a(u,v)\leq \vert u\vert_a\,\vert v\vert_a \hspace{0.2cm} \mbox{and}\hspace{0.2cm}c(p,q)\leq\vert p\vert_c\,\vert q\vert_c,
\end{equation}
and
\begin{equation}
	\label{continuity-seminorms-operators}
	\Vert \mA u\Vert_{\mV'}^2\leq\Vert a\Vert\,\vert u\vert_a^2 \hspace{0.2cm}\mbox{and}\hspace{0.2cm}\Vert \mC p\Vert_{\mQ'}^2\leq \Vert c\Vert\,\vert p\vert_c^2.
\end{equation}

Since $\mathcal{K}$ and $\mathcal{H}$ are closed subspaces of $\mV$ and $\mQ$, respectively, we know that $\mV=\mathcal{K}\oplus\mathcal{K}^{\bot}$ and $\mQ=\mathcal{H}\oplus\mathcal{H}^{\bot}$. Hence, from Assumption \ref{assumption-1} we have that each $v\in \mV$ and $q\in\mQ$ can be written as 
\begin{equation*}
	\label{split-v-q}
	v=v_0+ \overline{v},\qquad q=q_0+\overline{q},
\end{equation*}
with $v_0\in \mK, \overline{v}\in \mK^{\perp}, q_0\in \mH$ and $\overline{q}\in \mH^{\perp}$. Note that
\begin{equation*}
	\label{bilinear-form-b-K-H}
	b(v,q)=b(\overline{v},q)=b(\overline{v},\overline{q})=b(v,\overline{q}).
\end{equation*}
In a similar way, we split $f\in \mV'$ and $g\in \mQ'$ as
\begin{equation*}
	\label{functions-f-g-K-H}
	f=f_0+\overline{f}\quad\text{ and }\quad g=g_0+\overline{g},
\end{equation*}
with $f_0\in \mK',~ \overline{f}\in(\mK^{\perp})',~ g_0\in \mH'$ and $\overline{g}\in (\mH^{\perp})'$ and, in virtue of the Riesz representation Theorem, we note that
\begin{equation}
	\label{functions-f-g-K-H-2}
	\langle f,v\rangle_{\mV} =\langle f_0,v_0\rangle_{\mV} + \langle \overline{f},\overline{v}\rangle_{\mV} \hspace{0.2cm}\mbox{and}\hspace{0.2cm} \langle g,q\rangle_{\mQ} =\langle g_0,q_0\rangle_{\mQ}  +\langle \overline{g},\overline{q}\rangle_{\mQ}.
\end{equation}

The next stability result will use several arguments from the proof given in \cite[Theorem 4.3.1]{boffi2013mixed}. In such theorem, half of the constants are shown, and the rest are inferred by the symmetry of the problem. In our case, something similar will happen.

\begin{teo}[Main Theorem 1]
	\label{teo8}
	Together with Assumption \ref{assumption-1}, assume that $\Ima\mB$ is closed and there exists  $\beta>0$ such that 
	$$\displaystyle\sup_{v\in \mV}\frac{b(v,q)}{\Vert v\Vert_\mV}\geq \beta \Vert q\Vert_\mQ,\;\;\forall q\in \mH^{\perp}\hspace{0.2cm}\mbox{and}\hspace{0.2cm}\sup_{q\in \mQ}\frac{b(v,q)}{\Vert q\Vert_\mQ}\geq \beta \Vert v\Vert_\mV,\;\;\forall v\in \mK^{\perp}.
	$$
	Then, for all $f\in L_{\mathcal{J}}^1(\mV')$ and $g\in L_{\mathcal{J}}^1(\mQ')$, we have that Problem \ref{goal-prob-regular} has a unique solution that  satisfies
	\begin{equation*}
		\Vert u\Vert_{L_\mathcal{J}^1(\mV)}+\Vert p\Vert_{L_\mathcal{J}^1(\mQ)}\leq C\left( \Vert f\Vert_{L_\mathcal{J}^1(\mV')} + \Vert g\Vert_{L_\mathcal{J}^1(\mQ')}\right),
	\end{equation*}
	where $C$ is a positive constant depending on  $\alpha_0,\beta,\gamma_0$, the continuity constants $\Vert a \Vert, \Vert c\Vert$, the kernel constants $C_{k_i},\,i=1,2,3,4,$ and the time observation $T$. Moreover, there exist positive constants $C_i, i=1,\ldots,16$ such that
	\begin{equation*}
		\Vert \overline{u}\Vert_{L_\mathcal{J}^1(\mV)}\leq C_9\Vert\overline{f}\Vert_{L_\mathcal{J}^1(\mV')} + C_{13}\Vert f_0\Vert_{L_\mathcal{J}^1(\mV')} + C_3\Vert \overline{g}\Vert_{L_\mathcal{J}^1(\mQ')} + C_7\Vert g_0\Vert_{L_\mathcal{J}^1(\mQ')},
	\end{equation*}
	\begin{equation*}
		\Vert u_0\Vert_{L_\mathcal{J}^1(\mV)}\leq C_{10}\Vert\overline{f}\Vert_{L_\mathcal{J}^1(\mV')} + C_{14}\Vert f_0\Vert_{L_\mathcal{J}^1(\mV')} + C_4\Vert \overline{g}\Vert_{L_\mathcal{J}^1(\mQ')} + C_8\Vert g_0\Vert_{L_\mathcal{J}^1(\mQ')},
	\end{equation*}
	\begin{equation*}
		\Vert \overline{p}\Vert_{L_\mathcal{J}^1(\mQ)}\leq C_{11}\Vert\overline{f}\Vert_{L_\mathcal{J}^1(\mV')} + C_{15}\Vert f_0\Vert_{L_\mathcal{J}^1(\mV')} + C_1\Vert \overline{g}\Vert_{L_\mathcal{J}^1(\mQ')} + C_5\Vert g_0\Vert_{L_\mathcal{J}^1(\mQ')},
	\end{equation*}
	\begin{equation*}
		\Vert p_0\Vert_{L_\mathcal{J}^1(\mQ)}\leq C_{12}\Vert\overline{f}\Vert_{L_\mathcal{J}^1(\mV')} + C_{16}\Vert f_0\Vert_{L_\mathcal{J}^1(V')} + C_2\Vert \overline{g}\Vert_{L_\mathcal{J}^1(\mQ')} + C_6\Vert g_0\Vert_{L_\mathcal{J}^1(\mQ')}.
	\end{equation*}
\end{teo}

\begin{proof}
	The theorem assumptions guarantees that the operator $$\begin{bmatrix}
		\mA &\mB^*\\\mB & \mC 
	\end{bmatrix}\in \mathcal{L(\mV\times\mQ;\mV'\times\mQ')}, $$ is invertible. Then, from  \cite[Chapter 2, Section 3]{gripenberg1990volterra} we conclude that there exists a unique pair $(u,p)\in L_\mathcal{J}^1(\mV\times\mQ)$ solution to Problem \ref{goal-prob-regular} (resp. to Problem \ref{po-global-prob-regular}). 
	
	To obtain the estimates of $\overline{u}$, $u_0$, $\overline{p}$ and $p_0$, we will adapt the proof of  \cite[Theorem 4.3.1]{boffi2003finite} to our case. We begin considering as a first case, $f=0$, where the  obtained estimates are similar to those  for the case $g=0$, due to the symmetry as we will see below. 
	
	Since $f=0$, we take norms in the first equation of Problem \ref{po-global-prob-regular} and use the boundedness of the linear operators and the kernels, in order to obtain
	$$
	\begin{aligned}
		\Vert \mA u + \mB^*p\Vert_{\mV'}\leq C_{\widetilde{k}}\int_{0}^t\Vert \mA u(s)+\mB^*p(s)\Vert_{\mV'}ds
	\end{aligned}
	$$
	where $C_{\widetilde{k}}=\max\{C_{k_1},C_{k_2} \}$. Then, from Gronwall's lemma we obtain that
	\begin{equation}
		\label{teo_8-001}
		\mA u +\mB^*p=0\qquad \text{ a. e. in }\mathcal{J},
	\end{equation}
	or equivalently,
	\begin{equation}
		\label{teo_8-002}
		a(u,v)+ b(v,p)=0 \qquad \forall v\in\mV.
	\end{equation}
	Testing \eqref{teo_8-002} with $u_0(t)\in \mK\subset{\mV}$ we have that
	\begin{equation*}
		\label{teo-goal-001}
		a(u,u_0)=-b(u_0,p)=0.
	\end{equation*}
	Then, from \eqref{seminorms} it follows that
	\begin{equation*}
		\label{teo-goal-002}
		\begin{aligned}
			\vert u_0\vert_a^2=a(u,u_0)-a(\overline{u},u_0)\leq \vert\overline{u}\vert_a\vert u_0\vert_a.
		\end{aligned}
	\end{equation*}
	Using the estimates $\vert v\vert_a\leq\Vert a\Vert^{1/2}\Vert v\Vert_{\mV}$ and $\alpha_0^{1/2}\Vert v\Vert_{\mV}\leq \vert v\vert_a$, we have that
	\begin{equation}
		\label{teo-goal-005}
		\Vert u_0\Vert_{\mV}\leq\frac{\Vert a \Vert^{1/2}}{\alpha_0^{1/2}}\Vert\overline{u}\Vert_{\mV}.
	\end{equation}

	On the other hand, by testing \eqref{teo_8-002} with $v=u(t)\in\mV$, along with \eqref{teo_8-002} and the continuity constants of $a(\cdot,\cdot)$ and $b(\cdot,\cdot)$, we have that
	\begin{equation}
		\label{teo-goal-006}
		\begin{aligned}
			a(u,u)+b(u,p)&\leq I(t)\bigg(\Vert u\Vert_{\mV}+\Vert p\Vert_{\mQ}\bigg),
		\end{aligned}
	\end{equation}
	where
	$$
	I(t)=\max\{C_{\widetilde{k}},C_{\overline{k}}\}\max\left\{\Vert a\Vert,\Vert b\Vert,\Vert c\Vert\right\}\bigg[\int_{0}^t\Vert u(s)\Vert_{\mV}+\int_{0}^t\Vert p(s)\Vert_{\mQ}ds\bigg].
	$$
	where $C_{\overline{k}}=\max\{C_{k_3},C_{k_4} \}$. Similarly, by testing the second equation in Problem \ref{goal-prob-regular} with $q=p(t)\in\mQ$, we have that
	\begin{equation}
		\label{teo-goal-007}
		\begin{aligned}
			b(u,p)-c(p,p)&\leq \langle g,p\rangle_{\mQ} + I(t)\bigg(\Vert u\Vert_{\mV}+\Vert p\Vert_{\mQ}\bigg).
		\end{aligned}
	\end{equation}
	Thus, subtracting \eqref{teo-goal-007} from \eqref{teo-goal-006} and using \eqref{continuity-seminorms-operators}, yield to
	\begin{equation}
		\label{teo8-eq-6}
		\frac{\Vert \mA u\Vert_{\mV'}^2}{\Vert a \Vert} + \frac{\Vert \mC p\Vert_{\mQ'}^2}{\Vert c\Vert}\leq -\langle g, p\rangle_{\mQ}.
	\end{equation}
	We now divide the analysis. Let us consider $g_{0}=0$.
	\\
	\underline{\textbf{\textit{Case 1}}}: We consider $g_0=0$.	From \eqref{functions-f-g-K-H-2} we rewrite \eqref{teo8-eq-6} as follows
	\begin{equation}
		\label{teo8-eq-7}
		\frac{\Vert \mA u\Vert_{\mV'}^2}{\Vert a \Vert} + \frac{\Vert \mC p\Vert_{\mQ'}^2}{\Vert c\Vert}\leq -\langle \overline{g}, \overline{p}\rangle_{\mQ}.
	\end{equation}
	Using \eqref{teo_8-001}, inequality \eqref{teo8-eq-7}, the inf-sup condition on $\mB^*\overline{p}$, the fact that $\mB^*p=\mB^*\overline{p}$, since $p-\overline{p}=p_0\in \mathcal{H}$, we obtain
	\begin{equation*}
		\label{teo-goal-008}
		\begin{aligned}
			\Vert \mB^*\overline{p}\Vert_{\mQ}^2=\Vert \mB^*p\Vert_{\mV'}^2=\Vert \mA u\Vert_{\mV'}^2	\leq \Vert a\Vert\, \Vert\overline{g}\Vert_{\mQ'}\Vert\overline{p}\Vert_{\mQ}\leq\frac{1}{\beta} \Vert a\Vert\, \Vert\overline{g}\Vert_{\mQ'}\Vert\mB^*\overline{p}\Vert_{\mV'}.
		\end{aligned}
	\end{equation*}
	Hence, from the inf-sup condition of $\mB^*$ we have that
	\begin{equation*}
		\label{teo-goal-009}
		\Vert \overline{p}\Vert_{\mQ}\leq \frac{1}{\beta}\Vert \mB^*p\Vert_{\mV'}\leq C_1\Vert \overline{g}\Vert_{\mQ'},
	\end{equation*}
	where $C_1:=\Vert a \Vert /\beta^2$. Hence, integrating in $\mathcal{J}$ yields
	\begin{equation}
		\label{teo-goal-estimate-1}
		\Vert \overline{p}\Vert_{L_\mathcal{J}^1(\mQ)}\leq C_1 \Vert \overline{g}\Vert_{L_\mathcal{J}^1(\mQ')}.
	\end{equation}
	
	The next step is to obtain a bound for the remaining term $p_{0}$. To accomplish this task, testing the second equation in Problem \ref{goal-prob-regular} with $q=p_0(t)$, we observe that $b(u,p_0)=\langle \overline{g},p_0\rangle_{\mQ}=0$, since $p_0\in\mathcal{H}$. The splitting $c(p,p_0)=c(p_0,p_0)+c(\overline{p},p_0)$, in combination with \eqref{seminorms}, implies that
	\begin{equation*}
		\label{teo8-eq-11}
		\begin{aligned}
			\vert p_0\vert_c^2=c(p_0,p_0)&=c(p,p_0)-c(\overline{p},p_0)\\
			&=\int_{0}^tk_4(t,s)\bigg[c(p_0(s),p_0(t))+c(\overline{p}(s),p_0(t)) \bigg]ds - c(\overline{p},p_0).
		\end{aligned}
	\end{equation*}
	Then, from \eqref{continuity-seminorms-bilinear} we have that
	\begin{equation*}
		\label{teo8-eq-12}
		\begin{aligned}
			\vert p_0\vert_c&\leq C_{k_4}\int_{0}^{t}\vert p_0(s)\vert_c ds + C_{k_4}\int_{0}^T\vert\overline{p}(s)\vert_c ds + \vert \overline{p}\vert_c.
		\end{aligned}
	\end{equation*}
	Therefore, applying the Gronwall's lemma, allow us to conclude that
	\begin{equation*}
		\label{teo8-eq-13}
		\vert p_0\vert_c\leq \vert \overline{p}\vert_c + C_{k_4}\bigg[1 +e^{C_{k_4}T}(1+C_{k_4}T) \bigg]\int_{0}^{T}\vert \overline{p}(s)\vert_cds.
	\end{equation*}
	Using that $\Vert p_0\Vert_{\mQ}\leq \frac{1}{\gamma_0^{1/2}}\vert p_0\vert_c$ and $\vert \overline{p}\vert_c\leq \Vert c\Vert^{1/2}\Vert \overline{p}\Vert_{\mQ}$, in the previous inequality, implies that
	\begin{equation}
		\label{teo8-eq-14}
		\Vert p_0\Vert_{\mQ}\leq \frac{\Vert c\Vert^{1/2}}{\gamma_0^{1/2}}\bigg\{\Vert \overline{p}\Vert_{\mQ} + C_{k_4}\bigg[1 +e^{C_{k_4}T}(1+C_{k_4}T) \bigg] \int_{0}^{T}\Vert \overline{p}(s)\Vert_{\mQ}\,ds\bigg\}.
	\end{equation}
	Then, the desired bound for $\Vert p_0\Vert_{L_\mathcal{J}^1(\mQ)}$ follows by integrating \eqref{teo8-eq-14} in $\mathcal{J}$ and \eqref{teo-goal-estimate-1}, this is
	\begin{equation*}
		\label{teo8-eq-estimate2}
		\Vert p_0\Vert_{L_\mathcal{J}^1(\mQ)}\leq C_{2}\Vert \overline{g}\Vert_{L_\mathcal{J}^1(\mQ')},
	\end{equation*}
	where the constant $C_{2}$ is given by
	\begin{equation*}
		\label{teo-goal-c2}
		C_2:= C_1\frac{\Vert c\Vert^{1/2}}{\gamma_0^{1/2}}\bigg\{1 +C_{k_4}T\bigg[1+e^{C_{k_4}T}(1+C_{k_4}T)\bigg] \bigg\}\Vert \overline{g}\Vert_{L_\mathcal{J}^1(\mQ')}.
	\end{equation*}
	
	Our next goal is to obtain estimates for $u_0$ and $\overline{u}$. To accomplish this task, first we observe that from \eqref{teo8-eq-7} and the second equation in Problem \ref{po-global-prob-regular}, we derive the following estimate
	\begin{equation*}
		\label{teo8-eq-15}
		\begin{aligned}
			\bigg\Vert \mB u-\overline{g}- \int_{0}^{t}\bigg[k_3(t,s)\mB u(s)-k_4(t,s)\mC p(s) \bigg]ds\bigg\Vert_{\mQ'}^2&=\Vert \mC p\Vert_{\mQ}^2\leq \frac{\mu^2}{\beta^2}\Vert \overline{g}\Vert_{\mQ'}^2
		\end{aligned}
	\end{equation*}
	where $\mu^2=\Vert c \Vert\,\Vert a\Vert$. From the triangle inequality, the continuity of $\mC$, and the split $p=p_0+\overline{p}$, we obtain
	\begin{equation*}
		\label{teo-goal-0012}
		\begin{aligned}
			&\Vert \mB \overline{u}\Vert_{\mQ'}\leq \left(\frac{\mu+\beta}{\beta}\right)\Vert\overline{g}\Vert_{\mQ'} + C_{k_4}\int_{0}^t\Vert \mC p(s)\Vert_{\mQ'}ds + C_{k_3}\int_{0}^t\Vert \mB \overline{u}(s)\Vert_{\mQ'}ds\\
			&\leq \left(\frac{\mu+\beta}{\beta}\right)\Vert\overline{g}\Vert_{\mQ'}+C_{k_4}\Vert c\Vert\int_{0}^t\bigg(\Vert p_0(s)\Vert_{\mQ}+\Vert \overline{p}(s)\Vert_{\mQ}\bigg)ds+C_{k_3}\int_{0}^t\Vert \mB \overline{u}(s)\Vert_{\mV'}\\
			&\leq \left(\frac{\mu+\beta}{\beta}\right)\Vert\overline{g}\Vert_{\mQ'} + (C_1+C_2)C_{k_4}\Vert c\Vert \Vert\overline{g}\Vert_{L_\mathcal{J}^1(\mQ')}+C_{k_3}\int_{0}^t\Vert \mB \overline{u}(s)\Vert_{\mV'}.
		\end{aligned}
	\end{equation*}
	Since $\overline{u}\in\mH^{\perp}$, we apply the Gronwall's lemma, together with the inf-sup condition of $\mB$, in order to obtain
	\begin{equation}
		\label{teo-goal-estimate-overline-u}
		\Vert \overline{u}\Vert_{L_\mathcal{J}^1(\mV)}\leq C_3\Vert\overline{g}\Vert_{L_\mathcal{J}^1(\mQ')},
	\end{equation}
	where $C_3$ is defined by
	\begin{equation*}
		\label{teo-goal-c3}
		C_3:=\frac{1}{\beta}\left(1+Te^{C_{k_3}T}\right)\left[\frac{\mu+\beta}{\beta}+ C_{k_4}T\Vert c\Vert (C_1+C_2) \right].
	\end{equation*}
	
	Hence, by gathering \eqref{teo-goal-005} and \eqref{teo-goal-estimate-overline-u}, we derive the following estimate for $u_0$
	\begin{equation*}
		\label{teo-goal-estimate-u-0}
		\Vert u_0\Vert_{L_\mathcal{J}^1(\mV)}\leq C_4\Vert \overline{g}\Vert_{L_\mathcal{J}^1(\mQ')},
	\end{equation*}
	where $C_4:=(\Vert a \Vert/\alpha_0)^{1/2}$.
	
	\underline{\textbf{\textit{Case 2}}}: We consider $\overline{g}=0$. From \eqref{functions-f-g-K-H-2} and \eqref{teo8-eq-7}  we write
	\begin{equation}
		\label{teo8-eq-19}
		\frac{\Vert \mA u\Vert_{\mV'}^2}{\Vert a \Vert} + \frac{\Vert \mC p\Vert_{\mQ'}^2}{\Vert c\Vert}\leq -\langle g_0, p_0\rangle_{\mQ}.
	\end{equation}
	Note that $p_0\in\mathcal{H}=\ker \mB^*$. Then, using \eqref{seminorms}, \eqref{continuity-seminorms}, \eqref{continuity-seminorms-bilinear}, and the continuity of $c(\cdot,\cdot)$,  we have that
	\begin{equation*}
		\label{teo8-eq-20}
		\begin{aligned}
			\vert p_0\vert_c^2&=c(p,p_0)-c(\overline{p},p_0)\\
			&=-\langle g_0,p_0\rangle_{\mQ}- \int_{0}^tk_4(t,s)\bigg[c(p_0(s),p_0) + c(\overline{p}(s),p_0)\bigg]ds - c(\overline{p},p_0)\\
			&\leq\frac{1}{\gamma_0^{1/2}}\Vert g_0\Vert_{\mQ'}\vert p_0\vert_c +C_{k_4} \int_{0}^t\bigg(\vert p_0(s)\vert_{c}+\vert \overline{p}(s)\vert_{c} \bigg)\,ds\vert p_0\vert_{c} + \vert\overline{p}\vert_{c}\vert p_0\vert_{c}.
		\end{aligned}
	\end{equation*}
	Using again that $\Vert p_0\Vert_{\mQ}\leq \frac{1}{\gamma_0^{1/2}}\vert p_0\vert_c$, together with \eqref{continuity-seminorms}, yields
	\begin{equation*}
		\label{teo8-eq-21}
		\begin{aligned}
			&\Vert p_0\Vert_{\mQ}\leq \frac{1}{\gamma_0}\Vert g_0\Vert_{\mQ'}+\frac{C_{k_4}\Vert c\Vert^{1/2}}{\gamma_0^{1/2}}\int_{0}^t\bigg[\Vert p_0(s)\Vert_{\mQ}+\Vert \overline{p}(s)\Vert_{\mQ} \bigg]ds + \frac{\Vert c\Vert^{1/2}}{\gamma_0^{1/2}}\Vert \overline{p}\Vert_{\mQ}.\\
		\end{aligned}
	\end{equation*}
	Hence, applying Gronwall's lemma to the inequality above we obtain
	\begin{equation}
		\label{teo8-eq-24}
		\Vert p_0\Vert_{\mQ}\leq R(t) + \frac{C_{k_4}\Vert c\Vert^{1/2}e^{\displaystyle\frac{C_{k_4}\Vert c\Vert^{1/2}T}{\gamma_0^{1/2}}}}{\gamma_0^{1/2}}\int_{0}^t R(s)\,ds.
	\end{equation}
	where
	\begin{equation*}
		\displaystyle R(t):=\frac{1}{\gamma_0}\Vert g_0\Vert_{\mQ'}+ \frac{\Vert c\Vert^{1/2}}{\gamma_0^{1/2}}\Vert \overline{p}\Vert_{\mQ}+\frac{C_{k_4}\Vert c\Vert^{1/2}}{\gamma_0^{1/2}}\int_{0}^t\Vert \overline{p}(s)\Vert_{\mQ}.
	\end{equation*}	
	Now, integrating \eqref{teo8-eq-24} in $\mathcal{J}$, we derive the following estimate for $p_0$,
	\begin{equation}
		\label{teo8-eq-estimate6-pre}
		\begin{aligned}
			&\Vert p_0\Vert_{L_\mathcal{J}^1(\mQ)}\leq C_{p_0}\bigg[\frac{1}{\gamma_0}\Vert g_0\Vert_{L_\mathcal{J}^1(\mQ')} + \frac{\Vert c\Vert^{1/2}}{\gamma_0^{1/2}}(1+ TC_{k_4})\Vert \overline{p}\Vert_{L_\mathcal{J}^1(\mQ)} \bigg],
		\end{aligned}
	\end{equation}
	where
	$$
	C_{p_0}:=1+\frac{TC_{k_4}\Vert c\Vert^{1/2}e^{\displaystyle\frac{C_{k_4}\Vert c\Vert^{1/2}T}{\gamma_0^{1/2}}}}{\gamma_0^{1/2}}.
	$$
	
	The following step is to estimate $\overline{p}$, and consequently $p_0$. To accomplish this task, we observe that $\mB^* p=\mB^* \overline{p}$. Then, from \eqref{teo_8-001} along with \eqref{teo8-eq-19},  we have that
	\begin{equation}
		\label{teo-goal-0015}
		\begin{aligned}
			\Vert\mB^*\overline{p}\Vert_{\mV'}^2&=\Vert\mB^*p\Vert_{\mV'}^2= \Vert \mA u\Vert_{\mV'}^2\leq \Vert a\Vert\,\Vert g_0\Vert_{\mQ'}\Vert p_0\Vert_{\mQ}.
		\end{aligned}
	\end{equation}
	Hence, using the Cauchy inequality $ab\leq \varepsilon a^2+\frac{b^2}{4\varepsilon}$ with $\varepsilon=\beta^2\gamma_0/4\Vert c\Vert$ in the right side of \eqref{teo-goal-0015},  results
	\begin{equation}
		\label{teo-8-0017}
		\begin{aligned}
			\Vert a\Vert\,\Vert g_0\Vert_{\mQ'}\Vert p_0\Vert_{\mQ}&\leq \frac{\Vert c\Vert\,\Vert a\Vert^2\Vert g_0\Vert_{\mQ'}^2}{\beta^2\gamma_0}+\frac{\beta^2\gamma_0\Vert p_0\Vert_{\mQ}^2}{4\Vert c\Vert}\\
			&\leq \left( \frac{\Vert c\Vert^{1/2}\,\Vert a\Vert\Vert g_0\Vert_{\mQ'}}{\beta\gamma_0^{1/2}}+\frac{\beta\gamma_0^{1/2}\Vert p_0\Vert_{\mQ}}{2\Vert c\Vert^{1/2}} \right)^2.
		\end{aligned}
	\end{equation}
	Thanks to the inf-sup condition of $\mB^*$ we have that $\beta\Vert \overline{p}\Vert_{\mQ}\leq\Vert \mB^*\overline{p}\Vert_{\mV'}$. This fact, together with \eqref{teo8-eq-24}, \eqref{teo-goal-0015} and \eqref{teo-8-0017}, allow us to obtain
	\begin{equation*}
		\begin{aligned}
			\beta\Vert \overline{p}\Vert_{\mQ}&\leq \frac{\Vert c\Vert^{1/2}\,\Vert a\Vert\Vert g_0\Vert_{\mQ'}}{\beta\gamma_0^{1/2}}+\frac{\beta\gamma_0^{1/2}\Vert p_0\Vert_{\mQ}}{2\Vert c\Vert^{1/2}}\\
			&\leq \frac{\Vert c\Vert^{1/2}\,\Vert a\Vert\Vert g_0\Vert_{\mQ'}}{\beta\gamma_0^{1/2}}+\frac{\beta\gamma_0^{1/2}R(t)}{2\Vert c\Vert^{1/2}}+  \frac{\beta C_{k_4}e^{\displaystyle\frac{C_{k_4}\Vert c\Vert^{1/2}T}{\gamma_0^{1/2}}}}{2}\int_{0}^t R(s)ds.
		\end{aligned}
	\end{equation*}
	Replacing $R(t)$ and rearranging terms, yields to
	$$
	\Vert \overline{p}\Vert_{\mQ}\leq G(t)+\chi\int_{0}^t\Vert \overline{p}(s)\Vert_{\mQ}\,ds,
	$$
	where
	$$
	G(t):=\frac{2}{\gamma_0}\left(\frac{\Vert c\Vert^{1/2}\Vert a\Vert}{\beta^2}+\frac{1}{2\Vert c\Vert^{1/2}} \right)\Vert g_0\Vert_{\mQ'}+\frac{C_{k_4}e^{\displaystyle\frac{C_{k_4}\Vert c\Vert^{1/2}T}{\gamma_0^{1/2}}}}{\gamma_0}\int_{0}^t\Vert g_0(s)\Vert_{\mQ'}ds,
	$$
	and $\chi=C_{k_4}\left[1+ C_{p0} \right]$. Then, the Gronwall's lemma yields to
	\begin{equation}
		\label{teo-goal-0018}
		\Vert \overline{p}\Vert_{\mQ}\leq G(t) + \chi e^{ T\chi }\int_{0}^tG(s)ds.
	\end{equation}
	Integrating \eqref{teo-goal-0018} in $\mathcal{J}$, we obtain the estimate
	\begin{equation}
		\label{teo-goal-estimate-overline-p-2}
		\Vert \overline{p}\Vert_{L_\mathcal{J}^1(\mV)}\leq C_5\Vert g_0\Vert_{L_\mathcal{J}^1(\mQ')},
	\end{equation}
	where
	\begin{equation*}
		\label{teo-goal-estimate-c5}
		C_5:=\left(1+T\chi e^{T\chi }\right)\left[\frac{1}{\gamma_0}\left(\frac{\Vert c\Vert^{1/2}\Vert a\Vert}{\beta^2}+\frac{1}{2\Vert c\Vert^{1/2}} \right)+\frac{TC_{k_4}e^{\displaystyle\frac{C_{k_4}\Vert c\Vert^{1/2}T}{\gamma_0^{1/2}}}}{2\gamma_0}\right].
	\end{equation*}
	
	Since we have now an estimate for $\overline{p}$, then an estimate for $p_0$ is easily followed from \eqref{teo8-eq-estimate6-pre} and \eqref{teo-goal-estimate-overline-p-2}, i.e., we have that 
	\begin{equation}
		\label{teo-goal-estimate-p0-2}
		\Vert p_0\Vert_{L_\mathcal{J}^1(\mV)}\leq C_6\Vert g_0\Vert_{L_\mathcal{J}^1(\mQ')},
	\end{equation}
	where
	\begin{equation*}
		\label{teo-goal-estimate-c6}
		C_6:=C_{p0}\bigg[\frac{1}{\gamma_0} + \frac{\Vert c\Vert^{1/2}}{\gamma_0^{1/2}}(1+ TC_{k_4})C_5 \bigg].
	\end{equation*}
	
	Now it only remains to provide estimates for $u_0$ and $\overline{u}$. To do this task, we first rewrite the second equation in Problem \ref{po-global-prob-regular} as $$\mB u=\mC p + g_0 + \int_{0}^{t}\bigg[k_3(t,s)\mB u(s)-k_4(t,s)\mC p(s) \bigg]ds,$$ and take the $\mQ'$ norm to obtain
	\begin{equation*}
		\label{teo8-eq-43}
		\Vert \mB u\Vert_{\mQ'}\leq \Vert \mC p\Vert_{\mQ'} + \Vert g_0\Vert_{\mQ'} + C_{k_4}\Vert \mC p\Vert_{L_\mathcal{J}^1(\mQ')}+C_{k_3}\int_{0}^t\Vert \mB u(s)\Vert_{\mQ'}ds.
	\end{equation*}
	Then, by Gronwall's lemma we conclude that
	\begin{equation}
		\label{teo8-eq-46}
		\Vert \mB u\Vert_{\mQ'}\leq M(t) + C_{k_3}e^{C_{k_3}T}\int_{0}^tG(s)ds,
	\end{equation}
	where $
	M(t):=\Vert \mC p\Vert_{\mQ'} + \Vert g_0\Vert_{\mQ'} + C_{k_4}\Vert \mC p\Vert_{L_\mathcal{J}^1(\mQ')}.
	$
	After integrating \eqref{teo8-eq-46} in $\mathcal{J}$, we obtain
	\begin{multline}
		\label{teo8-eq-48}
		\Vert \mB u\Vert_{L_\mathcal{J}^1(\mQ')}\leq\bigg(1+ TC_{k_3}e^{C_{k_3}T}\bigg)\left[(1+C_{k_4}T)\Vert \mC p\Vert_{L_\mathcal{J}^1(\mQ')} + \Vert g_0\Vert_{L_\mathcal{J}^1(\mQ')}\right].
	\end{multline}
	
	On the other hand, from \eqref{teo8-eq-19} we have that 
	\begin{equation*}
		\label{teo8-eq-40}
		\Vert \mC p\Vert_{\mQ'}^2\leq \Vert c\Vert\Vert g_0\Vert_{\mQ'}\Vert p_0\Vert_{\mQ}\leq \frac{\Vert c\Vert}{2}\bigg(\Vert g_0\Vert_{\mQ'} + \Vert p_0\Vert_{\mQ}\bigg)^2.
	\end{equation*}
	Then, invoking \eqref{teo-goal-estimate-p0-2}, we obtain the following for $\mC p$ in $L_\mathcal{J}^1(\mQ')$:
	\begin{equation*}
		\label{teo8-eq-42}
		\begin{aligned}
			\Vert \mC p\Vert_{L_\mathcal{J}^1(\mQ')}\leq\frac{\Vert c\Vert^{1/2}}{\sqrt{2}}(1+C_6)\Vert g_0\Vert_{L_\mathcal{J}^1(\mQ')}.
		\end{aligned}
	\end{equation*}
	Replacing the estimate above in \eqref{teo8-eq-48}, we obtain
	$$
	\Vert \mB u\Vert_{L_\mathcal{J}^1(\mQ')}\leq\bigg(1+ TC_{k_3}e^{C_{k_3}T}\bigg)\left[\frac{\Vert c\Vert^{1/2}}{\sqrt{2}}(1+C_6)(1+C_{k_4}T) + 1 \right]\Vert g_0\Vert_{L_\mathcal{J}^1(\mQ')}.
	$$
	
	The fact that $\mB u=\mB \overline{u}$ and $\overline{u}\in \mathcal{K}^{\perp}$ implies that $\beta\Vert\overline{u}\Vert_{L_\mathcal{J}^1(\mV)}\leq \Vert \mB \overline{u}\Vert_{L_\mathcal{J}^1(\mQ')}$, from which we obtain the following estimate for $\overline{u}$:
	\begin{equation*}
		\label{teo8-eq-estimate7}
		\Vert \overline{u}\Vert_{L_\mathcal{J}^1(\mV)}\leq C_7\Vert g_0\Vert_{L_\mathcal{J}^1(\mQ')},
	\end{equation*}
	where
	\begin{equation*}
		\label{teo8-eq-estimate-constant}
		C_7:=\frac{(1+ TC_{k_3}e^{C_{k_3}T})}{\beta}\left[\frac{\Vert c\Vert^{1/2}}{\sqrt{2}}(1+C_6)(1+C_{k_4}T) + 1 \right].
	\end{equation*}
	Consequently, the estimate for $u_0$ is easily obtained from \eqref{teo-goal-005} as
	\begin{equation*}
		\label{teo8-eq-estimate8}
		\Vert u_0\Vert_{L_\mathcal{J}^1(\mV)}\leq C_8\Vert g_0\Vert_{L_\mathcal{J}^1(\mQ')},
	\end{equation*}
	where $C_8:=C_7(\Vert a \Vert/\alpha_0)^{1/2}$.
	
	To obtain the rest of the constants it is enough to repeat the same arguments used but in the $f=0$ case. In this sense, the resulting constants will be similar to those obtained previously by exchanging $\Vert a\Vert$ for $\Vert c\Vert$ and $\gamma_0$ for $\alpha_0$. With respect to the kernels, it is enough to exchange $k_3$ for $k_2$ and $k_4$ for $k_1$. In fact, the remaining constants are:
	$$
	\begin{aligned}
	&C_9:= \Vert c \Vert /\beta^2, \quad C_{10}:= C_9\frac{\Vert a\Vert^{1/2}}{\alpha_0^{1/2}}\bigg\{1 +C_{k_1}T\bigg[1+e^{C_{k_1}T}(1+C_{k_1}T)\bigg] \bigg\},\\
	&C_{11}:=\frac{1}{\beta}\left(1+Te^{C_{k_2}T}\right)\left[\frac{\mu+\beta}{\beta}+ C_{k_1}T\Vert c\Vert (C_9+C_{10}) \right], \quad C_{12}:=\Vert c \Vert^{1/2}/\gamma_0^{1/2},\\
	&C_{13}:=\left(1+T\chi e^{T\chi }\right)\left[\frac{1}{\alpha_0}\left(\frac{\Vert a\Vert^{1/2}\Vert c\Vert}{\beta^2}+\frac{1}{2\Vert a\Vert^{1/2}} \right)+\frac{TC_{k_1}e^{\displaystyle\frac{C_{k_1}\Vert a\Vert^{1/2}T}{\alpha_0^{1/2}}}}{2\alpha_0}\right],\\
	&C_{14}:=C_{u0}\bigg[\frac{1}{\alpha_0} + \frac{\Vert a\Vert^{1/2}}{\alpha_0^{1/2}}(1+ TC_{k_1})C_{13} \bigg],\quad C_{16}=C_{15}\Vert c \Vert^{1/2}/\gamma_0^{1/2},\\
	&C_{15}:=\frac{(1+ TC_{k_2}e^{C_{k_2}T})}{\beta}\left[\frac{\Vert a\Vert^{1/2}}{\sqrt{2}}(1+C_{14})(1+C_{k_1}T) + 1 \right].
	\end{aligned}
	$$
	where now $\chi=C_{k_1}\left[1+ C_{u_0} \right]$, with 
		$$
	C_{u_0}:=1+\frac{TC_{k_1}\Vert a\Vert^{1/2}e^{\displaystyle\frac{C_{k_1}\Vert a\Vert^{1/2}T}{\alpha_0^{1/2}}}}{\alpha_0^{1/2}}.
	$$
\end{proof}

\subsection{Parameter-dependent problem}

In this section we will obtain stability energy-type estimates of the following problem:
\begin{problem}	Given $f\in L_\mathcal{J}^1(\mV')$ and $g\in L_\mathcal{J}^1(\mQ')$, find $(u,p)\in L_\mathcal{J}^1(\mV\times\mQ)$ such that
	\label{prob2-weak}
	\begin{equation*}
		\label{prob2-weak-formulation}
		\left\{
		\begin{aligned}
			&a(u,v)+b(v,p)=\langle f,v\rangle_{\mV}+\int_{0}^t\bigg[k_1(t,s) a(u(s),v)+k_2(t,s)b(v,p(s))\bigg]\,ds,\\
			&b(u,q) -\lambda (p,q)_{\mQ}=\langle g,q\rangle_{\mQ},
		\end{aligned}
		\right.
	\end{equation*}
	for all $(v,q)\in\mV\times\mQ.$
\end{problem}

Note that this problem is a particular, but very important case of the previously studied model. Indeed, assuming that the bilinear form $c(\cdot,\cdot)$ is given by 
\begin{equation}
	\label{bilinear_form_inner_product_Q}
	c(p,q)=\lambda(p,q)_\mQ, \,\,\,\lambda> 0,
\end{equation}
where $(\cdot,\cdot)_\mQ$ denotes the inner product in $\mQ$, and taking $k_3,k_4\equiv 0$ in Problem \ref{goal-prob-regular}, we arrive at Problem \ref{prob2-weak}. It is important to observe that this modifications allow to obtain estimates directly from Theorem \ref{teo8} at the expenses of having $\lambda$ in the denominator of several terms. 

Let us define a regular perturbation of the form
\begin{equation*}
	c(\cdot,\cdot)=\lambda\widetilde{c}(\cdot,\cdot),
\end{equation*}
where $\widetilde{c}:\mQ\times\mQ\rightarrow\mathbb{R}$ is a continuous $\mH-$elliptic bilinear form with ellipticity constant $\widetilde{\gamma}_{0}$ and continuity constant $\Vert \widetilde{c}\Vert$. Then, $\gamma_{0}=\lambda\widetilde{\gamma}_0$ and $\Vert c\Vert=\lambda\Vert \widetilde{c}\Vert$ can be seen as the ellipticity and continuity constants of $c(\cdot,\cdot)$, respectively. Hence, from Theorem \ref{teo8} we obtain that the bounding constants are dependent on $\lambda^{-1}$, which clearly is an important drawback since this parameter deteriorates the constants in the estimates. Therefore, the analysis of this model will take a different path from that of the problem with regular perturbation, in order to obtain uniform bounds with respect to $\lambda$.

Let us introduce the following problem:

\begin{problem}\label{prob2}
	Given $f\in L_\mathcal{J}^1(\mV')$ and $g\in L_\mathcal{J}^1(\mQ')$, find $(u,p)\in L_\mathcal{J}^1(\mV\times\mQ)$ such that
	\begin{equation*}
		\label{goal-mixed-problem1}
		\left\{\begin{aligned}
			\mA u(t) +\mB^*p(t) &=f(t) + \int_{0}^{t}\bigg[k_1(t,s)\mA u(s)+ k_2(t,s) \mB^*p(s)\bigg]\,ds,\\
			\mB u(t) -\lambda\mathfrak{R}_{\mQ} p(t)&= g(t).
		\end{aligned}\right.
	\end{equation*}
\end{problem}
Here, $\mathfrak{R}_{\mQ}:\mQ\rightarrow\mQ'$ is the Riesz operator. 

For the following result, we will assume that $\mB$ is surjective.

\begin{teo}[Main Theorem 2]
	\label{teo-p_lambda}
	Together with Assumption \ref{assumption-1}, assume further that $\Ima\mB=\mQ'$, i.e., there exists $\beta>0$ such that
	$$
	\displaystyle\sup_{v\in \mV}\frac{b(v,q)}{\Vert v\Vert_\mV}\geq \beta \Vert q\Vert_\mQ,\;\;\forall q\in \mQ.
	$$
	Moreover, assume that $c(\cdot,\cdot)$ is given by \eqref{bilinear_form_inner_product_Q} with $\lambda>0$. Then, for every $f\in\mV'$ and for every $g\in\mQ'$, Problem \ref{prob2} has a unique solution. Moreover, there exist positive constants $C_i, i=1,2,3,4$, all uniform with respect to $\lambda$, such that
	\begin{align}
		\nonumber&\Vert u \Vert_{L_\mathcal{J}^1(\mV)}\leq C_1\Vert f\Vert_{L_\mathcal{J}^1(\mV')} + C_2\Vert g\Vert_{L_\mathcal{J}^1(\mQ')},\\
		\nonumber&\Vert p \Vert_{L_\mathcal{J}^1(\mQ)}\leq C_3\Vert f\Vert_{L_\mathcal{J}^1(\mV')} + C_4\Vert g\Vert_{L_\mathcal{J}^1(\mQ')}.
	\end{align}
\end{teo}
\begin{proof}
	We divide the proof in two cases. The first one corresponds to the case when $(u,p)\in\mV\times\mQ$ solves  the problem
	\begin{equation}
		\label{teo-p_lambda-pv2}\left\{
		\begin{aligned}
			a(u,v) + b(v,p) &=\langle f,v\rangle_{\mV} + \int_{0}^{t}\bigg[k_1(t,s)a(u(s),v) +k_2(t,s)b(v,p(s)) \bigg]\,ds,\\
			b(u,q) -\lambda(p,q)_{\mQ} &=0,
		\end{aligned}
		\right.
	\end{equation}
	for all $(v,q)\in\mV\times\mQ$,	or equivalently,
	\begin{equation}
		\label{teo-p_lambda-po2}\left\{
		\begin{aligned}
			\mA u + \mB^* p&=f + \int_{0}^{t}\bigg[k_1(t,s)\mA u(s) + k_2(t,s)\mB^* p(s) \bigg]\,ds,\\
			\mB u-\lambda\mathfrak{R}_\mQ p&=0,
		\end{aligned}
		\right.
	\end{equation}
	a.e. in $\mathcal{J}$. By using the lifting operator $\mL_{\mB}$ we set $\widetilde{u}:=\mL_{\mB}\lambda\mathfrak{R}_\mQ p$, 
	so that 
	$\mB u=\mB\widetilde{u}=\lambda\mathfrak{R}_\mQ p.$
	Then, defining $u_0:=u-\widetilde{u}$, we have that $u_0\in\mK$. Testing with $v=\widetilde{u}(t)$ in the first equation of \eqref{teo-p_lambda-pv2} and setting $p=\mathfrak{R}_\mQ^{-1}\lambda^{-1}\mB u$, we obtain
	$$
	\begin{aligned}
		&a(u,\widetilde{u})+b\big(\widetilde{u},\mathfrak{R}_\mQ^{-1}\lambda^{-1}\mB u\big)=\langle f,\widetilde{u}\rangle_{\mV}\\
		&\hspace{3cm}+ \int_{0}^{t}\bigg[k_1(t,s)a\big(u(s),\widetilde{u}\big) + k_2(t,s)b\big(\widetilde{u},\mathfrak{R}_\mQ^{-1}\lambda^{-1}\mB u(s)\big) \bigg]\,ds.
	\end{aligned}
	$$
	Since $\mB \widetilde{u}=\mB u$, from de above it follows that
	\begin{align}
		\nonumber\lambda^{-1}\Vert &\mB u\Vert_{\mQ'}^2=\langle f,\widetilde{u}\rangle_{\mV}- a(u,\widetilde{u})\\
		&+ \int_{0}^{t}\bigg[k_1(t,s)a\big(u(s),\widetilde{u}\big) + k_2(t,s)\lambda^{-1}\langle \mB u(s),\mathfrak{R}_\mQ^{-1}\mB u\rangle_{\mQ} \bigg]ds.\label{teo-p-eq22}
	\end{align}
	Now we will estimate $-a(u,\widetilde{u})$. With this aim, first we observe that  \eqref{seminorms}, \eqref{continuity-seminorms-bilinear}, and the splitting $u=\widetilde{u}+u_0$ yields to
	\begin{equation}
		\label{teo-p-eq23}
		-a(u,\widetilde{u})=a(\widetilde{u}+u_0,\widetilde{u})=-a(\widetilde{u},\widetilde{u})-a(u_0,\widetilde{u})\leq -\vert \widetilde{u}\vert_a^2+\vert \widetilde{u}\vert_a\vert u_0\vert_a.
	\end{equation}
	On the other hand, testing the first equation in \eqref{teo-p_lambda-pv2} with $v=u_0(t)$ gives
	\begin{equation*}
		\label{teo-p-eq24}
		a(u,u_0)=\langle f,u_0\rangle_{\mV} + \int_{0}^{t}k_1(t,s)a(u_0(s),u_0)ds,
	\end{equation*} 
	and then, from \eqref{seminorms} we have that
	\begin{equation}
		\label{teo-p-eq25}
		\begin{aligned}
			\vert u_0\vert_a^2=a(u_0,u_0)&=a(u,u_0)-a(\widetilde{u},u_0)\\
			&\leq \Vert f\Vert_{\mV'}\Vert u_0\Vert_{\mV} + \vert u_0\vert_a\vert \widetilde{u}\vert_a + C_{k_1}\int_{0}^{t}\vert u(s)\vert_ads\vert u_0\vert_a.
		\end{aligned}
	\end{equation} 
	Notice that the ellipticity in the kernel of $\mB$ implies that $\alpha_0\Vert u_0\Vert_{\mV}^2\leq \vert u_0\vert_a^2$. Then, \eqref{teo-p-eq25} is reduced to
	\begin{equation}
		\label{teo-p-eq26}
		\begin{aligned}
			\vert u_0\vert_a\leq\frac{1}{\alpha_0^{1/2}}\Vert f\Vert_{\mV'} + \vert \widetilde{u}\vert_a + C_{k_1}\int_{0}^{t}\vert u(s)\vert_a ds.
		\end{aligned}
	\end{equation}
	Inserting \eqref{teo-p-eq26} in \eqref{teo-p-eq23} yields to 
	\begin{equation*}
		\label{teo-p-eq27}
		-a(u,\widetilde{u})\leq \frac{1}{\alpha_0^{1/2}}\Vert f\Vert_{\mV'}\vert\widetilde{u}\vert_a + C_{k_1}\vert\widetilde{u}\vert_a\int_{0}^{t}\vert u(s)\vert_a ds.
	\end{equation*} 
	Replacing this inequality in \eqref{teo-p-eq22} we obtain
	\begin{equation}
		\label{teo-p-eq28}
		\begin{aligned}
			\lambda^{-1}\Vert &\mB u\Vert_{\mQ'}^2\leq\Vert f\Vert_{\mV'}\Vert \widetilde{u}\Vert_{\mV} + \frac{\Vert a\Vert^{1/2}}{\alpha_0^{1/2}}\Vert f\Vert_{\mV'}\Vert \widetilde{u}\Vert_{\mV} + C_{k_1}\Vert a\Vert\Vert \widetilde{u}\Vert_{\mV}\int_{0}^{t}\Vert u(s)\Vert_{\mV}ds\\
			&\hspace{0.5cm}+C_{k_1}\Vert a \Vert\int_{0}^{t}\Vert u(s)\Vert_{\mV}ds\Vert \widetilde{u}\Vert_{\mV} +C_{k_2}\Vert \mB u\Vert_{\mQ'}\int_{0}^{t}\lambda^{-1}\Vert \mB u(s)\Vert_{\mQ'}ds\\
			&\leq \bigg(1+\frac{\Vert a\Vert^{1/2}}{\alpha_0^{1/2}}\bigg)\Vert f\Vert_{\mV'}\Vert \widetilde{u}\Vert_{\mV} \\
			&\hspace{0.5cm}+ 2C_{k_1}\Vert a\Vert\int_{0}^{t}\Vert u(s)\Vert_{\mV}ds\Vert \widetilde{u}\Vert_{\mV}+C_{k_2}\int_{0}^{t}\lambda^{-1}\Vert \mB u(s)\Vert_{\mQ'}ds\Vert \mB u\Vert_{\mQ'}.
		\end{aligned}
	\end{equation} 
	From the inf-sup condition of $\mathbb{B}$, we have that $\beta\Vert\widetilde{u}\Vert_{\mV}\leq\Vert\mB \widetilde{u}\Vert_{\mQ'}=\Vert \mB u\Vert_{\mQ'}$, so the inequality \eqref{teo-p-eq28} becomes
	\begin{equation}
		\label{teo-p-eq29}
		\begin{aligned}
			\lambda^{-1}&\Vert \mB u\Vert_{\mQ'}
			\leq\frac{1}{\beta}\bigg(1+\frac{\Vert a\Vert^{1/2}}{\alpha_0^{1/2}}\bigg)\Vert f\Vert_{\mV'} \\
			&\hspace{0.45cm}+ \frac{2C_{k_1}\Vert a\Vert}{\beta}\int_{0}^{t}\bigg[\Vert \widetilde{u}(s)\Vert_{\mV}+\Vert u_0(s)\Vert_{\mV}\bigg]ds+C_{k_2}\int_{0}^{t}\lambda^{-1}\Vert \mB u(s)\Vert_{\mQ'}\,ds\\
			&\leq\frac{1}{\beta}\bigg(1+\frac{\Vert a\Vert^{1/2}}{\alpha_0^{1/2}}\bigg)\Vert f\Vert_{\mV'}\\
			&\hspace{0.45cm}+ \frac{2C_{k_1}\Vert a\Vert}{\beta}\int_{0}^{t}\Vert u_0(s)\Vert_{\mV}\,ds + \bigg(\frac{2\lambda C_{k_1}\Vert a\Vert}{\beta^2}+C_{k_2}\bigg)\int_{0}^{t}\lambda^{-1}\Vert \mB u(s)\Vert_{\mQ'}\,ds.
		\end{aligned}
	\end{equation}
	
	Now we will estimate $\int_{0}^{t}\Vert u_0(s)\Vert_{\mV}\,ds$. We begin by noticing that the continuity and the  $\mK$-ellipticity of $a(\cdot,\cdot)$ can be used in \eqref{teo-p-eq26}, along with \eqref{continuity-seminorms}, and the split $u=\widetilde{u}+u_0$, in order  to obtain
	\begin{equation*}
		\label{teo-p-eq30}
		\begin{aligned}
			\Vert u_0\Vert_{\mV}&\leq \frac{1}{\alpha_0}\Vert f\Vert_{\mV'} + \frac{\Vert a\Vert^{1/2}}{\alpha_0^{1/2}}\Vert \widetilde{u}\Vert_{\mV} \\
			&\hspace{1cm}+ \frac{C_{k_1}\Vert a \Vert^{1/2}}{\alpha_0^{1/2}}\int_{0}^{t}\Vert \widetilde{u}(s)\Vert_{\mV} ds + \frac{C_{k_1}\Vert a \Vert^{1/2}}{\alpha_0^{1/2}}\int_{0}^{t}\Vert u_0(s)\Vert_{\mV}\, ds .
		\end{aligned}
	\end{equation*} 
	From Gronwall's lemma we have that
	\begin{equation}
		\label{teo-p-eq33}
		\Vert u_0\Vert_{\mV}\leq \widetilde{m}(t)+\bigg(1 +\frac{TC_{k_1}\Vert a \Vert^{1/2}}{\alpha_0^{1/2}}e^{\displaystyle\frac{TC_{k_1}\Vert a \Vert^{1/2}}{\alpha_0^{1/2}}} \bigg)\int_{0}^{t}\widetilde{m}(s)\,ds,
	\end{equation}
	where
	\begin{equation*}
		\widetilde{m}(t):=\frac{1}{\alpha_0}\Vert f\Vert_{\mV'} + \frac{\Vert a\Vert^{1/2}}{\alpha_0^{1/2}}\Vert \widetilde{u}\Vert_{\mV} + \frac{C_{k_1}\Vert a \Vert^{1/2}}{\alpha_0^{1/2}}\int_{0}^{t}\Vert \widetilde{u}(s)\Vert_{\mV}\, ds.
	\end{equation*}
	Using the inf-sup condition of $\mB$ and integrating \eqref{teo-p-eq33} over $[0,t]$, we obtain
	\begin{equation}
		\label{teo-p-eq35}
		\int_{0}^{t}\Vert u_0(s)\Vert_{\mV}ds\leq M_1\int_{0}^{t}\Vert f(s)\Vert_{\mV'} + M_2\int_{0}^{t}\Vert \mB u(s)\Vert_{\mQ'}\,ds,
	\end{equation}
	where
	\begin{equation*}
		\label{teo-p-constant-M1}
		M_1:=\frac{1}{\alpha_0}\bigg(1 +\frac{TC_{k_1}\Vert a \Vert^{1/2}}{\alpha_0^{1/2}}e^{\displaystyle\frac{TC_{k_1}\Vert a \Vert^{1/2}}{\alpha_0^{1/2}}} \bigg),
	\end{equation*}
	and
	$$
	\qquad M_2:=\frac{\Vert a\Vert^{1/2}}{\beta\alpha_0^{1/2}}\bigg(1 +\frac{TC_{k_1}\Vert a \Vert^{1/2}}{\alpha_0^{1/2}}e^{\displaystyle\frac{TC_{k_1}\Vert a \Vert^{1/2}}{\alpha_0^{1/2}}} \bigg)\bigg(1+TC_{k_1} \bigg).
	$$
	Replacing \eqref{teo-p-eq35} in \eqref{teo-p-eq29} and rearranging terms, we obtain
	\begin{equation}
		\label{teo-p-eq37}
		\Vert \mB u\Vert_{\mQ'}\leq \lambda n(t)+N\int_{0}^{t}\Vert\mB u(s)\Vert_{\mQ'}ds,
	\end{equation}
	where
	\begin{equation*}
		n(t):=\frac{1}{\beta}\bigg[\bigg(1+\frac{\Vert a\Vert^{1/2}}{\alpha_0^{1/2}}\bigg)\Vert f\Vert_{\mV'}+ 2C_{k_1}M_1\Vert a\Vert\int_{0}^{t}\Vert f(s)\Vert_{\mV'}\,ds\bigg],
	\end{equation*}
	and
	\begin{equation*}
		N:= \frac{2\lambda C_{k_1}M_2\Vert a\Vert}{\beta}+\frac{2\lambda C_{k_1}\Vert a\Vert}{\beta^2}+C_{k_2}.
	\end{equation*}
	Thus, we apply the Gronwall's inequality to \eqref{teo-p-eq37} to obtain
	\begin{equation*}
		\label{teo-p-eq38}
		\Vert \mB u\Vert_{\mQ'}\leq \lambda n(t)+\lambda Ne^{TN}\int_{0}^{t}n(s)\,ds,
	\end{equation*}
	and, integrating the previous expression in $\mathcal{J}$, yields to
	\begin{equation}
		\label{teo-p-eq39}
		\lambda^{-1}\Vert \mB u\Vert_{L_\mathcal{J}^1(\mQ')}\leq (1+TNe^{TN})\int_{0}^{t}n(s)\,ds.
	\end{equation}

	Then, the fact that $\beta\Vert \widetilde{u}\Vert_{\mV}\leq\Vert \mB \widetilde{u}\Vert_{\mQ'}=\Vert \mB u\Vert_{\mQ'}$, allows us to obtain the following estimate for $\widetilde{u}$
	\begin{equation*}
		\Vert \widetilde{u}\Vert_{L_\mathcal{J}^1(\mV)}\leq C_{u_1}\Vert f\Vert_{L_\mathcal{J}^1(\mV')},
	\end{equation*}
	where
	$$
	C_{u_1}:=\frac{\lambda}{\beta^2}\bigg[\bigg(1+\frac{\Vert a\Vert^{1/2}}{\alpha_0^{1/2}}\bigg) + 2TC_kM_1\Vert a \Vert\bigg]\bigg(1+TNe^{TN}\bigg).
	$$
	
	Now we proceed to estimate $u_0$. To accomplish this task,  we set $t=T$ in \eqref{teo-p-eq33} and, using the previous bound for $\widetilde{u}$, we obtain the following estimate
	$$
	\begin{aligned}
		\Vert u_0\Vert_{L_\mathcal{J}^1(\mV)}\leq M_1\Vert f\Vert_{L_\mathcal{J}^1(\mV')} + \beta M_2\Vert \widetilde{u}\Vert_{L_\mathcal{J}^1(\mV)}\leq C_{u_2} \Vert f\Vert_{L_\mathcal{J}^1(\mV')},
	\end{aligned}
	$$
	where $	C_{u_2}:=M_1+ \beta M_2 C_{u_1}.$
	
	Recalling that $u_{0}=u-\widetilde{u}$, the triangle inequality and the previous bound yields to
	\begin{equation*}
		\label{teo-p-estimate-u2}
		\Vert u\Vert_{L_\mathcal{J}^1(\mV)}\leq\Vert u_0\Vert_{L_\mathcal{J}^1(\mV)}+\Vert \widetilde{u}\Vert_{L_\mathcal{J}^1(\mV)}\leq C_1\Vert f\Vert_{L_\mathcal{J}^1(\mV')},
	\end{equation*}
	where $C_1:=C_{u_1}+C_{u_2}$. From the second equation in \eqref{teo-p_lambda-po2}, together with \eqref{teo-p-eq39}, we derive the following estimate for $p$,
	\begin{equation*}
		\label{teo-p-estimate-p2}
		\Vert p\Vert_{L_\mathcal{J}^1(\mQ)}\leq\Vert\lambda^{-1}\mB u\Vert_{L_\mathcal{J}^1(\mQ')}\leq C_3\Vert f\Vert_{L_\mathcal{J}^1(\mV')},
	\end{equation*}
	where
	\begin{equation*}
		\label{teo-p-constant-c3}
		C_3:= \frac{1}{\beta}\bigg[\bigg(1+\frac{\Vert a\Vert^{1/2}}{\alpha_0^{1/2}}\bigg) + 2TC_kM_1\Vert a \Vert\bigg]\bigg(1+TNe^{TN}\bigg).
	\end{equation*}
	
	Now, we consider the second case and assume that $u,p$ and $g$ satisfy
	\begin{equation*}
		\label{teo-p_lambda-pv}
		\left\{
		\begin{aligned}
			&a(u,v)+b(v,p)=\int_{0}^t\bigg[ k_1(t,s)a(u(s),v)+k_2b(v,p(s))\bigg]\,ds,\\
			&b(u,q) -\lambda(p,q)_\mQ=\langle g,q\rangle_\mQ,
		\end{aligned}
		\right.
	\end{equation*}
	for all $(v,q)\in\mV\times\mQ$, which in operator form reads a.e in $\mathcal{J}$ as follows
	\begin{equation}
		\label{teo-p_lambda-po}
		\left\{
		\begin{aligned}
			&\mA u+\mB^*p=\int_{0}^t\bigg[k_1(t,s)\mA u(s)+k_2(t,s)\mB^*p(s)\bigg]\,ds,\\
			&\mB u -\lambda \mathfrak{R}_\mQ p=g.
		\end{aligned}
		\right.
	\end{equation}
	
	Analogous to the proof of Theorem \ref{teo8}, we take the $\mV'$-norm in the first equation of \eqref{teo-p_lambda-po} and apply Gronwall's lemma to obtain the equalities \eqref{teo_8-001} and \eqref{teo_8-002}. From this, we arrive to an usual mixed formulation. Hence, we resort to \cite[Theorem 4.3.2]{boffi2013mixed} in order to obtain the remaining bounds:
	$$
	\begin{aligned}
		\Vert u\Vert_{L_\mathcal{J}^1(\mV)}\leq C_2\Vert g\Vert_{L_\mathcal{J}^1(\mQ')} \quad\text{and}\quad\Vert p\Vert_{L_\mathcal{J}^1(\mQ)}\leq C_4\Vert g\Vert_{L_\mathcal{J}^1(\mQ')},
	\end{aligned}
	$$
	where
	$$
	C_2:=\frac{\beta^2+4\lambda\Vert a\Vert}{\alpha_0\beta^2} \quad\text{and}\quad C_4:= \frac{4\Vert a\Vert}{\Vert a\Vert\lambda + 2\beta^2}.
	$$
	This concludes the proof. 
\end{proof}

For completeness, below we provide a result that considers the operator $\mB$ as closed, but not surjective. The proof is a straightforward application of Theorem \ref{teo-p_lambda}.
\begin{cor}
	Together with Assumption \ref{assumption-1}, assume that $\Ima\mB$ is closed and $c(\cdot,\cdot)$ is given by \eqref{bilinear_form_inner_product_Q} with $\lambda>0$. Set $g=\overline{g}+g_0$ and set $p=\overline{p}+p_0$. Then, for every $f\in L_{\mathcal{J}}^1(\mV')$ and every $g\in L_{\mathcal{J}}^1(\mQ')$, Problem \ref{prob2-weak} has a unique solution. Moreover, there exist $C_i>0,\, i=1,2,3,4$, uniform in $\lambda$, such that
	
	\begin{align}
		\nonumber&\Vert u \Vert_{L_\mathcal{J}^1(\mV)}+\Vert \overline{p} \Vert_{L_\mathcal{J}^1(\mQ)}\leq (C_1+C_3)\Vert f\Vert_{L_\mathcal{J}^1(\mV')} + (C_2+C_4)\Vert g\Vert_{L_\mathcal{J}^1(\mQ')},
\end{align}
	and
	$$
	\Vert p_0\Vert_{L_\mathcal{J}^1(\mQ)}\leq\frac{1}{\lambda}\Vert g_0\Vert_{L_\mathcal{J}^1(\mQ')}.
	$$
\end{cor}

The importance of the previous result lies in the fact that all the constants involved in the continuous dependency are uniform with respect to the parameter $\lambda$. We remark that this fact is of great importance in the study of slender structures, like Timoshenko beams, Reissner-Mindlin plates, among others, since the thickness parameter is the one responsible of the so called \emph{locking phenomenon} in the design of numerical methods. For this reason, if $\lambda$ represents the thickness of some particular structure, Theorem \ref{teo-p_lambda} states that all the constants will be uniform with respect to it.


\subsection{Semi-discrete abstract analysis}
\label{subsection2-3}
The goal of the present section is to analyze the semi-discrete counterpart of the proposed mixed problems and obtain a priori error estimates. Here, we consider the necessary hypotheses for the existence and uniqueness of semi-discrete solutions such as ellipticity in the kernel and the discrete inf-sup condition (see for instance \cite{aminikhah2010new,saedpanah2016existence} for further details related to the existence of semi-discrete solutions of Volterra equations of the second kind), so this section will be focused on deriving error estimates coming from continuous abstract models, which are characterized by having constants that do not deteriorate as $\lambda$ becomes small.

To begin, we introduce the following assumption.

\begin{assunp}
	\label{assumption-discrete-regular}
	Assume that there exist two finite dimensional spaces $\mV_h$ and $\mQ_h$ such that $\mV_h\subset\mV$ and $\mQ_h\subset \mQ$. Together with the continuous space kernels $\mK$ and $\mH$, we consider the discrete counterparts
	\begin{equation*}
		\mK_h:=\bigg\{ v_h\in \mV_h\,\,:\,\, b(v_h,q_h)=0,\,\,\,\forall q_h\in\mQ_h\bigg\},
	\end{equation*}
	\begin{equation*}
		\mH_h:=\bigg\{q_h\in\mQ_h\,\,:\,\, b(v_h,q_h)=0,\,\,\,\forall v_h\in\mV_h\bigg\},
	\end{equation*}
	such that there exist positive constants $\alpha_*^0,\gamma_*^0>0$, independent of $h$, such that
	\begin{equation*}
		\label{discrete-ellipticity1}
		\begin{aligned}
			&a(v_h^0,v_h^0)\geq\alpha_*^0\|v_h^0\|_{\mV}^2\,\,\,\,\forall v_h^0\in\mK_h,\\
			&c(q_h^0,q_h^0)\geq\gamma_*^0\|q_h^0\|_{\mQ}^2\,\,\,\,\forall q_h^0\in\mQ_h.
		\end{aligned}
	\end{equation*}
	This coercivity can be also extended to the whole space $\mV_h$ and $\mQ_h$.
\end{assunp}

To simplify the analysis, the semi-discrete test functions will be written as $v$ instead of $v_h$ and $q$ instead of $q_h$, as long as the complete writing is not required.

We define the corresponding errors as follows
\begin{equation*}\label{errors}
	\eu:=u_h-u=\xi_u-\eta_u, \,\,\,\,\,\,\,\ep:=p_h-p=\xi_p-\eta_p,
\end{equation*}
where $\xi_u:=u_h-u_I,\,\xi_p:=p_h-p_I,\, \eta_u=u-u_I,$ and $\eta_p=p-p_I$. Here, $u_{I}\in \mV_h$ and $p_I\in \mQ_h$ represent general interpolations of $u$ and $p$, respectively (see for example \cite[Chapter 5]{boffi2013mixed} or \cite[Chapter 4.]{miao1998finite}).

We shall now study the semi-discretization of the mixed formulations considered in the previous section. The following problem corresponds to the semi-discretized version of Problem \ref{goal-prob-regular}.
\begin{problem}
	\label{goal-prob-regular-discreto}
	Find $u_h\in L_\mathcal{J}^1(\mV_h)$ and $p_h\in L_\mathcal{J}^1(\mQ_h)$ such that
	\begin{equation*}
		\label{pv-goal-prob-regular-discreto-1}
		\left\{\begin{aligned}
			&a(u_h,v) + b(v,p_h) = \langle f,v\rangle_{\mV} + \int_{0}^{t} \bigg[k_1(t,s)a(u_h(s),v) + k_2(t,s)b(v,p_h(s))\bigg]\,ds\\ 
			&b(u_h,q) - c(p_h,q)=\langle g,q\rangle_{\mQ} + \int_{0}^{t} \bigg[k_3(t,s)b\big(u_h(s),v\big) - k_4(t,s)c\big(v,p_h(s)\big)\bigg]\,ds, 
		\end{aligned}\right.
	\end{equation*}
	for all $(v,q)\in \mV_h\times\mQ_h$.
\end{problem}

In agreement with previous notations, each $v\in\mV_h$ and each $q\in\mQ_h$ might be splitted as
\begin{equation}
	\label{regular-sd-v-q}
	v=v^0+\overline{v},\hspace{1cm}q=q^0+\overline{q},
\end{equation}
with $v^0\in\mK_h,\, \overline{v}\in\mK_h^{\perp},\,q^0\in\mH_h$ and $\overline{q}\in\mH_h^{\perp}$. Similarly, the given data will be decomposed as follows
\begin{equation}
	\label{regular-sd-f-g}
	f=f_h^0+\overline{f}_h,\hspace{1cm}g_h=g_h^0+\overline{g}_h,
\end{equation}
with $f_h^0\in L_\mathcal{J}^1(\mK'),\,\overline{f}_h\in L_\mathcal{J}^1((\mK_h^{\perp})'),\,g_h^0\in L_\mathcal{J}^1(\mH'),$ and $\overline{g}_h\in L_\mathcal{J}^1((\mH_h^{\perp})')$. Note that the splitting \eqref{regular-sd-f-g} might be different from the splitting made in the continuous case since, in general, $\mK_h\subsetneq\mK$ and $\mH_h\subsetneq\mH$. Moreover,  the spaces $\mK_h^{\perp}$ and $\mH_h^{\perp}$ should always be understood as subspaces of $\mV_h$ and $\mQ_h$, respectively.

Subtracting Problem \ref{goal-prob-regular} and Problem \ref{goal-prob-regular-discreto} we obtain
\begin{equation*}
	\label{semi-discrete-9}
	\left\{\begin{aligned}
		a(\eu,v) + b(v,\ep) &= \int_{0}^{t}\bigg[k_1(t,s)a(\eu(s),v) + k_2(t,s)b(v,\ep(s))\bigg]\,ds,\\
		b(\eu,q) - c(\ep,q)&= \int_{0}^{t}\bigg[k_3(t,s)b(\eu(s),v) - k_4(t,s)c(v,\ep(s))\bigg]\,ds,
	\end{aligned}\right.
\end{equation*}
for all $(v,q)\in\mV_h\times\mQ_h$, which after adding and subtracting $u_I$ and $p_I$, allows to infer that $(\xi_u,\xi_p)\in L_{\mathcal{J}}^1(\mV_h\times\mQ_h\big)$  is the solution of the variational problem: Find $(\xi_u,\xi_p)\in L_{\mathcal{J}}^1(\mV_h\times\mQ_h\big)$ such that
\begin{equation*}
	\label{semi-discrete-10}
	\left\{\begin{aligned}
		a(\xi_u,v) + b(v,\xi_p) &=\langle\mF,v\rangle_{\mV}+ \int_{0}^{t} \bigg[k_1(t,s)a(\xi_u(s),v) + k_2(t,s)b(v,\xi_p(s))\bigg]\,ds, \\ 
		b(\xi_u,q) - c(\xi_p,q)&=\langle\mG,q\rangle_{\mQ}+ \int_{0}^{t}\bigg[k_3(t,s)b(\xi_u(s),v) - k_4(t,s)c(v,\xi_p(s))\bigg]\,ds,\\ 
	\end{aligned}\right.
\end{equation*}
for all $(v,q)\in \mV_h\times\mQ_h$, where
\begin{align*}
	\langle \mF,v\rangle_{\mV}&=a(\eta_u,v)+b(v,\eta_p)-\int_{0}^{t}\bigg[k_1(t,s)a(\eta_u(s),v)+k_2(t,s)b(v,\eta_p(s)) \bigg]\,ds\\
	\langle \mG,q\rangle_{\mQ} &= b(\eta_u,q) -c(\eta_p,q)-\int_{0}^{t}\bigg[k_3(t,s)b(\eta_u(s),q)-k_4(t,s)c(\eta_p(s),q) \bigg]\,ds.
\end{align*}

By using the boundedness of the kernels, the continuity of the bilinear forms, the definition of $\eta_u$ and $\eta_p$ and \eqref{regular-sd-v-q}-\eqref{regular-sd-f-g}, it follows that
\begin{align*}
	\Vert \overline{\mF}\Vert_{L_\mathcal{J}^1(\mV')}+\Vert \mF_0\Vert_{L_\mathcal{J}^1(\mV')}&\leq (1+TC_k)\bigg[\Vert a\Vert\,\Vert u - u_I\Vert_{L_\mathcal{J}^1(\mV)} + \Vert b\Vert\,\Vert p-p_I\Vert_{L_\mathcal{J}^1(\mQ)} \bigg],\\
		\Vert\overline{\mG}\Vert_{L_\mathcal{J}^1(\mV')}+\Vert\mG_0\Vert_{L_\mathcal{J}^1(\mV')}&\leq (1+TC_{\widetilde{k}})\bigg[\Vert b\Vert\,\Vert u - u_I\Vert_{L_\mathcal{J}^1(\mV)} + \Vert c\Vert\,\Vert p-p_I\Vert_{L_\mathcal{J}^1(\mQ)} \bigg],
	\end{align*}
where $C_k:=\max\{C_{k_1},C_{k_2}\}$ and $C_{\widetilde{k}}:=\max\{C_{k_3},C_{k_4}\}$. This leads to the estimate
\begin{equation*}
	\begin{aligned}
		\Vert&\overline{\mF}\Vert_{L_\mathcal{J}^1(\mV')}+ \Vert \mF_0\Vert_{L_\mathcal{J}^1(\mV')}+\Vert\overline{\mG}\Vert_{L_\mathcal{J}^1(\mV')}+\Vert\mG_0\Vert_{L_\mathcal{J}^1(\mV')}\\
		&\leq \bigg(1+T\max\{C_k,C_{\widetilde{k}}\}\bigg)\bigg[\bigg(2\Vert a\Vert + 4\Vert b\Vert + 2\Vert c\Vert\bigg)\bigg(\Vert \eta_u\Vert_{L_\mathcal{J}^1(\mV)}+\Vert \eta_p\Vert_{L_\mathcal{J}^1(\mQ)}\bigg)\bigg].
	\end{aligned}
\end{equation*}

Since {$\xi_u=\overline{\xi}_u+\xi_u^0$ and $\xi_p=\overline{\xi}_p+\xi_p^0$, we are in position to apply Theorem \ref{teo8} in order to obtain an error estimate for the generalized semi-discrete variational problem considered.  
	\begin{teo}
		\label{teo-sd-perturbed-estimate1}
		Under Assumption \ref{assumption-discrete-regular}, suppose that there exists $\beta_*>0$ such that 
		$$
		\displaystyle\sup_{v\in \mV_h}\frac{b(v,q)}{\Vert v\Vert_\mV}\geq \beta_* \Vert q\Vert_\mQ,\;\;\forall q\in \mH_h^{\perp} \quad\text{and}\quad \sup_{q\in \mQ_h}\frac{b(v,q)}{\Vert q\Vert_\mQ}\geq \beta_* \Vert v\Vert_\mV,\;\;\forall v\in \mK_h^{\perp}.
		$$
		Then, for every $f\in L_\mathcal{J}^1(\mV')$ and $g\in L_\mathcal{J}^1(\mQ')$, we have that Problem \ref{goal-prob-regular-discreto} has a unique solution. Moreover, if $(u,p)$ is a solution of Problem \ref{goal-prob-regular}, then for every $u_I\in L_\mathcal{J}^1(\mV_h)$ and for every $p_I\in L_\mathcal{J}^1(\mQ_h)$, we have the estimates
		\begin{equation*}
			\label{teo8-estimates-bounds-overline_u}
			\Vert \overline{\xi}_u\Vert_{L_\mathcal{J}^1(\mV)}\leq C_9\Vert\overline{\mF}\Vert_{L_\mathcal{J}^1(\mV')} + C_{13}\Vert \mF_0\Vert_{L_\mathcal{J}^1(\mV')} + C_3\Vert \overline{\mG}\Vert_{L_\mathcal{J}^1(\mQ')} + C_7\Vert \mG_0\Vert_{L_\mathcal{J}^1(\mQ')}
		\end{equation*}
		\begin{equation*}
			\label{teo8-estimates-bounds-u_0}
			\Vert \xi_u^0\Vert_{L_\mathcal{J}^1(\mV)}\leq C_{10}\Vert\overline{\mF}\Vert_{L_\mathcal{J}^1(\mV')} + C_{14}\Vert \mF_0\Vert_{L_\mathcal{J}^1(\mV')} + C_4\Vert \overline{\mG}\Vert_{L_\mathcal{J}^1(\mQ')} + C_8\Vert \mG_0\Vert_{L_\mathcal{J}^1(\mQ')}
		\end{equation*}
		\begin{equation*}
			\label{teo8-estimates-bounds-overline_p}
			\Vert \overline{\xi}_p\Vert_{L_\mathcal{J}^1(\mQ)}\leq C_{11}\Vert\overline{\mF}\Vert_{L_\mathcal{J}^1(\mV')} + C_{15}\Vert \mF_0\Vert_{L_\mathcal{J}^1(\mV')} + C_1\Vert \overline{\mG}\Vert_{L_\mathcal{J}^1(\mQ')} + C_5\Vert \mG_0\Vert_{L_\mathcal{J}^1(\mQ')}
		\end{equation*}
		\begin{equation*}
			\label{teo8-estimates-bounds-p_0}
			\Vert \xi_p^0\Vert_{L_\mathcal{J}^1(\mQ)}\leq C_{12}\Vert\overline{\mF}\Vert_{L_\mathcal{J}^1(\mV')} + C_{16}\Vert \mF_0\Vert_{L_\mathcal{J}^1(V')} + C_2\Vert \overline{\mG}\Vert_{L_\mathcal{J}^1(\mQ')} + C_6\Vert \mG_0\Vert_{L_\mathcal{J}^1(\mQ')},
		\end{equation*}
		where $C_i$, with $i=1,\ldots, 16$, are positive constants depending on the semi-discrete stability constants $\alpha_*^0,\beta_*,\gamma_*^0$, the continuity constants $\Vert a \Vert, \Vert c\Vert$, the constants $C_{k_i}, i=1,2,3,4,$ and the time of observation $T$. Moreover, we have that
		\begin{multline*}
			\Vert u_h-u\Vert_{L_\mathcal{J}^1(\mV)} + \Vert p_h-p\Vert_{L_\mathcal{J}^1(\mQ)}\leq C\bigg(\inf_{v\in \mV_h}\Vert u-v\Vert_{L_\mathcal{J}^1(\mV)}+\inf_{q\in \mQ_h}\Vert p-q\Vert_{L_\mathcal{J}^1(\mQ)}\bigg),
		\end{multline*}
		where $C>0$ is a constant depending on $C_i$.
		
	\end{teo}
	\begin{proof}
		The arguments provided above allow us to obtain the estimates for $\overline{\xi}_u$, $\xi_u^0$, $\overline{\xi}_p$ and $\xi_p^0$ by a direct application of Theorem \ref{teo8}. The constants obtained correspond to the semi-discrete counterparts of the constants $C_i$ in the continuous case. On the other hand, using the triangle inequality we obtain the error estimate for $u$ and $p$, where the constant $C$ will depend on $C_i$ and $C_{k_i}$.
	\end{proof}
	
	Returning to the parameter-dependent problem, in Problem \ref{goal-prob-regular-discreto} we assume that $c(\cdot,\cdot)$ is of the form \eqref{bilinear_form_inner_product_Q} and take $k_3\equiv k_4\equiv 0$ in order to obtain the semi-discrete form of Problem \ref{prob2-weak}.
	
	\begin{problem}
		\label{prob-lambda}
		Find $(u_h,p_h)\in L_\mathcal{J}^1(\mV_h\times\mQ_h)$ such that 
		\begin{equation*}
			\label{teo-p_lambda-discreto}\left\{
			\begin{aligned}
				a(u_h,v) + b(v,p_h) &=\langle f,v\rangle_{\mV} + \int_{0}^{t}\bigg[k_1(t,s)a\big(u_h(s),v\big) +k_2(t,s) b\big(v,p_h(s)\big) \bigg]\,ds,\\ 
				b(u_h,q) -\lambda(p_h,q)_{\mQ} &=\langle g,q\rangle_{\mQ} 
			\end{aligned}
			\right.
		\end{equation*}
		for all $(v,q)\in \mV_h\times\mQ_h$.
	\end{problem}
	Hence, we have the corresponding system, obtained from subtracting Problem \ref{prob2-weak} and Problem \ref{prob-lambda}:
	\begin{equation}
		\label{semi-discrete-perturbed-1}
		\left\{
		\begin{aligned}
			a(\eu,v) + b(v,\ep) &= \int_{0}^{t}\bigg[k_1(t,s)a\big(\eu(s),v\big) +k_2(t,s)b\big(v,\ep(s)\big) \bigg]\,ds,\\
			b(\eu,q) -\lambda(\ep,q)_{\mQ} &=0,
		\end{aligned}
		\right.
	\end{equation}
	for all $(v,q)\in \mV_h\times\mQ_h$. Then, from the linearity of $a(\cdot,\cdot)$ and  $b(\cdot,\cdot)$ we rewrite the problem above as follows: Find  $(\xi_u,\xi_p)\in L_{\mathcal{J}}^1(\mV_h\times\mQ_h\big)$  such that
	\begin{equation}
		\label{semi-discrete-perturbed-2}
		\left\{
		\begin{aligned}
			a(\xi_u,v) + b(v,\xi_p) &=\langle \mF,v\rangle_{\mV}+ \int_{0}^{t}\bigg[k_1(t,s)a\big(\xi_u(s),v\big) + k_2(t,s)b\big(v,\xi_p(s)\big) \bigg]\,ds,\\
			b(\xi_u,q) -\lambda(\xi_p,q)_{\mQ} &=\langle \mG,q\rangle_{\mQ} ,
		\end{aligned}
		\right.
	\end{equation}
	for all $(v,q)\in \mV_h\times\mQ_h$, where
	\begin{align}
		\langle \mF,v\rangle_{\mV}&=a(\eta_u,v)+b(v,\eta_p)-\int_{0}^{t}\bigg[k_1(t,s)a\big(\eta_u(s),v\big)+k_2(t,s)b\big(v,\eta_p(s)\big) \bigg]\,ds,\label{F_sd-estimate-3}\\
		\langle \mG,q\rangle_{\mQ} &= b(\eta_u,q) -\lambda(\eta_p,q)_\mQ \label{G_sd-estimate-3}.
	\end{align}

	Then, from Theorem \ref{teo-p_lambda} we have the following result.

	\begin{teo}
		\label{teo-sd-perturbed-main}
		Together with Assumption \ref{assumption-discrete-regular}, assume that $(u,p)$ is a solution of Problem \ref{prob2-weak} and let $(u_h,p_h)$ be the unique solution of Problem \ref{prob-lambda}. Then, we have the error estimates
		\begin{equation*}
			\begin{aligned}
				&\Vert \eu\Vert_{L_\mathcal{J}^1(\mV)}\leq C_{1u} \inf_{v\in \mV_h}\Vert u-v\Vert_{L_\mathcal{J}^1(\mV)}+ C_{1p}\inf_{q\in \mQ_h}\Vert p-q\Vert_{L_\mathcal{J}^1(\mQ)},\\
				&\Vert\ep\Vert_{L_\mathcal{J}^1(\mV)}\leq C_{2u}\inf_{v\in \mV_h}\Vert u-v\Vert_{L_\mathcal{J}^1(\mV)} +C_{2p}\inf_{q\in \mQ_h}\Vert p-q\Vert_{L_\mathcal{J}^1(\mQ)}.
			\end{aligned}
		\end{equation*}
		Moreover, we have that there exists a constant $C>0$ depending on $C_{iu}$ and $C_{ip}, i=1,2$, uniform with respect to $\lambda$, such that
		$$
		\Vert u_h-u\Vert_{L_\mathcal{J}^1(\mV)} + \Vert p_h-p\Vert_{L_\mathcal{J}^1(\mQ)}\leq C\left(\inf_{v\in \mV_h}\Vert u-v\Vert_{L_\mathcal{J}^1(\mV)}+\inf_{q\in \mQ_h}\Vert p-q\Vert_{L_\mathcal{J}^1(\mQ)}\right).
		$$
	\end{teo}
	\begin{proof}
		Let $u_I\in L_\mathcal{J}^1(\mV_h)$ and  $p_I\in L_\mathcal{J}^1(\mQ_h)$. Then, applying Theorem \ref{teo-p_lambda} in \eqref{semi-discrete-perturbed-2}, we have
		\begin{align*}
			\label{teo-sd-perturbed-estimate2-1}
			\Vert \xi_u\Vert_{L_\mathcal{J}^1(\mV)}&\leq C_1\Vert \mF\Vert_{L_\mathcal{J}^1(\mV')} + C_2\Vert \mG\Vert_{L_\mathcal{J}^1(\mQ')},\nonumber\\
			\Vert \xi_p \Vert_{L_\mathcal{J}^1(\mQ)}&\leq C_3\Vert \mF\Vert_{L_\mathcal{J}^1(\mV')} + C_4\Vert \mG\Vert_{L_\mathcal{J}^1(\mQ')}\nonumber,
		\end{align*}
		where $C_i, i=1,\ldots,4$ are the space finite dimensional counterparts of constants, uniform in $\lambda$, given by Theorem \ref{teo-p_lambda} depending on $\alpha_*^0$, $\beta_*$, the stability constant $\Vert a \Vert$ and period of observation $T$.
		
		Now, estimating \eqref{F_sd-estimate-3} and \eqref{G_sd-estimate-3} gives
		\begin{equation*}
			\begin{aligned}
				\Vert \mF\Vert_{L_\mathcal{J}^1(\mV')}&\leq \bigg(1+TC_k\bigg)\bigg[\Vert a\Vert\,\Vert u - u_I\Vert_{L_\mathcal{J}^1(\mV)} + \Vert b\Vert\,\Vert p-p_I\Vert_{L_\mathcal{J}^1(\mQ)} \bigg],\\
				\Vert \mG\Vert_{L_\mathcal{J}^1(\mQ')}&\leq \Vert b\Vert\,\Vert u-u_I\Vert_{L_\mathcal{J}^1(\mV)}+ \lambda\Vert p-p_I\Vert_{L_\mathcal{J}^1(\mQ)},
			\end{aligned}
		\end{equation*}
		where $C_{k}:=\max\{C_{k_1},C_{k_2} \}$.	From the triangle inequality we obtain
		\begin{equation*}
			\label{teo-sd-basic-error4-1}
			\begin{aligned}
				\Vert \eu\Vert_{L_\mathcal{J}^1(\mV)}&\leq \Vert \xi_{u} - \eta_u\Vert_{L_\mathcal{J}^1(\mV)}\leq \Vert \xi_{u}\Vert_{L_\mathcal{J}^1(\mV)}+\Vert \eta_u\Vert_{L_\mathcal{J}^1(\mV)}\\
				&\leq C_{1u} \Vert \eta_u\Vert_{L_\mathcal{J}^1(\mV)}  +C_{1p}\Vert \eta_p\Vert_{L_\mathcal{J}^1(\mQ)},
			\end{aligned}
		\end{equation*}
		where
		\begin{equation*}
			C_{1u}:=C_1\Vert a\Vert\ \bigg(1+TC_k\bigg)+ C_2\Vert b\Vert+1,
			\quad\text{
				and}
			\quad
			C_{1p}:=C_1\Vert b\Vert\bigg(1+TC_k\bigg)+C_2\lambda.
		\end{equation*}
		Similarly, we have
		\begin{equation*}
			\label{teo-sd-basic-error4-2}
			\begin{aligned}
				\Vert \ep\Vert_{L_\mathcal{J}^1(\mQ)}&\leq \Vert \xi_p-\eta_p\Vert_{L_\mathcal{J}^1(\mQ)}\leq \Vert \xi_p\Vert_{L_\mathcal{J}^1(\mQ)}+ \Vert \eta_p\Vert_{L_\mathcal{J}^1(\mQ)}\\
				&\leq C_{2u} \Vert \eta_u\Vert_{L_\mathcal{J}^1(\mV)}  +C_{2p}\Vert \eta_p\Vert_{L_\mathcal{J}^1(\mQ)},
			\end{aligned}
		\end{equation*}
		where
		\begin{equation*}
			C_{2u}:=C_3\Vert a\Vert\ \bigg(1+TC_k\bigg)+ C_4\Vert b\Vert
			\quad\text{
				and}
			\quad
			C_{2p}:=C_3\Vert b\Vert\bigg(1+TC_k\bigg)+C_4\lambda+1.
		\end{equation*}
		We conclude the proof by taking the infimum over all $u_I$ and $p_I$.
	\end{proof}
	
	\subsection{Error estimates in a weaker norm}
	\label{subsection2-4}
	
	Now we include some additional estimates using a \textit{duality} argument in the sense of the Volterra theory (See \cite{shaw2001optimal} for instance). Let us consider two spaces $\mV_{-}$ and $\mQ_{-}$, which we assume to be less regular than $\mV$ and $\mQ$, respectively, satisfying the folowing dense inclusions
	\begin{equation}
		\label{dual-inclusion1}
		\mV \xhookrightarrow[]{}\mV_{-}\quad\text{and}\quad \mQ\xhookrightarrow[]{}\mQ_{-}.
	\end{equation}
	Our aim is to estimate $\Vert u-u_h\Vert_{L_\mathcal{J}^1(\mV_{-})}$ and $\Vert p - p_h\Vert_{L_\mathcal{J}^1(\mQ_{-})}$. To accomplish this task, we define
	\begin{equation*}
		\mV_{+}':=(\mV_{-})',\quad \mQ_{+}':=(\mQ_{-})',
	\end{equation*}
	where the subindex $``+"$ suggests that we have a more regular space. On the other hand, the inclusions provided in \eqref{dual-inclusion1} allow us to obtain
	\begin{equation*}
		\label{dual-inclusion2}
		\mV_{+}'\xhookrightarrow[]{}\mV',\quad \mQ_{+}'\xhookrightarrow[]{}\mQ'.
	\end{equation*} 
	For each $w\in \mV_{++}$, $w_h\in \mV_h$, $m\in\mQ_{++}$ and $m_h\in\mQ_h$, we define the errors $\ew:=w-w_h$ and $\emm:=m-m_h$, , where $\mV_{++}$ and $\mQ_{++}$ will be understood as more regular spaces than $\mV$ and $\mQ$, respectively, satisfying the inclusions
	\begin{equation*}
		\label{dual-inclusion3}
		\mV_{++}\xhookrightarrow[]{}\mV,\quad\mQ_{++}\xhookrightarrow[]{}\mQ.
	\end{equation*}
	Now we introduce the following hypothesis, related with a dual-backward mixed formulation of Problem \ref{goal-mixed-problem1}. 
	\begin{hypo}
		\label{dual-hypo1-goal-prob1}
		For any $\tau\in\mathcal{J}$ and for any $(f_+,g_+)\in L_{[0,\tau]}^{\infty}(\mV_{+}'\times\mQ_{+}')$, we assume that the solution $(w,m)$ to 
		\begin{equation}
			\label{dual-problem3}
			\left\{\begin{aligned}
				a(v,w)+b(v,m)&=\langle f_+,v\rangle_{\mV_{+}'\times \mV}+\int_{t}^\tau\bigg[  k_1(s,t)a(v,w(s))+ k_2(s,t)b(v,m(s))\bigg]ds,\\
				b(w,q) -\lambda (q,m)_{\mQ}&=\langle g_+,q\rangle_{\mQ_{+}'\times \mQ},
			\end{aligned}\right.
		\end{equation}
		for all $(v,q)\in\mV\times\mQ$, belongs to $L_{[0,\tau]}^{\infty}(\mV_{++}\times\mQ_{++})$ a.e. in $[0,\tau]$. Moreover, there exists a constant $\widehat{C}$, independent of $f_+$ and $g_+$, such that
		\begin{equation*}
			\Vert w\Vert_{L_{[0,\tau]}^{\infty}(\mV_{++})}+\Vert m\Vert_{L_{[0,\tau]}^{\infty}(\mQ_{++})}\leq \widehat{C}\bigg(\Vert f_+\Vert_{L_{[0,\tau]}^{\infty}(\mV_{+}')}+\Vert g_+\Vert_{L_{[0,\tau]}^{\infty}(\mQ_{+}')}\bigg).
		\end{equation*}
	\end{hypo}
	With this hypothesis at hand, we prove the following result.
	\begin{teo}
		\label{dual-teo3-goal-prob}
		Under the hypotheses of Theorem \ref{teo-sd-perturbed-main}, assume that Hypothesis \ref{dual-hypo1-goal-prob1} holds. Then, there exists a constant $C$, independent of $h$ and $\lambda$, such that
		\begin{equation*}
			\begin{aligned}
				&\Vert u - u_h\Vert_{L_\mathcal{J}^1(\mV_{-})} + \Vert p - p_h\Vert_{L_\mathcal{J}^1(\mQ_{-})} \\
				&\hspace{1cm}\leq C\left(\inf_{v\in \mV_h}\Vert u - v\Vert_{L_\mathcal{J}^1(\mV)} + \inf_{q\in \mQ_h}\Vert p - q\Vert_{L_\mathcal{J}^1(\mQ)}\right)(r(h)+ n(h)),
			\end{aligned}
		\end{equation*}
		where
		$$
		\begin{aligned}
			&r(h):=\sup_{w\in L_{[0,\tau]}^{\infty}(\mV_{++})}\inf_{w_h\in L_{[0,\tau]}^{\infty}(\mV_h)}\frac{\Vert w- w_h\Vert_{L_{[0,\tau]}^{\infty}(\mV)}}{\Vert w \Vert_{L_{[0,\tau]}^{\infty}(\mV_{++})}},\\
			&n(h):=\sup_{m\in L_{[0,\tau]}^{\infty}(\mQ_{++})}\inf_{m_h\in L_{[0,\tau]}^{\infty}(\mQ_h)}\frac{\Vert m- m_h\Vert_{L_{[0,\tau]}^{\infty}(\mQ)}}{\Vert m \Vert_{L_{[0,\tau]}^{\infty}(\mQ_{++})}}.
		\end{aligned}
		$$
		Moreover, if $r(h)+n(h)\leq Ch$, then
		\begin{equation*}
			\begin{aligned}
				&\Vert u - u_h\Vert_{L_\mathcal{J}^1(\mV_{-})} + \Vert p - p_h\Vert_{L_\mathcal{J}^1(\mQ_{-})}\\
				&\hspace{1cm}\leq Ch\bigg(\inf_{v\in \mV_h}\Vert u - v\Vert_{L_\mathcal{J}^1(\mV)} + \inf_{q\in \mQ_h}\Vert p - q\Vert_{L_\mathcal{J}^1(\mQ)}\bigg).
			\end{aligned}
		\end{equation*}
	\end{teo}
	\begin{proof}
		Taking time dependent test functions $v\in L_{[0,\tau]}^1(\mV)$ in the first equation of \eqref{dual-problem3}, integrating in $[0,\tau]$, and interchanging the order of integration gives,
		\begin{equation*}
			\begin{aligned}
				&\int_{0}^\tau\langle f_+,v(t)\rangle_{\mV_{+}'\times \mV} dt =\int_0^\tau\bigg\{a\big(v(t),w(t)\big)+b\big(v(t),m(t)\big)\\
				&\hspace{2cm} - \int_{0}^{t}\bigg[k_1(t,s)a(v(s),w(t))+k_2(t,s)b(v(s),m(t)) \bigg]ds\bigg\}dt.
			\end{aligned}
		\end{equation*}
		On the other hand, taking $q\in L_{[0,\tau]}^1(\mV)$ in the second equation of \eqref{dual-problem3} gives
		\begin{equation*}
			\begin{aligned}
				\int_{0}^\tau\langle g_+,q(t)\rangle_{\mQ_{+}'\times \mQ} dt&= \int_0^\tau \bigg[b(w(t),q(t)) -\lambda (q(t),m(t))_\mQ\bigg] dt.
			\end{aligned}
		\end{equation*}
		Now, we set $v=u-u_h$ and $q=p-p_h$ in order to obtain
		\begin{equation}
			\label{dual-teo3-003}
	\begin{aligned}
				\int_{0}^\tau\langle f_+,\eu(t)\rangle_{\mV_{+}'\times \mV} dt&=\int_0^\tau\bigg\{a(\eu(t),w(t))+b(\eu(t),m(t)) \\
			&\hspace{-2cm}- \int_{0}^{t}\bigg[k_1(t,s)a\big(\eu(s),w(t)\big)+k_2(t,s)b\big(\eu(s),m(t)\big) \bigg]ds\bigg\}dt,
	\end{aligned}
		\end{equation}
	and
	\begin{equation}
		\label{dual-teo3-004}
		\int_{0}^{\tau}\langle g_+,\ep(t)\rangle_{\mQ_{+}'\times \mQ}dt=\int_{0}^{\tau}\bigg[b(w(t),\ep(t)) -\lambda (\ep(t),m(t))_\mQ\bigg]dt.
		\end{equation}
		If $f_+(t)=\eu(t)\Vert \eu(t)\Vert_{\mV_{-}}^{-1}$, then $\Vert f_+\Vert_{\mV_{-}}=1$ and $\langle f_+,\eu\rangle_{\mV_{+}'\times \mV_{-}}=\Vert \eu\Vert_{\mV_{-}}.$ This result applied on \eqref{dual-teo3-003} yields to
		\begin{equation*}
			\begin{aligned}
				\int_{0}^\tau\Vert &\eu(t)\Vert_{\mV_{-}}dt=\int_0^\tau\bigg\{a(\eu(t),w(t))+b(\eu(t),m(t)) \\
				&\hspace{1cm}-\int_{0}^{t}\bigg[k_1(t,s)a(\eu(s),w(t))+k_2(t,s)b(\eu(s),m(t)) \bigg]ds\bigg\}dt.
			\end{aligned}
		\end{equation*}
		Similarly, if $g_+(t)=z(t)\ep(t)\Vert \ep(t)\Vert_{\mQ_{-}}^{-1}$, with  
		\begin{equation*}
			z(t)=1+\int_0^tk_2(t,s)\bigg[b(w_h(t),\ep(s))-\lambda(\ep(s),m_h(t))_\mQ\bigg]ds\Vert \ep(t)\Vert_{\mQ_{-}}^{-1},
		\end{equation*}
		then $\Vert g_+(t)\Vert_{\mQ_{-}}=\vert z(t)\vert$. Replacing this in  \eqref{dual-teo3-004},  yields to
		\begin{equation*}
			\begin{aligned}
				\label{dual-teo3-006}
				&\int_{0}^{\tau}\Vert \ep(t)\Vert_{\mQ_{-}}dt= \int_{0}^{\tau}\bigg\{b(w(t),\ep(t))-\lambda(\ep(t),m(t))_\mQ\\
				&\hspace{2cm}- \int_0^tk_2(t,s)\bigg[b(w_h(t),\ep(s))-\lambda(\ep(s),m_h(t))_\mQ\bigg]ds\bigg\}dt.
			\end{aligned}
		\end{equation*}
		But from \eqref{semi-discrete-perturbed-1} we have that
		\begin{equation*}
			\left\{
			\begin{aligned}
				a(\eu,w_h) + b(w_h,\ep) &= \int_{0}^{t}k(t,s)\bigg[a(\eu(s),w_h) + b(w_h,\ep(s)) \bigg]ds,\\
				b(\eu,m_h) -\lambda(\ep,m_h)_{\mQ} &=0,
			\end{aligned}
			\right.
		\end{equation*}
		for all $w_h\in \mV_h$ and for all $m_h\in \mQ_h$. Hence, using the continuity of the bilinear forms, the boundedness of the kernels, H\"older's inequality, and \cite[Lemma 3]{shaw2001optimal}, we obtain 
		\begin{equation*}
			\label{dual-teo3-007}
			\begin{aligned}
				&\int_{0}^\tau\bigg(\Vert \eu\Vert_{\mV_{-}}+\Vert \ep\Vert_{\mQ_{-}}\bigg)dt=\int_0^\tau\bigg\{a(\eu,\ew)+b(\eu,\emm) +b(\ew,\ep)  \\
				&\hspace{0.2cm}-\lambda(\ep,\emm)_\mQ- \int_{0}^{t}\bigg[k_1(t,s)a(\eu(s),\ew)+k_2(t,s)b(\eu(s),\emm)\bigg]ds\bigg\}dt\\
				&\leq \int_{0}^\tau\bigg\{\Vert a \Vert\,\Vert \eu\Vert_{\mV}\Vert \ew\Vert_{\mV}+\Vert b\Vert \,\Vert \emm\Vert_{\mQ}\Vert\eu\Vert_{\mV} +  \Vert b\Vert \,\Vert \ew\Vert_{\mQ}\Vert\ep\Vert_{\mQ}   \\
				&\hspace{0.2cm}+ \lambda\Vert \ep\Vert_{\mQ}\Vert\emm\Vert_{\mQ}+C_k\int_{0}^t\bigg[\Vert a\Vert\,\Vert\eu(s)\Vert_{\mV}\Vert \ew\Vert_{\mV}+ \Vert b\Vert\,\Vert\eu(s)\Vert_{\mV}\Vert \emm\Vert_{\mQ}\bigg]ds \bigg\}dt\\
				&\leq M\bigg(\Vert \eu\Vert_{L_{[0,\tau]}^1(\mV)}+\Vert \ep\Vert_{L_{[0,\tau]}^1(\mQ)}\bigg)\bigg(\Vert\ew\Vert_{L_{[0,\tau]}^{\infty}(\mV)}+\Vert\emm\Vert_{L_{[0,\tau]}^{\infty}(\mQ)}\bigg),
			\end{aligned}
		\end{equation*}
		where $M:=\max\{\Vert a\Vert,\Vert b\Vert,\lambda\}(1+TC_{k})$, $C_k:=\max\{C_{k_1},C_{k_2} \}$. On the other hand, from Hypothesis \ref{dual-hypo1-goal-prob1} and the definition of $r(h)$ and $n(h)$, we have that there exists a constant $\widehat{C}$ such that
		$$\Vert\ew\Vert_{L_{[0,\tau]}^{\infty}(\mV)}+\Vert\emm\Vert_{L_{[0,\tau]}^{\infty}(\mQ)}\leq \widehat{C}\,\bigg[r(h)+n(h)\bigg]\bigg(\Vert f_+\Vert_{L_{[0,\tau]}^{\infty}(\mV_{+}')}+\Vert g_+\Vert_{L_{[0,\tau]}^{\infty}(\mQ_{+}')}\bigg).$$
		Then, we have that 
		\begin{equation}
			\label{dual01}
			\begin{aligned}
				&\Vert \eu\Vert_{L_{[0,\tau]}^1(\mV_{-})}+\Vert \ep\Vert_{L_{[0,\tau]}^1(\mQ_{-})}\\
				&\hspace{0.5cm}\leq \widetilde{C}\,\bigg(\Vert f_+\Vert_{L_{[0,\tau]}^{\infty}(\mV_{+}')}+\Vert g_+\Vert_{L_{[0,\tau]}^{\infty}(\mQ_{+}')}\bigg)\bigg(\Vert \eu\Vert_{L_{[0,\tau]}^1(\mV)}+\Vert \ep\Vert_{L_{[0,\tau]}^1(\mQ)}\bigg),
			\end{aligned}
		\end{equation}
		where $\widetilde{C}:=M\widehat{C}\bigg[r(h)+n(h)\bigg]
		$. We conclude the proof taking $\tau=T,\,\Vert f\Vert_{\mV_{+}'}=\Vert g_+\Vert_{\mQ_{+}'}=1$ and applying Theorem \ref{teo-sd-perturbed-main} to the semi-discrete error estimates
		for $u$ and $p$ in the right side of \eqref{dual01}.
	\end{proof}

	\section{Application to a  linear viscoelastic Timoshenko beam}
	\label{section:TIMO}
	In this section we will apply the abstract theory previously developed to a linear viscoelastic Timoshenko beam model. It is well known that, in the non viscoelastic case, the Timoshenko beam equations lead to a parameter dependent problem, where the thickness  plays the role of
	deteriorate the standard  numerical methods. In the viscoelastic setting this drawback is expectable, and our abstract framework will show that the mixed numerical methods avoid the locking effect for the viscoelastic mixed formulation of this beam.

	Let $\Omega:=[0,L]$ and $\mathcal{J}:=[0,T]$. We consider the space of square-integrable functions $L^2(\Omega)$ with inner product $(u,v)_{0,\Omega}=\int_{\Omega} u\,v\, dx$, and its induced  norm $\Vert f\Vert_{L^2(\Omega)}=\sqrt{(f,f)_{0,\Omega}}$. We denote by $H_0^1(\Omega)$ the subspace of $H^1(\Omega)$ that consist in all functions which together with their first derivative vanish at the ends of the interval $\Omega$. 
	
	The beam is assumed to be clamped, thus we consider the space
	\begin{equation*}
		\label{space-V-timoshenko1}
		\mathbb{H}=\bigg\{ (v\,,\eta\,)\in H_0^1(\Omega)\times H_0^1(\Omega) \bigg\},
	\end{equation*}
	endowed with the natural product space norm
	\begin{equation*}
		\Vert (\eta,w)\Vert_{\mathbb{H}}^2:=\Vert \eta'\Vert_{L^2(\Omega)}^2+\Vert w'\Vert_{L^2(\Omega)}^2.
	\end{equation*}

	The viscoelastic Timoshenko beam model to be analyzed is the linear version of the one studied in \cite{payette2010nonlinear}: Find $(w,\theta)\in L_{\mathcal{J}}^1(\mathbb{H})$ such that
	\begin{equation}
		\label{s-m-b-timoshenko1}
		\begin{aligned}
			&E(0)(I(x) \theta', \eta' )_{0,\Omega}+ k_sG(0)\big( A(x)(\theta-w'),\eta-v'\big)_{0,\Omega}=(\widetilde{q},v)_{0,\Omega}\\
			&+\int_0^t\dot{E}(t-s) (I(x)\theta'(s),\eta\,' )_{0,\Omega}\, ds\\&+k_s\int_0^t \dot{G}(t-s)\big(A(x) (\theta(s) -w'(s)),\eta\,-v' \big)_{0,\Omega}\, ds,\hspace{0.6cm} \forall (v\,,\eta\,)\in\mathbb{H},
		\end{aligned}
	\end{equation}
	where $w$ is the displacement of the beam, $\theta$ represent the rotations,  $k_s$ is the correction factor, $E(t)$ is the relaxation modulus,  $G(t):=E(t)/2(1+\nu)$ is the shear modulus,  $\nu$ is the time-independent Poisson ratio, $I(x)$ is the moment of inertia of the cross-section, $A(x)$ is the area of the cross-sectio,n and $\widetilde{q}(x,t)$ is a uniform distributed transverse given  load.
	
	It is well known that, in elastic beams, numerical locking arises when standard finite elements are used because most of the energy of the system will be given by the shear term $(\theta-w')$ and this is not physically correct (see \cite{arnold1981discretization,chapelle2010finite,hernandez2008approximation}). A usual approach to analyze this phenomenon is to rescale the formulation \eqref{s-m-b-timoshenko1} to identify a family of viscoelastic problems whose limit is well-posed when the thickness of the beam goes to zero. With this purpose, we introduce the following classic non-dimensional parameter, characteristic of the thickness of the beam
	\begin{equation*}
		\varepsilon^2=\frac{1}{L}\int_\Omega \frac{I(x)}{A(x)L^2}\,dx,
	\end{equation*}
	which is assumed to be independent of time and is such that  $\varepsilon\in(0,\varepsilon_{\max}]$.
	
	By scaling the load as $\widetilde{q}(x,t)=\varepsilon^3 q(x,t)$, with $q(x,t)$ independent of $\varepsilon$, and defining
	\begin{equation*}
		\hat{I}(x):=\frac{I(x)}{\varepsilon^3}, \qquad \hat{A}(x):=k_s\frac{A(x)}{\varepsilon},
	\end{equation*}
	we have that \eqref{s-m-b-timoshenko1} is equivalently written as follows: \begin{problem}
		\label{prob-timoshenko-2}
		Given $q\in L_{\mathcal{J}}^1(L^2(\Omega)\big)$, find $(\theta,w)\in L_\mathcal{J}^1(\mathbb{H})$ such that	
		\begin{equation*}
			\label{s-m-b-timoshenko4-qs2}
			\begin{aligned}
				(\hat{I} \theta'&, \eta' )_{0,\Omega}+ \frac{\varepsilon^{-2}}{2(1+\nu)}\big( \hat{A}(\theta-w'),\eta-v'\big)_{0,\Omega}=(q_E,v)_{0,\Omega}\\&+\int_0^t\frac{\dot{E}(t-s)}{E(0)}\left[  (\hat{I}\theta'(s),\eta\,' )_{0,\Omega}+\frac{\varepsilon^{-2}}{2(1+\nu)}\big(\hat{A} (\theta(s) -w'(s)) ,\eta\,-v' \big)_{0,\Omega}\right]\, ds,
			\end{aligned}
		\end{equation*}
		for all $(v\,,\eta\,)\in\mathbb{H}$, where $q_E=q/E(0)$.
	\end{problem}
	
	Assuming that there exist positive constants $\underline{C}_E,\overline{C}_E,\underline{C}_G,\overline{C}_G$ such that
	\begin{equation*}
		\overline{C}_E\geq E(0)\hat{I}(x)\geq\underline{C}_E,\quad
		\overline{C}_G\geq\frac{ E(0)}{2(1+\nu)}\hat{A}(x)\geq\underline{C}_G,
	\end{equation*}
	we have that, for each $\varepsilon>0$, the bilinear form from the left hand side of Problem \ref{prob-timoshenko-2} is continuous and elliptic in $\mathbb{H}$. Hence, the associated linear operator is invertible, and following \cite[Theorem 3.1 and Theorem 3.5]{gripenberg1990volterra}, we have that there exists a unique solution $(\theta,w)\in L_\mathcal{J}^1(\mathbb{H})$ of Problem \ref{prob-timoshenko-2} . Moreover, from \cite[Theorem 9]{shaw2001optimal} we conclude that there exists a positive constant $C>0$, such that 
	$$
	\Vert w\Vert_{L_{\mathcal{J}}^1(H^1(\Omega))} + \Vert \theta\Vert_{L_{\mathcal{J}}^1(H^1(\Omega))}\leq C\varepsilon^{-1}\Vert q_E\Vert_{L_\mathcal{J}^1(L^2(\Omega))}.
	$$

	For Problem \ref{prob-timoshenko-2} we define the unit elastic shear $\gamma\in L_{\mathcal{J}}^1(L^2(\Omega)\big)$ as
	\begin{equation}
		\label{gamma-timoshenko-beam}
		\gamma:=\frac{\varepsilon^{-2}}{2(1+\nu)}\hat{A}(\theta-w').
	\end{equation}
	Moreover, multiplying \eqref{gamma-timoshenko-beam} by a test function $\psi\in L^2(\Omega)$ and integrating in $\Omega$, we have 
	\begin{equation}
		\label{gamma-timoshenko-beam-weak}
		(\theta-w',\psi)_{0,\Omega}- \lambda (\gamma/\hat{A},\psi)_{0,\Omega}=0 \hspace{1cm} \forall \psi\in L^2(\Omega),
	\end{equation}
	where $\lambda=2(1+\nu)\varepsilon^2$. Hence, gathering Problem \ref{s-m-b-timoshenko4-qs2} and \eqref{gamma-timoshenko-beam-weak}, we obtain the following mixed formulation: \begin{problem}
		\label{prob-timoshenko-mixed1}	Find $(\theta,w,\gamma)\in L_\mathcal{J}^1(\mathbb{H}\times L^2(\Omega))$ such that
		\begin{equation*}
			\label{mixed-timoshenko-beam-1}
			\left\{\begin{aligned}
				&(\hat{I} \theta', \eta' )_{0,\Omega}+ \big( \gamma,\eta-v'\big)_{0,\Omega}=(q_E,v)_{0,\Omega}\\
				&\hspace{4cm}+\int_0^t\frac{\dot{E}(t-s)}{E(0)}\bigg[  \big(\hat{I}\theta'(s),\eta\,' \big)_{0,\Omega}+\big(\gamma(s),\eta-v' \big)_{0,\Omega}\bigg] ds,\\
				&(\theta-w',\psi)_{0,\Omega} - \lambda (\gamma/\hat{A},\psi)_{0,\Omega}=0,
			\end{aligned}\right.
		\end{equation*}
		for all $(v,\eta)\in\mathbb{H}$ and for all $\psi\in L^2(\Omega)$.
	\end{problem} 
	
	Now we will prove that Problem \ref{prob-timoshenko-mixed1}  lies in the framework of the abstract setting provided in the previous section. With this aim, we define $\mV:=\mathbb{H},\,\mQ:=L^2(\Omega)$, $u:=(\theta,w),\,v:=(\eta,v),p:=\gamma, q:=\psi$, the bilinear forms
	\begin{equation*}
		a\big((\theta,w);(\eta,v)\big):=(\hat{I}\theta',\eta')_{0,\Omega},\quad b\big((\eta,v);\gamma\big):=(\gamma,\eta-v')_{0,\Omega}
	\end{equation*}
	and $k(t,s)=\dot{E}(t-s)/E(0)$, for all $(\theta,w),(\eta,v)\in\mathbb{H}$, $\gamma\in L^2(\Omega)$ a.e. in $\mathcal{J}$. Clearly, the bilinear form $a(\cdot,\cdot)$ is positive semi-definite. On the other hand, we have that $\mK=\big\{(\eta,v)\in\mathbb{H}\;:\; v\in H_0^1(\Omega),\,\eta=v' \big\}.$ Hence, due to the Poincar\'e inequality, we conclude that
	$a(\cdot,\cdot)$ is elliptic over the space $\mK$. 
	
	On the other hand, from \cite[Section 5]{arnold1981discretization} we have that $b(\cdot,\cdot)$ satisfies an inf-sup condition. Then, we observe that Problem \ref{prob-timoshenko-mixed1} satisifes the hypotheses from Theorem \ref{teo-p_lambda}, implying the existence of a positive constant $C>0$, uniform in $\lambda$, such that
	\begin{equation*}
		\label{shear-formulation-001}
		\Vert (\theta,w)\Vert_{L_\mathcal{J}^1(\mathbb{H})} + \Vert \gamma\Vert_{L_\mathcal{J}^1(L^2(\Omega))}\leq C\Vert q_E\Vert_{L_\mathcal{J}^1(L^2(\Omega))}.
	\end{equation*}
	
	Here and thereafter, the positive constant $C$ represents, either uniform with respect to the thickness parameter or independent of it, which depends on the stability parameters (continuous or discrete), the spatial domain, the observation time and the nature of the material. To end the section, we note that the stability result
	\begin{equation}
		\label{perturbed-better-estimate}
		\begin{aligned}
			&\Vert \theta\Vert_{L_{\mathcal{J}}^ 1(H^{k}(\Omega))}+\Vert w\Vert_{L_{\mathcal{J}}^ 1(H^{k}(\Omega))} + \Vert \gamma\Vert_{L_{\mathcal{J}}^ 1(H^{k-1}(\Omega))}\leq C\Vert q_E\Vert_{L_{\mathcal{J}}^ 1(H^{k-2}(\Omega))},
		\end{aligned}
	\end{equation}
	holds, for $k\geq2$, with $H^0(\Omega)=L^2(\Omega)$.
	\subsection{Finite element analysis}
	\label{subsection4-1}
	In this section we analyze the finite element semi-discretization for the beam mixed formulation provided previously. The main goal is to derive error estimates, independent of the thickness parameter. As a starting point, consider a finite partition $\mathscr{T}_h=\big\{\Omega_i \big\}_{i=1}^{n}$  of the computational domain $\Omega$ such that $\Omega_i=]x_{i-1},x_i[$, with length $h_i=x_i-x_{i-1}$, and satisfying $\bigcap_{i=1}^n\Omega_i=\emptyset$ and  $\Omega=\bigcup_{i=1}^n\Omega_i$, $i=1,\dots,n$. The maximum interval length is denoted by $h=\max_{1\leq i \leq n}h_i$.
	
	The approximations in each case will be based in the following finite element spaces:
	\begin{align*}
		\mV_h^r&:=\bigg\{v\in H^1(\Omega)\;:\; v_{\vert \Omega_i}\in \mathbb{P}_r(\Omega_i),\, \Omega_i\in\mathscr{T}_h \bigg\},\\
		\mQ_h^r&:=\bigg\{q\in L^2(\Omega)\;:\; v_{\vert \Omega_i}\in \mathbb{P}_{r-1}(\Omega_i),\, \Omega_i\in\mathscr{T}_h \bigg\},
	\end{align*}
	where $\mathbb{P}_r$ represents the space of polynomials of degree $r\geq 1$ defined on each $\Omega_i$. Also, we introduce the $L^2-$ projection onto $\mQ_h^r$, defined by
	\begin{equation*}
		\begin{aligned}
			\Pi_h: &L^2(\Omega)\rightarrow \mQ_h^r\\
			&p\mapsto \Pi_h(p):=\widetilde{p},
		\end{aligned}
	\end{equation*}
	which satisfies $(p-\widetilde{p},q)_{0,\Omega}=0,$ for all $q\in \mQ_h^r$, and the following error estimate
	\begin{equation}
		\label{projector-l2-bound}
		\Vert p - \Pi_h(p)\Vert_{L^2(\Omega)}\leq Ch^{r}\Vert p\Vert_{H^{r}(\Omega)},
	\end{equation}
	is satisfied (see for example \cite[Section 1.6.3]{ern2013theory}). We also consider the Lagrange interpolant $\mathcal{L}_h:C(\overline{\Omega})\rightarrow \mV_h^r$ satisfying the error estimate (see for example \cite[Section 1.1.3]{ern2013theory}):
	\begin{equation}
		\label{lagrange-interpolant-r}
		\Vert v- \mathcal{L}_h(v)\Vert_{H^1(\Omega)}\leq Ch^r\Vert v\Vert_{H^{r+1}(\Omega)}.
	\end{equation} 
	
	Define $\mathbb{H}_h:=\mV_h^r\times \mV_h^r$ as a finite element subspace of  $\mathbb{H}$. Then, the corresponding semi-discrete counterpart of Problem \ref{prob-timoshenko-mixed1} is given as follows:
	
	\begin{problem}
		\label{prob-timoshenko-mixed1-discreto}	Find $(\theta_h,w_h,\gamma_h)\in L_\mathcal{J}^1(\mathbb{H}_h\times \mQ_h^r)$ such that
		\begin{equation*}
			\label{mixed-timoshenko-beam-1-discreto}
			\left\{\begin{aligned}
				&(\hat{I} \theta_h', \eta' )_{0,\Omega}+ \big( \gamma_h,\eta-v'\big)_{0,\Omega}=(q_E,v)_{0,\Omega}\\
				&\hspace{4cm}+\int_0^t\frac{\dot{E}(t-s)}{E(0)}\bigg[  \big(\hat{I}\theta_h'(s),\eta\,' \big)_{0,\Omega} +\big(\gamma_h(s),\eta-v' \big)_{0,\Omega}\bigg] ds,\\
				&(\theta_h-w_h',\psi)_{0,\Omega}- \lambda (\gamma_h/\hat{A},\psi)_{0,\Omega}= 0,\\
			\end{aligned}\right.
		\end{equation*}
		for all $(v,\eta)\in\mathbb{H}_h$ and for all $\psi\in \mQ_h^r$.
	\end{problem}

	Following \cite[Section 5]{arnold1981discretization}, we observe that the mixed formulation allows us to define the discrete kernel $\mK_h:=\left\{(\eta,v)\;:\;v'=\Pi_h\eta \right\}$ in order to conclude that the restriction of $a(\cdot,\cdot)$ to $\mathbb{H}_h$ satisfies the ellipticity condition in $\mK_h$, while the restriction  $b(\cdot,\cdot)$ to $\mathbb{H}_h\times\mQ_h^r$ satisfies an inf-sup condition. Thus, applying Theorem \ref{teo-sd-perturbed-main} we obtain that there exists a positive constant $C$, independent of $h$ and  $\lambda$, such that
	\begin{equation}
		\label{perturbed-001}
		\begin{aligned}
			\Vert (\theta,w)-&(\theta_h,w_h)\Vert_{L_\mathcal{J}^1(\mathbb{H})} + \Vert \gamma-\gamma_h\Vert_{L_\mathcal{J}^1(\mathbb{Q})}\\
			&\leq C\bigg(\inf_{(\eta,v)\in \mathbb{H}_h}\Vert(\theta,w)-(\eta,v)\Vert_{L_\mathcal{J}^1(\mathbb{H})}+\inf_{\psi\in \mQ_h^r}\Vert \gamma-\psi\Vert_{L_\mathcal{J}^1(\mathbb{Q})}\bigg).
		\end{aligned}
	\end{equation}
	
	Hence, we have the following convergence rate of the semi-discrete mixed Problem \ref{prob-timoshenko-mixed1-discreto}.
	\begin{prop}
		\label{prop:mixed-timoshenko-convergence1}
		Let $(\theta,w,\gamma)\in L_{\mathcal{J}}^ 1(\mathbb{H\times Q})$ be the solution of Problem \ref{prob-timoshenko-mixed1} and $(\theta_h,w_h,\gamma_h)\in L_\mathcal{J}^1(\mathbb{H}_h\times \mathbb{Q}_h)$. Then, if $q_E\in L_\mathcal{J}^1(H^r(\Omega))$, there exists a constant $C>0$, independent of $h$ and $\lambda$, such that
		\begin{equation*}
			\label{prop-mixed1-001}
			\Vert (\theta,w)-(\theta_h,w_h)\Vert_{L_\mathcal{J}^1(\mathbb{H})} + \Vert \gamma-\gamma_h\Vert_{L_\mathcal{J}^1(\mathbb{Q})}\leq Ch^r\Vert q_E\Vert_{L_\mathcal{J}^1(H^r(\Omega))}.
		\end{equation*}
	\end{prop}
	\begin{proof}
		From \eqref{perturbed-001} and the error estimates for $\mathcal{L}_h$ and $\Pi_h$ in \eqref{lagrange-interpolant-r} and \eqref{projector-l2-bound}, respectively, we obtain  
		\begin{equation*}
			\begin{aligned}
				&\Vert (\theta,w)-(\theta_h,w_h)\Vert_{L^1(\mathcal{J};\mathbb{H})} + \Vert \gamma-\gamma_h\Vert_{L^1(\mathcal{J};\mathbb{Q})}\\
				&\hspace{1cm}\leq C\bigg(\Vert(\theta,w)-(\mathcal{L}_h(\eta_h),\mathcal{L}_h( v_h))\Vert_{L^1(\mathcal{J};\mathbb{H})}+\Vert \gamma-\Pi_h(\psi_h)\Vert_{L^1(\mathcal{J};\mathbb{Q})}\bigg)\\
				&\hspace{1cm}\leq Ch^r\bigg(\Vert \theta\Vert_{L^1(\mathcal{J};H^{r+1}(\Omega))}+ \Vert w\Vert_{L^1(\mathcal{J};H^{r+1}(\Omega))}+\Vert \gamma\Vert_{L^1(\mathcal{J};H^r(\Omega))}\bigg).
			\end{aligned}
		\end{equation*}
		Then, the desired estimate follows from \eqref{perturbed-better-estimate}.  
	\end{proof}
	In what follows, we will consider the dual-backward version of Problem \ref{prob-timoshenko-mixed1-discreto} to obtain an additional error estimate. Note that the estimate for $\gamma$ can not be improved since the choice of a space less regular that $L^2(\Omega)$ is not available.
	
	\begin{prop}
		Under the assumptions of Proposition \ref{prop:mixed-timoshenko-convergence1}, there exists a constant $C>0$, independent of $h$ and $\lambda$, such that
		\begin{equation*}
			\label{prop-mixed1-002}
			\Vert (\theta,w)-(\theta_h,w_h)\Vert_{L_\mathcal{J}^1(L^2(\Omega))}\leq C h^{r+1}\Vert q_E\Vert_{L_\mathcal{J}^1(\,H^{r}(\Omega))}.
		\end{equation*}
	\end{prop}
	\begin{proof}
		The dual-backward formulation of Problem \ref{prob-timoshenko-mixed1}, is stated as follows: For any $\tau\in\mathcal{J}$ and for any $q_{E+}\in L_{[0,\tau]}^\infty(L^2(\Omega))$, find $(\widetilde{w},\widetilde{\theta},\widetilde{\gamma})\in L_{[0,\tau]}^\infty(\mV\times L^2(\Omega))$, such that  for a.e. in $[0,\tau]$,
		\begin{equation}
			\label{dual-problem3_discrete}
			\left\{\begin{aligned}
				&a((\eta,v);(\widetilde{\theta}(t),\widetilde{w}(t)))+b((\eta,v);\widetilde{\gamma}(t))=(q_{E+}(t),v)_{0,\Omega}\\
				&\hspace{2.5cm}+\int_{t}^\tau k(s,t)\bigg[ a((\eta,v);(\widetilde{\theta}(s),\widetilde{w}(s)))+b((\eta,v);\gamma(s))\bigg]ds,\\
				&b((\widetilde{\theta}(t),\widetilde{w}(t));q) -\lambda (q,\widetilde{\gamma}(t)/\hat{A})_{0,\Omega}=0,
			\end{aligned}\right.
		\end{equation}
		for all $(\eta,v)\in \mV\times\mQ$. 
		
		Setting $\xi:=\tau - t$, $\chi=\tau - s$, and defining $\overline{w}(\cdot):=\widetilde{w}(\tau - \cdot),\;\overline{\theta}(\cdot):=\widetilde{\theta}(\tau - \cdot),\;\overline{\gamma}(\cdot):=\widetilde{\gamma}(\tau - \cdot),\; \overline{f}_{+}(\xi):=f_{+}(\tau -\xi)$, and $\overline{k}(\xi,\chi):=k(\tau-\chi,\tau - \xi),$ it follows that the backward problem can be written in forward form as: Find $(\overline{w},\overline{\theta},\overline{\gamma})\in L_{[0,\tau]}^\infty(\mV\times L^2(\Omega))$, such that for a.e. $\xi\in [0,\tau]$, 
		\begin{equation}
			\label{dual-problem3_discrete2}
			\left\{\begin{aligned}
				&a((\eta,v);(\overline{\theta}(\xi),\overline{w}(\xi)))+b((\eta,v);\overline{\gamma}(\xi))=(q_{E+}(\tau-\xi),v)_{0,\Omega}\\
				&\hspace{2.5cm}+\int_{0}^\xi \overline{k}(\xi,\chi)\bigg[ a((\eta,v);(\overline{\theta}(\chi),\overline{w}(\chi)))+b((\eta,v);\gamma(\chi))\bigg]d\chi,\\
				&b((\overline{\theta}(\xi),\overline{w}(\xi));q) -\lambda (q,\overline{\gamma}(\xi)/\hat{A})_{0,\Omega}=0,
			\end{aligned}\right.
		\end{equation}
		
		Now we need the existence and uniqueness of a triplet $(\overline{w},\overline{\theta},\overline{\gamma})\in L_{[0,\tau]}^\infty(\mV\times L^2(\Omega))$, solution of \eqref{dual-problem3_discrete2} (resp. of a triplet $(\widetilde{w},\widetilde{\theta},\widetilde{\gamma})\in L_{[0,\tau]}^\infty(\mV\times L^2(\Omega))$ solution of \eqref{dual-problem3_discrete}). Note that that we have the same bilinear forms as those of the abstract setting, hence the associated linear operator is invertible. Also, we have that $\overline{k}\in L^1(\mathcal{T})$ and $q_{E+}\in L_{[0,\tau]}^\infty(L^2(\Omega))$. Hence, from \cite[Lemma 4]{shaw2001optimal}, we obtain the existence and uniqueness of the required dual solutions. 
		To obtain the additional regularity in space, we use the differential equations satisfied by $\overline{w},\overline{\theta},\overline{\gamma}$ and use mathematical induction. 
		
		Therefore, the dual-backward problem of Problem \ref{prob-timoshenko-mixed1} satisfies the Hypothesis \ref{dual-hypo1-goal-prob1}, so the desired improved convergence follows from Proposition \ref{prop:mixed-timoshenko-convergence1} and Theorem \ref{dual-teo3-goal-prob}.
	\end{proof}

	\subsection{Numerical tests}
	\label{section5}
	In this section, we report  several numerical tests to check the mechanical behavior of the Timoshenko beam obtained using the finite element formulations developed above. The discrete problem is solved with Python scripts and the FEniCS project \cite{fenics2015}. We divide this section in quasi-static and dynamical analysis, where different Timoshenko beam scenarios are considered. 
	
	In the following tests,  the experimental nature of the relaxation modulus is replaced by assumed values of spring constants and viscosity parameters in order to consider the \emph{Standard Linear Solid model} (SLS). The relaxation and shear modulus for this material is given by the truncated Prony series:
	\begin{equation*}
		\label{relaxation-modulus-SLS}
		E(t)=\frac{k_1k_2}{k_1+k_2}+\bigg(k_1-\frac{k_1k_2}{k_1+k_2} \bigg)e^{-t/\tau},\qquad G(t)=\frac{E(t)}{2(1+\nu)},
	\end{equation*}
	where $\tau=\eta/(k_1+k_2)$. We will consider  $\nu=0.35$ in all the experiments.
	\subsection{Convergence and quasi-static response of several beams}
	\label{subsection5-1}

	Although the analysis provided predicts errors when using polynomials of degree at most $r-1$, the numerical experiments will be restricted to the two most used low order elements, i.e., linear and quadratic piecewise continuous elements.  In the first part, clamped and simply supported beams are considered using $\mathbb{P}_1$ and $\mathbb{P}_2$ elements to check the convergence order of the finite element method.	
	
	For the quasi-static case, we define the errors by
	\begin{equation*}
		\texttt{e}_0(\kappa):=\Vert \kappa-\kappa_h\Vert_{L_\mathcal{J}^1(L^2(\Omega))} \quad\text{ and }\quad \texttt{e}_1(\kappa):=\Vert \kappa- \kappa_h\Vert_{L_\mathcal{J}^1(H^1(\Omega))},
	\end{equation*}
	for every function $f$. 
	This allows to define the experimental rate of convergence  $\texttt{r}_i(\cdot)$ as
	\begin{equation*}
		\texttt{r}_i(\cdot):=\frac{\log\big(\texttt{e}_i(\cdot)/\texttt{e}_i'(\cdot)\big)}{\log(h/h')},\;\; i=0,1,
	\end{equation*}
	where $\texttt{e}_i$ and $\texttt{e}_i'$ denote two consecutive relative errors and $h$ and $h'$ their corresponding mesh sizes. Also, let $\text{DOF}:=\dim (\mathbb{H}_h)$ be the degrees of freedom in the $h-$discretization.
	
	It is well known that the trapezoidal rule error is of order $2$. Thus, from the semi-discrete error estimates, we have that given a semi-discrete rate of convergence $\mathcal{O}(h^r)$, we expect that the fully discrete error estimates satisfies
	\begin{align}
		&\texttt{e}_0(\cdot)\leq C(h^{r+1}+\Delta t^2), &\texttt{e}_1(\cdot)\leq C(h^r+\Delta t^2),\label{errorh1-w-theta}
	\end{align}
	where $C$ is a constant that depends of $\Omega$, the step size $\Delta t$, the observation time $T$ and the material creep response, but not on the thickness of the beam (see \cite{hernandez2008approximation} for more details). 
	
	Hence, the experiments will consider a sufficiently small step size such that the estimates \eqref{errorh1-w-theta} obeys the $\mathcal{O}(h^r)$ rate only. Numerical examples in \cite{payette2010nonlinear} support this claim by performing several test to associate a small size step with an accurate solution, whether the method applied is convergent.
	
	We borrow the material considered in \cite{martin2016modified} and consider an homogeneous rectangular beam of length $L=4\,m$, with base $b=0.08\,m$ and thickness $d$. The corresponding moment of inertia is $I=0.08\,d^3/12$m$^4$ and the area is $A=0.08\,d$\,m$^2$. The corresponding thickness parameter is set to be $\varepsilon^2=I/AL^2$. The beam is subjected to a creep load $q(t)=8\,H(t)$\,N/m. This case considers the SLS parameters $k_1=9.8\times10^{7}$\,N/m$^2$, $k_2=2.44\times10^{7}$\,N/m$^2$ and $\eta=2.74\cdot 10^8$\,N$\cdot$ s/m$^2$.  The observation time is $10\,s$ with a step size $\Delta t$ depending on the elements used so that it is considerably smaller than any mesh size selected in the test. We subdivide this experiment in two boundary condition cases. The first case is a fully clamped boundary condition in order to test the locking-free nature of the method proposed. The second case aims to check if the method is also locking free in other boundary conditions.
	
	\subsubsection{Clamped beam}
	\label{subsubsection5-1-1}
	Now we report numerical results for a clamped viscoelastic beam. From Table \ref{table1:error_l2_w-clamped} to Table \ref{table2:error_h1_theta-clamped} we observe that the predicted convergence rates are reached for the elements used. For instance, we show in Figure \ref{fig:qs_ex_0_exacta vs reducida-clamped} the differences in the convergence when using different techniques to solve the discrete problem with the thinnest beam studied. The reduced integration gives $\mathcal{O}(h^r)$ and $\mathcal{O}(h^{r+1})$ rates, with $r=1$ for $\mathbb{P}_1$ and $r=2$ for $\mathbb{P}_2$ elements, while a total loss of convergence is observed when the exact integration and $\mathbb{P}_1$ are used. Moreover, from Figure \ref{fig:qs_ex_0_exacta vs reducida-clamped}(c) we observe that exact integration with $\mathbb{P}_2$ elements gives a convergence rate of $\mathcal{O}(h^{2})$ for $\texttt{e}_i(w),i=0,1$ and $\texttt{e}_0(\theta)$, while  $\mathcal{O}(h)$ is observed for $\texttt{e}_1(\theta)$. This is suboptimal with respect to the predicted rate of convergence, similar to that of an elastic beam.
	It is worth noting that, beside the error behavior of $\texttt{e}_1(w)$, the slope of the errors with exact integration and $\mathbb{P}_2$ elements is very similar to that of reduced integration and $\mathbb{P}_1$ elements.
		\begin{table}[!ht]
		\caption{$L_\mathcal{J}^1(L^2(\Omega))$ error values and rate of convergence of the transverse displacement $w$ in a fully clamped viscoelastic beam with $\eta=2.74$\,N$\cdot$s$/$m$^2$. The observation time is $T=10\,s$ with $5000$ and $32000$ time steps when using linear and quadratic elements, respectively.}
		\label{table1:error_l2_w-clamped}
		\tabcolsep=2pt\begin{tabular}{AACCCCCCCC}
			\hline
			DOF&$h$&\multicolumn{2}{c}{$d=10^{-1}\mbox{m}$}&&\multicolumn{2}{c}{$d=10^{-2}\mbox{m}$}&&\multicolumn{2}{c}{$d=10^{-3}\mbox{m}$}\\\cline{3-4}\cline{6-7}\cline{9-10}
			&&\texttt{e}_0(w)&\texttt{r}_0(w)&&\texttt{e}_0(w)&\texttt{r}_0(w)&&\texttt{e}_0(w)&\texttt{r}_0(w)\\
			\multicolumn{2}{c}{$\mathbb{P}_1,\; \Delta t=2.00e-3$}\\
			\hline
			42 &0.2   &6.4146e-10  &--   & & 6.4076e-10 &--   &   &6.4075e-10  &--  \\
			82 &0.1   &1.6071e-10  &1.99 & &1.6054e-10  &1.99 &   &1.6053e-10  &1.99   \\
			162&0.05  &4.0196e-11 &1.99 &  &4.0152e-11  &1.99 &   &4.0152e-11 & 1.99  \\
			202&0.04  &2.5725e-11 &2.00 &  &2.5697e-11  &2.00 &   &2.5697e-11 & 2.00  \\
			242&0.03  &1.7863e-11 &2.00 &  &1.7844e-11  &2.00 &   &1.7844e-11 & 2.00  \\
			282&0.028 &1.3123e-11 &2.00 &  &1.3109e-11  &2.00 &   &1.3109e-11 & 2.00  \\
			\multicolumn{2}{c}{$\mathbb{P}_2,\; \Delta t=3.13e-4$}\\
			\hline
			82 &0.2   &5.6609e-12  &--   &  &5.6609e-12 &--   &   &5.6609e-12  &--  \\
			162 &0.1  &7.0766e-13  &2.99 &  &7.0765e-13 &2.99 &   &7.0765e-13  &2.99   \\
			322&0.05  &8.8526e-14  &2.99 &  &8.8525e-14 &2.99 &   &8.8525e-14 & 2.99  \\
			402&0.04  &4.5366e-14  &2.99 &  &4.5365e-14 &2.99 &   &4.5365e-14 & 2.99  \\
			482&0.03  &2.6289e-14  &2.99 &  &2.6288e-14 &2.99 &   &2.6288e-14 & 2.99  \\
			562&0.028 &1.6586e-14  &2.98 &  &1.6586e-14 &2.98 &   &1.6586e-14 & 2.98  \\
			\hline
		\end{tabular}
		
	\end{table}
	
	\begin{table}[!ht]
		\caption{$L_\mathcal{J}^1(H^1(\Omega))$ error data and rate of convergence of the transverse displacement $w$ in a fully clamped viscoelastic beam with $\eta=2.74$\,N$\cdot$s$/$m$^2$. The observation time is $T=10\,s$ with $5000$ and $32000$ time steps when using linear and quadratic elements, respectively.}
		\label{table2:error_h1_w-clamped}
		\tabcolsep=2pt\begin{tabular}{AACCCCCCCC}
			\hline
			DOF&$h$&\multicolumn{2}{c}{$d=10^{-1}\mbox{m}$}&&\multicolumn{2}{c}{$d=10^{-2}\mbox{m}$}&&\multicolumn{2}{c}{$d=10^{-3}\mbox{m}$}\\\cline{3-4}\cline{6-7}\cline{9-10}
			&&\texttt{e}_1(w)&\texttt{r}_1(w)&&\texttt{e}_1(w)&\texttt{r}_1(w)&&\texttt{e}_1(w)&\texttt{r}_1(w)\\
			\multicolumn{2}{c}{$\mathbb{P}_1,\; \Delta t=2.00e-3$}\\
			\hline
			42 &0.2   &3.7300e-9  &--   &  &3.7298e-9  &--   &   &3.7298e-9 &--  \\
			82 &0.1   &1.8425e-9  &1.01 &  &1.8424e-9  &1.01 &   &1.8424e-9 &1.01   \\
			162&0.05  &9.1841e-10 &1.00 &  &9.1840e-10 &1.00 &   &9.1840e-10 & 1.00  \\
			202&0.04  &7.3446e-10 &1.00 &  &7.3445e-10 &1.00 &   &7.3445e-10 & 1.00  \\
			242&0.03  &6.1192e-10 &1.00 &  &6.1191e-10 &1.00 &   &6.1191e-10 & 1.00 \\
			282&0.028 &5.2444e-10 &1.00 &  &5.2443e-10 &1.00 &   &5.2443e-10 & 1.00  \\
			\multicolumn{2}{c}{$\mathbb{P}_2,\; \Delta t=3.13e-4$}\\
			\hline
			82 &0.2   &1.8335e-10  &--   &  &1.8335e-10 &--   &   &1.8335e-10  &--  \\
			162 &0.1  &4.5854e-11  &1.99 &  &4.5854e-11 &1.99 &   &4.5854e-11  &1.99   \\
			322&0.05  &1.1464e-11  &1.99 &  &1.1464e-11 &1.99 &   &1.1464e-11 & 1.99  \\
			402&0.04  &7.3373e-12  &1.99 &  &7.3373e-12 &1.99 &   &7.3373e-12 & 1.99 \\
			482&0.03  &5.0954e-12  &1.99 &  &5.0954e-12 &1.99 &   &5.0954e-12 & 1.99  \\
			562&0.028 &3.7435e-12  &1.99 &  &3.7435e-12 &1.99 &   &3.7435e-12 & 1.99  \\
			\hline
		\end{tabular}
	\end{table}
	
	\begin{table}[!ht]
		\caption{$L_{\mathcal{J}}^1(L^2(\Omega)\big)$ error values and rate of convergence of the rotation $\theta$ in a fully clamped viscoelastic beam with $\eta=2.74$\,N$\cdot$s$/$m$^2$. The observation time is $T=10\,s$ with $5000$ and $32000$ time steps when using linear and quadratic elements, respectively.}
		\label{table1:error_l2_theta-clamped}
		\tabcolsep=2pt\begin{tabular}{AACCCCCCCC}
			\hline
			DOF&$h$&\multicolumn{2}{c}{$d=10^{-1}\mbox{m}$}&&\multicolumn{2}{c}{$d=10^{-2}\mbox{m}$}&&\multicolumn{2}{c}{$d=10^{-3}\mbox{m}$}\\\cline{3-4}\cline{6-7}\cline{9-10}
			&&\texttt{e}_0(\theta)&\texttt{r}_0(\theta)&&\texttt{e}_0(\theta)&\texttt{r}_0(\theta)&&\texttt{e}_0(\theta)&\texttt{r}_0(\theta)\\
			\multicolumn{2}{c}{$\mathbb{P}_1,\; \Delta t=2.00e-3$}\\
			\hline
			42 &0.2   &4.4892e-10  &--  &  &4.4892e-10 &--   &   &4.4892e-10 &--  \\
			82 &0.1   &1.1232e-10 &1.99 &  &1.1232e-10 &1.99 &   &1.1232e-10 &1.99   \\
			162&0.05  &2.8085e-11 &1.99 &  &2.8085e-11 &1.99 &   &2.8085e-11 & 1.99  \\
			202&0.04  &1.7973e-11 &2.00 &  &1.7973e-11 &2.00 &   &1.7973e-11 & 2.00  \\
			242&0.03  &1.2481e-11 &2.00 &  &1.2481e-11 &2.00 &   &1.2481e-11 & 2.00  \\
			282&0.028 &9.1692e-12 &2.00 &  &9.1692e-12 &2.00 &   &9.1692e-12 & 2.00  \\
			\multicolumn{2}{c}{$\mathbb{P}_2,\; \Delta t=3.13e-4$}\\
			\hline
			82 &0.2   &4.9025e-12  &--  &  &4.9025e-12 &--   &   &4.9025e-12 &--  \\
			162 &0.1  &6.1285e-13 &2.99 &  &6.1285e-13 &2.99 &   &6.1285e-13 &2.99   \\
			322&0.05  &7.6668e-14 &2.99 &  &7.6668e-14 &2.99 &   &7.6668e-14 & 2.99  \\
			402&0.04  &3.9289e-14 &2.99 &  &3.9289e-14 &2.99 &   &3.9289e-14 & 2.99  \\
			482&0.03  &2.2768e-14 &2.99 &  &2.2768e-14 &2.99 &   &2.2768e-14 & 2.99  \\
			562&0.028 &1.4365e-14 &2.98 &  &1.4365e-14 &2.98 &   &1.4365e-14 & 2.98  \\
			\hline
		\end{tabular}
		
	\end{table}
	
	\begin{table}[!ht]
		\caption{$L_{\mathcal{J}}^1(H^1(\Omega)\big)$ error data and rate of convergence for the rotation $\theta$ in a fully clamped viscoelastic beam. The observation time is $T=10\,s$ with $5000$ and $32000$ time steps, when using linear and quadratic elements, respectively.}
		\label{table2:error_h1_theta-clamped}
		\tabcolsep=2pt\begin{tabular}{AACCCCCCCC}
			\hline
			DOF&$h$&\multicolumn{2}{c}{$d=10^{-1}\mbox{m}$}&&\multicolumn{2}{c}{$d=10^{-2}\mbox{m}$}&&\multicolumn{2}{c}{$d=10^{-3}\mbox{m}$}\\\cline{3-4}\cline{6-7}\cline{9-10}
			&&\texttt{e}_1(\theta)&\texttt{r}_1(\theta)&&\texttt{e}_1(\theta)&\texttt{r}_1(\theta)&&\texttt{e}_1(\theta)&\texttt{r}_1(\theta)\\
			\multicolumn{2}{c}{$\mathbb{P}_1,\; \Delta t=2.00e-3$}\\
			\hline
			42 &0.2   &7.1136e-9 &--    &   &7.1136e-9 &--   &   &7.1136e-9  & --  \\
			82 &0.1   &3.5541e-9 &1.00  &   &3.5541e-9 &1.00 &   &3.5541e-9  & 1.00   \\
			162&0.05  &1.7767e-9 &1.00  &   &1.7767e-9 &1.00 &   &1.7767e-9  & 1.00  \\
			202&0.04  &1.4213e-9 &1.00  &   &1.4213e-9 &1.00 &   &1.4213e-9  & 1.00   \\
			242&0.03  &1.1844e-9 &1.00  &   &1.1844e-9 &1.00 &   &1.1844e-9  & 1.00   \\
			282&0.028 &1.0152e-9 &1.00  &   &1.0152e-9 &1.00 &   &1.0152e-9  & 1.00   \\
			\multicolumn{2}{c}{$\mathbb{P}_2,\; \Delta t=3.13e-4$}\\
			\hline
			82 &0.2   &1.5893e-10 &--   &  &1.5893e-10 &--   &   &1.5893e-10  &--  \\
			162 &0.1  &3.9719e-11 &2.00 &  &3.9719e-11 &2.00 &   &3.9719e-11 &2.00   \\
			322&0.05  &9.9290e-12 &2.00 &  &9.9290e-12 &2.00 &   &9.9290e-12 & 2.00  \\
			402&0.04  &6.3545e-12 &2.00 &  &6.3545e-12 &2.00 &   &6.3545e-12 & 2.00  \\
			482&0.03  &4.4128e-12 &2.00 &  &4.4128e-12 &2.00 &   &4.4128e-12 & 2.00  \\
			562&0.028 &3.2421e-12 &1.99 &  &3.2421e-12 &1.99 &   &3.2421e-12 & 1.99 \\
			\hline
		\end{tabular}
	\end{table}
	
	\begin{figure}[!ht]
		\begin{minipage}[c]{0.45\linewidth}
			\begin{center}
				\includegraphics[scale=0.35]{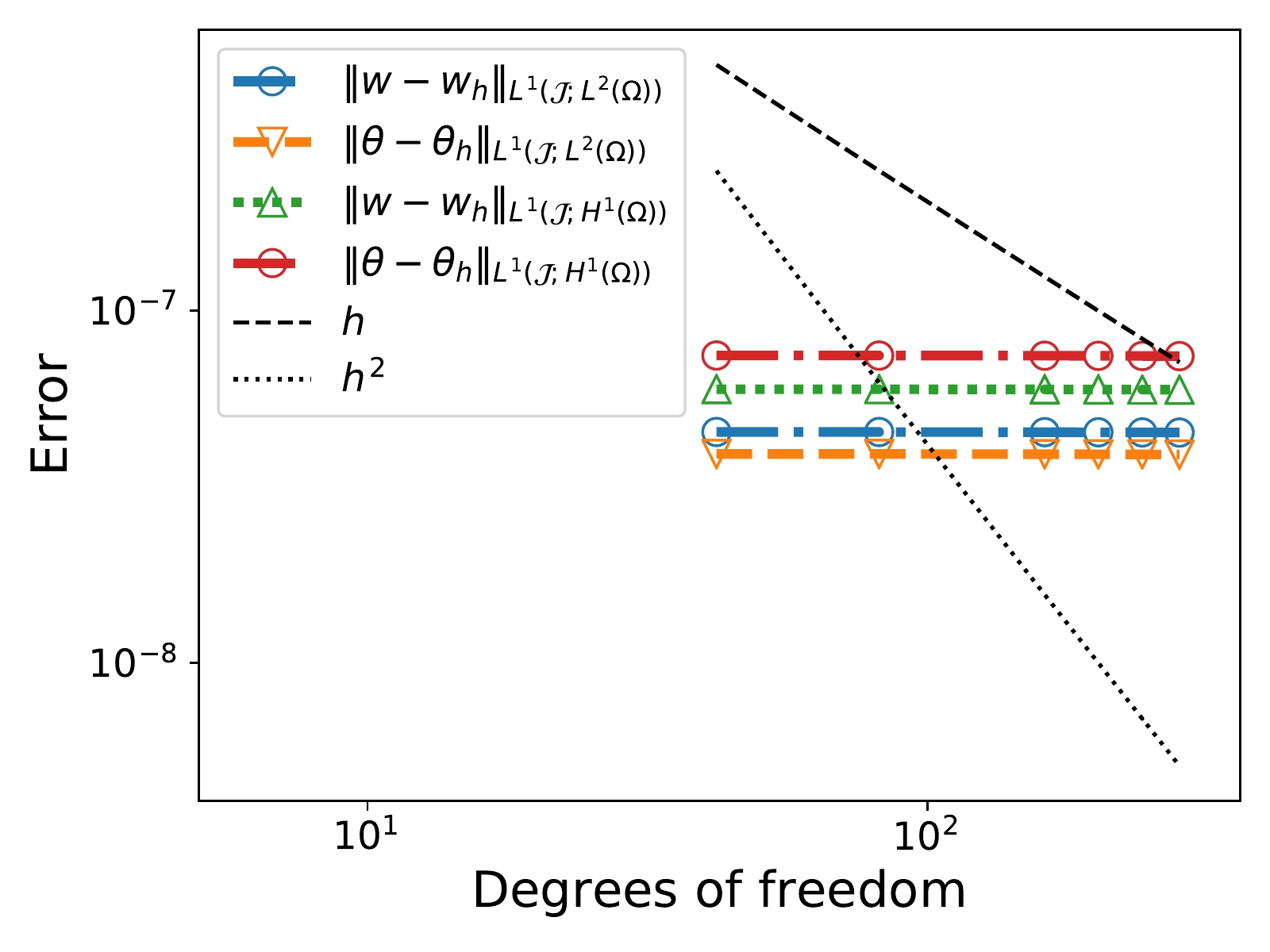}\\
				{\hspace{1cm}\footnotesize (a)}
			\end{center}
		\end{minipage}
		\begin{minipage}[c]{0.45\linewidth}
			\begin{center}
				\includegraphics[scale=0.35]{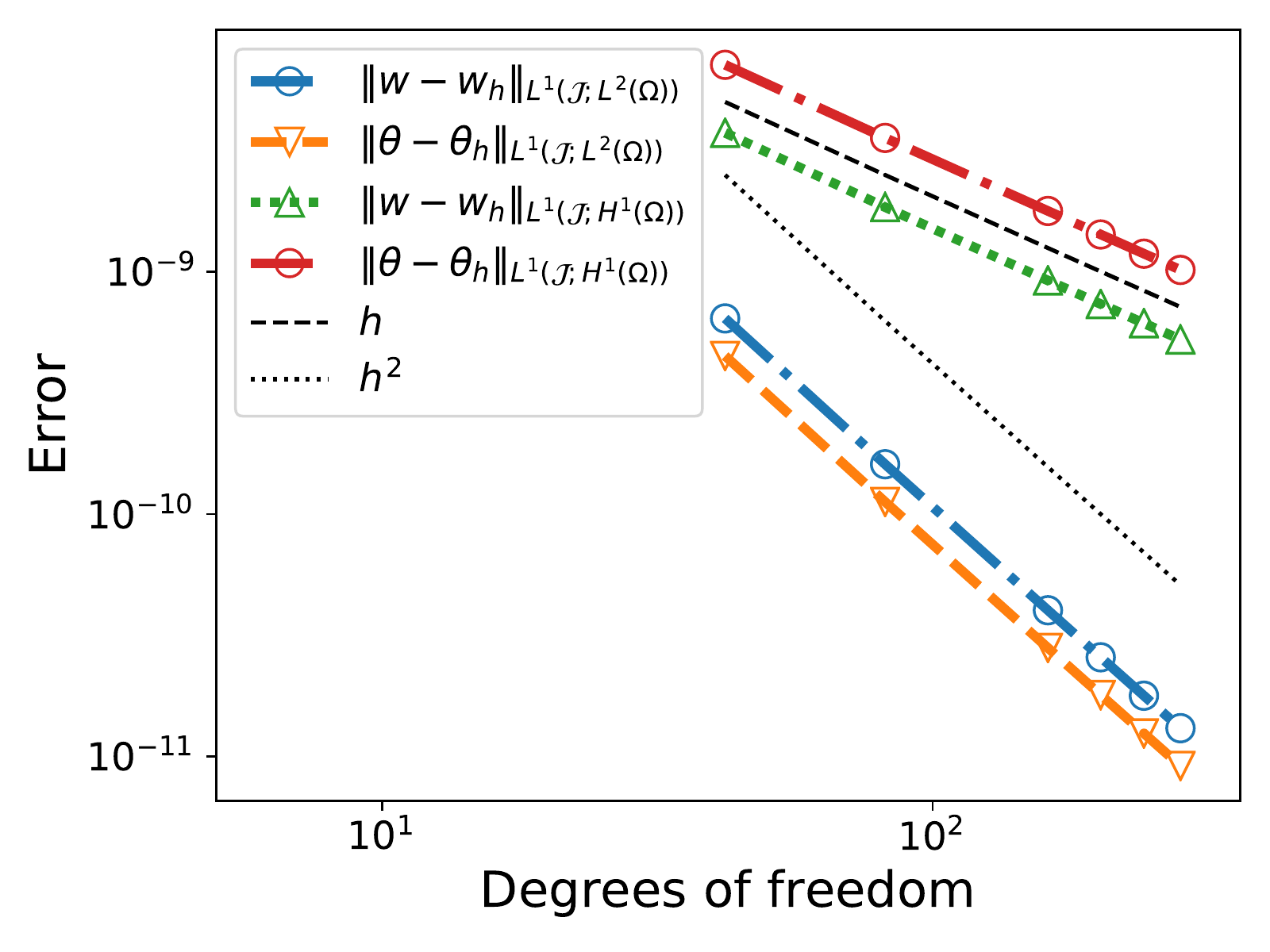}\\
				{\hspace{1cm}\footnotesize (b)}
			\end{center}
		\end{minipage}
		\begin{minipage}[c]{0.45\linewidth}
			\begin{center}
				\includegraphics[scale=0.35]{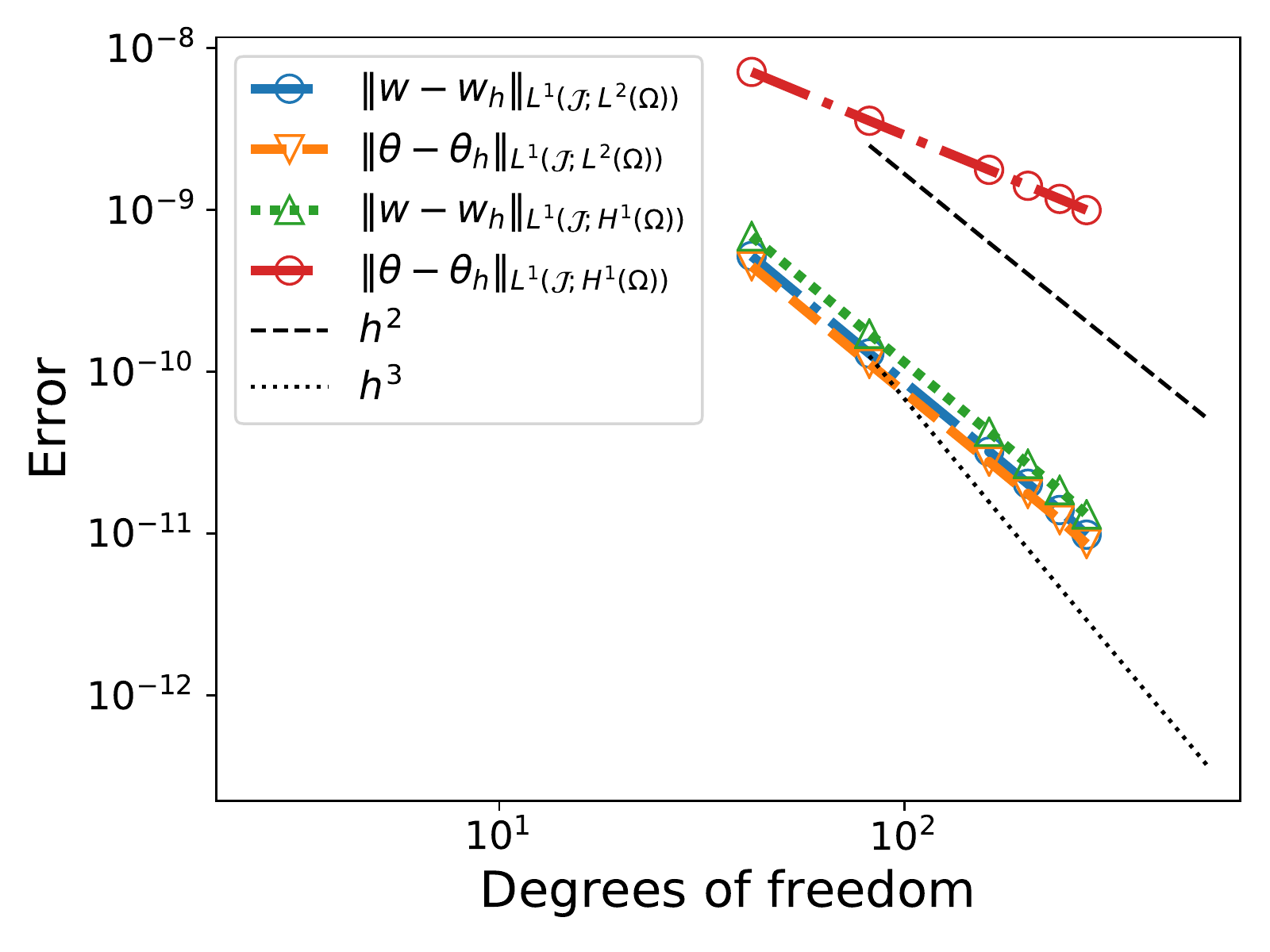}\\
				{\hspace{1cm}\footnotesize (c)}
			\end{center}
		\end{minipage}
		\begin{minipage}[c]{0.45\linewidth}
			\begin{center}
				\includegraphics[scale=0.35]{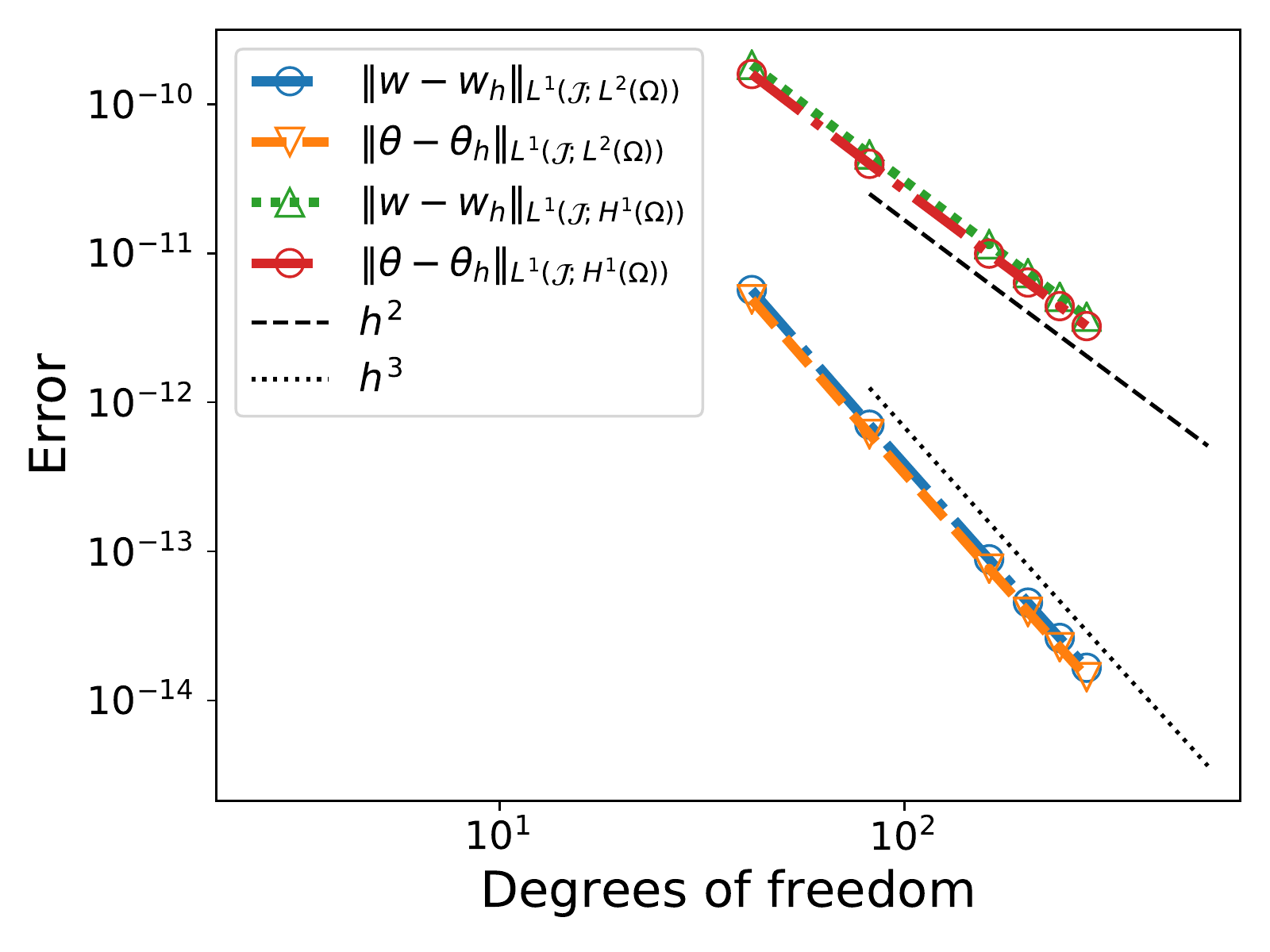}\\
				{\hspace{1cm}\footnotesize (d)}
			\end{center}
		\end{minipage}
		\caption{Effect of the thickness on the error  for a fully clamped homogeneous viscoelastic beam with thickness $d=0.001,m$ when using (a) exact integration and $\mathbb{P}_1$ elements, (b) reduced integration and $\mathbb{P}_1$ elements, (c) exact integration and $\mathbb{P}_2$ elements and (d) reduced integration and $\mathbb{P}_2$ elements.  The observation time is $T=10\,s$ with $5000$ and $32000$ time steps when using linear and quadratic elements, respectively. The values of $h$, $h^2$ and $h^3$ has been properly scaled to fit the graph.}
		\label{fig:qs_ex_0_exacta vs reducida-clamped}
	\end{figure}

	\subsubsection{Beam under simply supported boundary conditions. }
	\label{subsubsection5-1-2}
	Now we present numerical results for the simply supported Timoshenko beam. In Tables \ref{table1:error_l2_w} - \ref{table2:error_h1_theta} are reported the convergence results for different values of thickness, where we observe that the convergence rates are the expected according to our theory. We compare this behavior with the exact integration procedure in Figure \ref{fig:qs_ex_0_exacta vs reducida}, where we observe the errors  on $L_{\mathcal{J}}^1(L^2(\Omega))$ and $L_{\mathcal{J}}^1(H^1(\Omega))$  when the selected thickness is $d=0.001\,m$. Numerical locking using piecewise linear functions and exact integration is clearly visible. In Figure \ref{fig:qs_ex_0_exacta vs reducida} is clearly seen that the order of convergence is suboptimal when exact integration is used for quadratic elements.
	
	\begin{table}[!ht]
		\caption{$L_\mathcal{J}^1(L^2(\Omega))$ error values and rate of convergence of the transverse displacement $w$ in a simply supported viscoelastic beam with $\eta=2.74\,$N$\cdot$s$/$m$^2$. The observation time is $T=10\,s$ with $5000$ and $32000$ time steps when using linear and quadratic elements, respectively.}
		\label{table1:error_l2_w}
		\tabcolsep=2pt\begin{tabular}{AACCCCCCCC}
			\hline
			DOF&$h$&\multicolumn{2}{c}{$d=10^{-1}\mbox{m}$}&&\multicolumn{2}{c}{$d=10^{-2}\mbox{m}$}&&\multicolumn{2}{c}{$d=10^{-3}\mbox{m}$}\\\cline{3-4}\cline{6-7}\cline{9-10}
			&&\texttt{e}_0(w)&\texttt{r}_0(w)&&\texttt{e}_0(w)&\texttt{r}_0(w)&&\texttt{e}_0(w)&\texttt{r}_0(w)\\
			\multicolumn{2}{c}{$\mathbb{P}_1,\; \Delta t=2.00e-3$}\\
			\hline
			42 &0.2   &1.5723e-9  &--   & & 1.5714e-9 &--   &   &1.5714e-9  &--  \\
			82 &0.1   &3.9357e-10  &1.99 & &3.9334e-10  &1.99 &   &3.9334e-10  &1.99   \\
			162&0.05  &9.8399e-11 &1.99 &  &9.8342e-11  &1.99 &   &9.8341e-11 & 1.99  \\
			202&0.04  &6.2969e-11 &2.00 &  &6.2933e-11  &2.00 &   &6.2933e-11 & 2.00  \\
			242&0.03  &4.3724e-11 &2.00 &  &4.3698e-11  &2.00 &   &4.3698e-11 & 2.00  \\
			282&0.028 &3.2119e-11 &2.00 &  &3.2100e-11  &2.00 &   &3.2100e-11 & 2.00  \\
			\multicolumn{2}{c}{$\mathbb{P}_2,\; \Delta t=3.13e-4$}\\
			\hline
			82 &0.2   &5.6611e-12  &--   &  &5.6611e-12 &--   &   &5.6611e-12  &--  \\
			162 &0.1  &7.0799e-13  &2.99 &  &7.0799e-13 &2.99 &   &7.0799e-13  &2.99   \\
			322&0.05  &8.8901e-14  &2.99 &  &8.8900e-14 &2.99 &   &8.8900e-14 & 2.99  \\
			402&0.04  &4.5749e-14  &2.97 &  &4.5748e-14 &2.97 &   &4.5748e-14 & 2.97  \\
			482&0.03  &2.6681e-14  &2.95 &  &2.6680e-14 &2.95 &   &2.6680e-14 & 2.95  \\
			562&0.028 &1.6990e-14  &2.92 &  &1.6989e-14 &2.92 &   &1.6989e-14 & 2.92  \\
			\hline
		\end{tabular}
		
	\end{table}
	
	\begin{table}[!ht]
		\caption{$L_\mathcal{J}^1(H^1(\Omega))$ error data and rate of convergence of the transverse displacement $w$ in a simply supported viscoelastic beam with $\eta=2.74$\,N$\cdot$s$/$m$^2$. The observation time is $T=10\,s$ with $5000$ and $32000$ time steps when using linear and quadratic elements, respectively.}
		\label{table2:error_h1_w}
		\tabcolsep=2pt\begin{tabular}{AACCCCCCCC}
			\hline
			DOF&$h$&\multicolumn{2}{c}{$d=10^{-1}\mbox{m}$}&&\multicolumn{2}{c}{$d=10^{-2}\mbox{m}$}&&\multicolumn{2}{c}{$d=10^{-3}\mbox{m}$}\\\cline{3-4}\cline{6-7}\cline{9-10}
			&&\texttt{e}_1(w)&\texttt{r}_1(w)&&\texttt{e}_1(w)&\texttt{r}_1(w)&&\texttt{e}_1(w)&\texttt{r}_1(w)\\
			\multicolumn{2}{c}{$\mathbb{P}_1,\; \Delta t=2.00e-3$}\\
			\hline
			42 &0.2    &9.1736e-9  &--   &  &9.1587e-9  &--   &   &9.1586e-9 &--  \\
			82 &0.1    &4.5235e-9  &1.02 &  &4.5160e-9  &1.02 &   &4.5159e-9 &1.02   \\
			162&0.05   &2.2537e-9 &1.00 &   &2.2500e-9 &1.00 &    &2.2499e-9 & 1.00  \\
			202&0.04   &1.8022e-9 &1.00 &   &1.7992e-9 &1.00 &    &1.7992e-9 & 1.00  \\
			242&0.03   &1.5015e-9 &1.00 &   &1.4990e-9 &1.00 &    &1.4989e-9 & 1.00 \\
			282&0.028  &1.2868e-9 &1.00 &   &1.2846e-9 &1.00 &    &1.2846e-9 & 1.00  \\
			\multicolumn{2}{c}{$\mathbb{P}_2,\; \Delta t=3.13e-4$}\\
			\hline
			82 &0.2   &1.8335e-10  &--   &  &1.8335e-10 &--   &   &1.8335e-10  &--  \\
			162 &0.1  &4.5854e-11  &1.99 &  &4.5854e-11 &1.99 &   &4.5854e-11  &1.99   \\
			322&0.05  &1.1464e-11  &1.99 &  &1.1464e-11 &1.99 &   &1.1464e-11 & 1.99  \\
			402&0.04  &7.3375e-12  &1.99 &  &7.3375e-12 &1.99 &   &7.3375e-12 & 1.99 \\
			482&0.03  &5.0957e-12  &1.99 &  &5.0956e-12 &1.99 &   &5.0956e-12 & 1.99  \\
			562&0.028 &3.7439e-12  &1.99 &  &3.7439e-12 &1.99 &   &3.7439e-12 & 1.99  \\
			\hline
		\end{tabular}
		
	\end{table}
	
	\begin{table}[!ht]
		\caption{$L_\mathcal{J}^1(L^2(\Omega))$ error values and rate of convergence of the rotation $\theta$ in a simply supported viscoelastic beam with $\eta=2.74$\,N$\cdot$s$/$m$^2$. The observation time is $T=10\,s$ with $5000$ and $32000$ time steps when using linear and quadratic elements, respectively.}
		\label{table1:error_l2_theta}
		\tabcolsep=2pt\begin{tabular}{AACCCCCCCC}
			\hline
			DOF&$h$&\multicolumn{2}{c}{$d=10^{-1}\mbox{m}$}&&\multicolumn{2}{c}{$d=10^{-2}\mbox{m}$}&&\multicolumn{2}{c}{$d=10^{-3}\mbox{m}$}\\\cline{3-4}\cline{6-7}\cline{9-10}
			&&\texttt{e}_0(\theta)&\texttt{r}_0(\theta)&&\texttt{e}_0(\theta)&\texttt{r}_0(\theta)&&\texttt{e}_0(\theta)&\texttt{r}_0(\theta)\\
			\multicolumn{2}{c}{$\mathbb{P}_1,\; \Delta t=2.00e-3$}\\
			\hline
			42 &0.2   &8.4015e-10  &--  &  &8.4015e-10 &--   &   &8.4015e-10 &--  \\
			82 &0.1   &2.1014e-10 &1.99 &  &2.1014e-10 &1.99 &   &2.1014e-10 &1.99   \\
			162&0.05  &5.2527e-11 &2.00 &  &5.2527e-11 &2.00 &   &5.2527e-11 & 2.00  \\
			202&0.04  &3.3613e-11 &2.00 &  &3.3613e-11 &2.00 &   &3.3613e-11 & 2.00  \\
			242&0.03  &2.3338e-11 &2.00 &  &2.3338e-11 &2.00 &   &2.3338e-11 & 2.00  \\
			282&0.028 &1.7142e-11 &2.00 &  &1.7142e-11 &2.00 &   &1.7142e-11 & 2.00  \\
			\multicolumn{2}{c}{$\mathbb{P}_2,\; \Delta t=3.13e-4$}\\
			\hline
			82 &0.2   &4.9026e-12  &--  &  &4.9026e-12 &--   &   &4.9026e-12 &--  \\
			162 &0.1  &6.1311e-13 &2.99 &  &6.1311e-13 &2.99 &   &6.1311e-13 &2.99   \\
			322&0.05  &7.6960e-14 &2.99 &  &7.6960e-14 &2.99 &   &7.6960e-14 & 2.99  \\
			402&0.04  &3.9587e-14 &2.97 &  &3.9587e-14 &2.97 &   &3.9587e-14 & 2.97  \\
			482&0.03  &2.3071e-14 &2.96 &  &2.3071e-14 &2.96 &   &2.3071e-14 & 2.96  \\
			562&0.028 &1.4676e-14 &2.93 &  &1.4676e-14 &2.93 &   &1.4676e-14 & 2.93  \\
			\hline
		\end{tabular}
		
	\end{table}
	
	\begin{table}[!ht]
		\caption{$L_\mathcal{J}^1(H^1(\Omega))$ error data and rate of convergence of the rotation $\theta$ in a simply supported viscoelastic beam. The observation time is $T=10\,s$ with $5000$ and $32000$ time steps when using linear and quadratic elements, respectively.}
		\label{table2:error_h1_theta}
		\tabcolsep=2pt\begin{tabular}{AACCCCCCCC}
			\hline
			DOF&$h$&\multicolumn{2}{c}{$d=10^{-1}\mbox{m}$}&&\multicolumn{2}{c}{$d=10^{-2}\mbox{m}$}&&\multicolumn{2}{c}{$d=10^{-3}\mbox{m}$}\\\cline{3-4}\cline{6-7}\cline{9-10}
			&&\texttt{e}_1(\theta)&\texttt{r}_1(\theta)&&\texttt{e}_1(\theta)&\texttt{r}_1(\theta)&&\texttt{e}_1(\theta)&\texttt{r}_1(\theta)\\
			\multicolumn{2}{c}{$\mathbb{P}_1,\; \Delta t=2.00e-3$}\\
			\hline
			42 &0.2   &7.1578e-9 &--    &   &7.1578e-9 &--   &   &7.1578e-9  & --  \\
			82 &0.1   &3.5597e-9 &1.00  &   &3.5597e-9 &1.00 &   &3.5597e-9  & 1.00   \\
			162&0.05  &1.7774e-9 &1.00  &   &1.7774e-9 &1.00 &   &1.7774e-9  & 1.00  \\
			202&0.04  &1.4217e-9 &1.00  &   &1.4217e-9 &1.00 &   &1.4217e-9  & 1.00   \\
			242&0.03  &1.1846e-9 &1.00  &   &1.1846e-9 &1.00 &   &1.1846e-9  & 1.00   \\
			282&0.028 &1.0153e-9 &1.00  &   &1.0153e-9 &1.00 &   &1.0153e-9  & 1.00   \\
			\multicolumn{2}{c}{$\mathbb{P}_2,\; \Delta t=3.13e-4$}\\
			\hline
			82 &0.2   &1.5893\textit{e}-10 &--   &  &1.5893e-10 &--   &   &1.5893e-10  &--  \\
			162 &0.1  &3.9719e-11 &2.00 &  &3.9719e-11 &2.00 &   &3.9719e-11 &2.00   \\
			322&0.05  &9.9292e-12 &2.00 &  &9.9292e-12 &2.00 &   &9.9292e-12 & 2.00  \\
			402&0.04  &6.3547e-12 &1.99 &  &6.3547e-12 &1.99 &   &6.3547e-12 & 1.99  \\
			482&0.03  &4.4130e-12 &1.99 &  &4.4130e-12 &1.99 &   &4.4130e-12 & 1.99  \\
			562&0.028 &3.2423e-12 &1.99 &  &3.2423e-12 &1.99 &   &3.2423e-12 & 1.99 \\
			\hline
		\end{tabular}
		
	\end{table}
	
	\begin{figure}[!ht]
		\begin{minipage}[c]{0.45\linewidth}
			\begin{center}
				\includegraphics[scale=0.35]{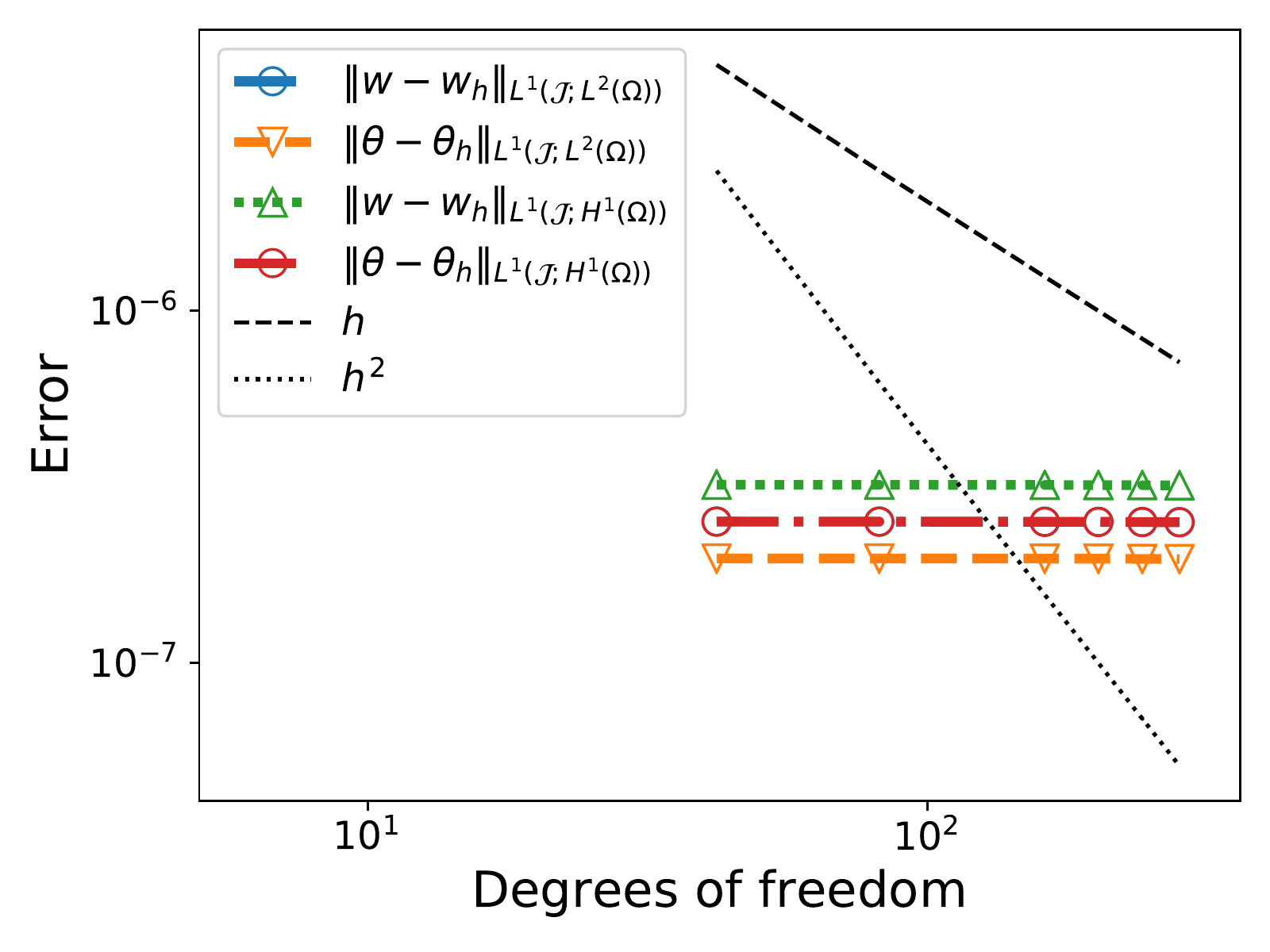}\\
				{\hspace{1cm}\footnotesize (a)}
			\end{center}
		\end{minipage}
		\begin{minipage}[c]{0.45\linewidth}
			\begin{center}
				\includegraphics[scale=0.35]{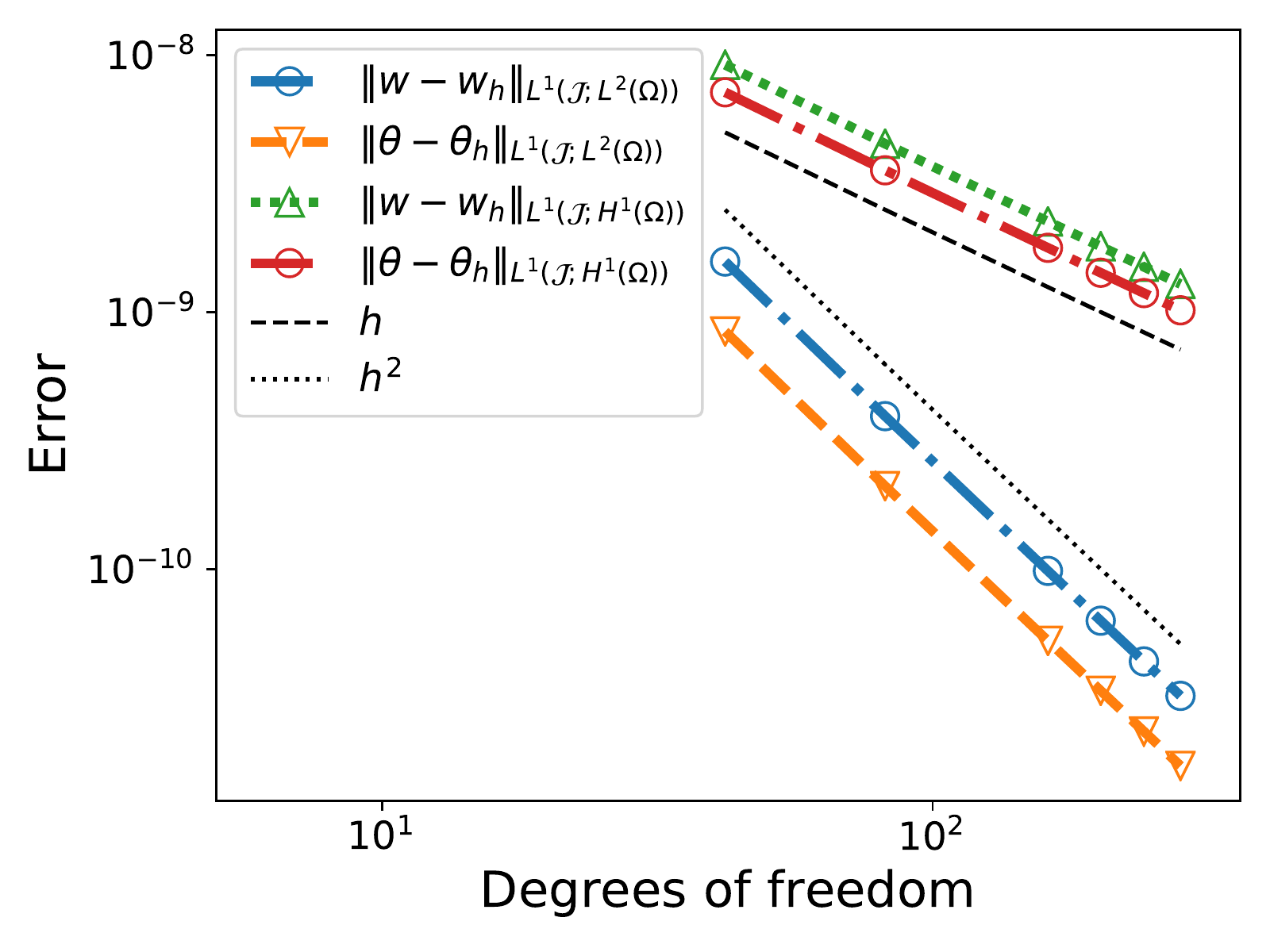}\\
				{\hspace{1cm}\footnotesize (b)}
			\end{center}
		\end{minipage}
		\begin{minipage}[c]{0.45\linewidth}
			\begin{center}
				\includegraphics[scale=0.35]{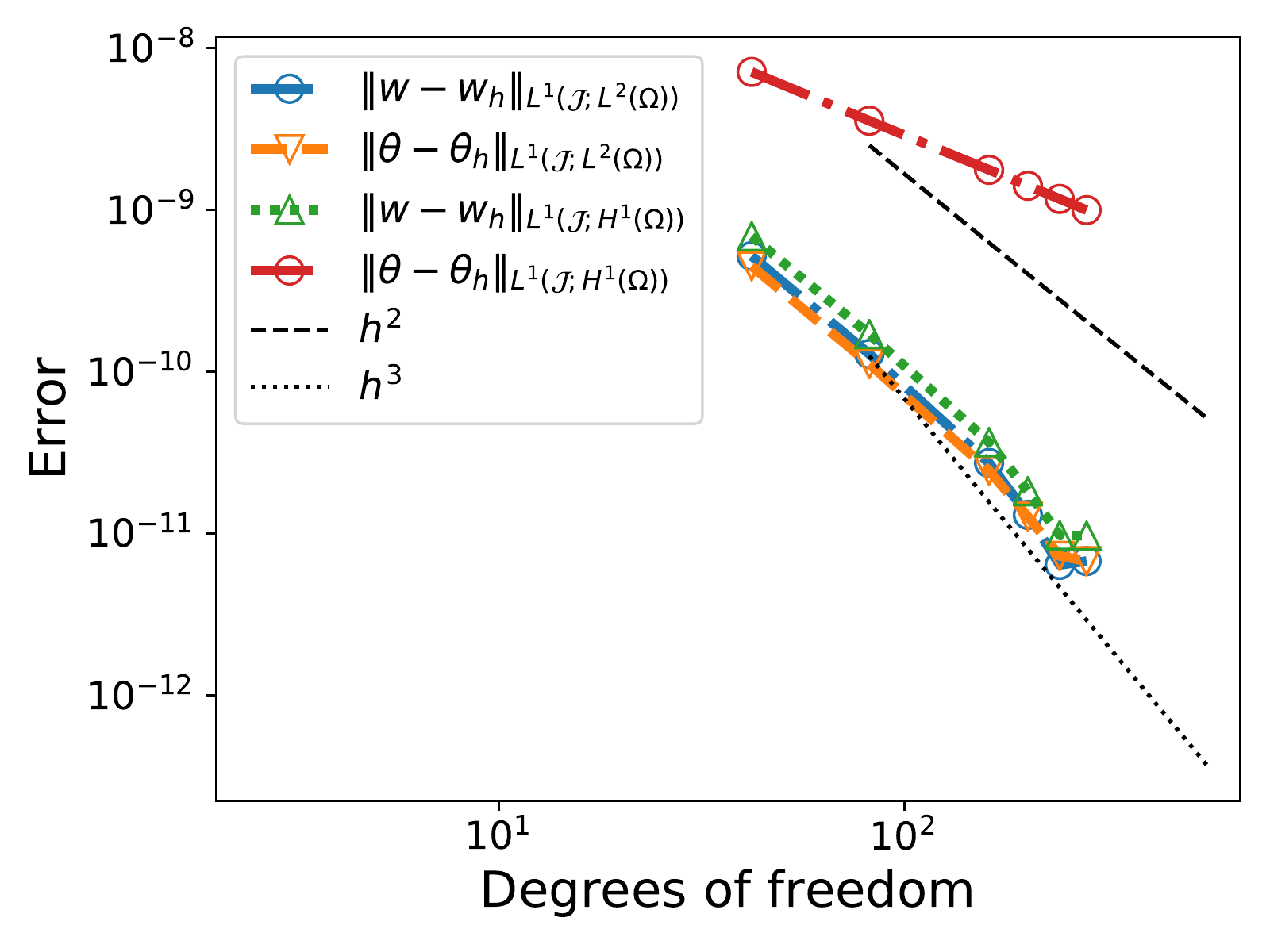}\\
				{\hspace{1cm}\footnotesize (c)}
			\end{center}
		\end{minipage}
		\begin{minipage}[c]{0.45\linewidth}
			\begin{center}
				\includegraphics[scale=0.35]{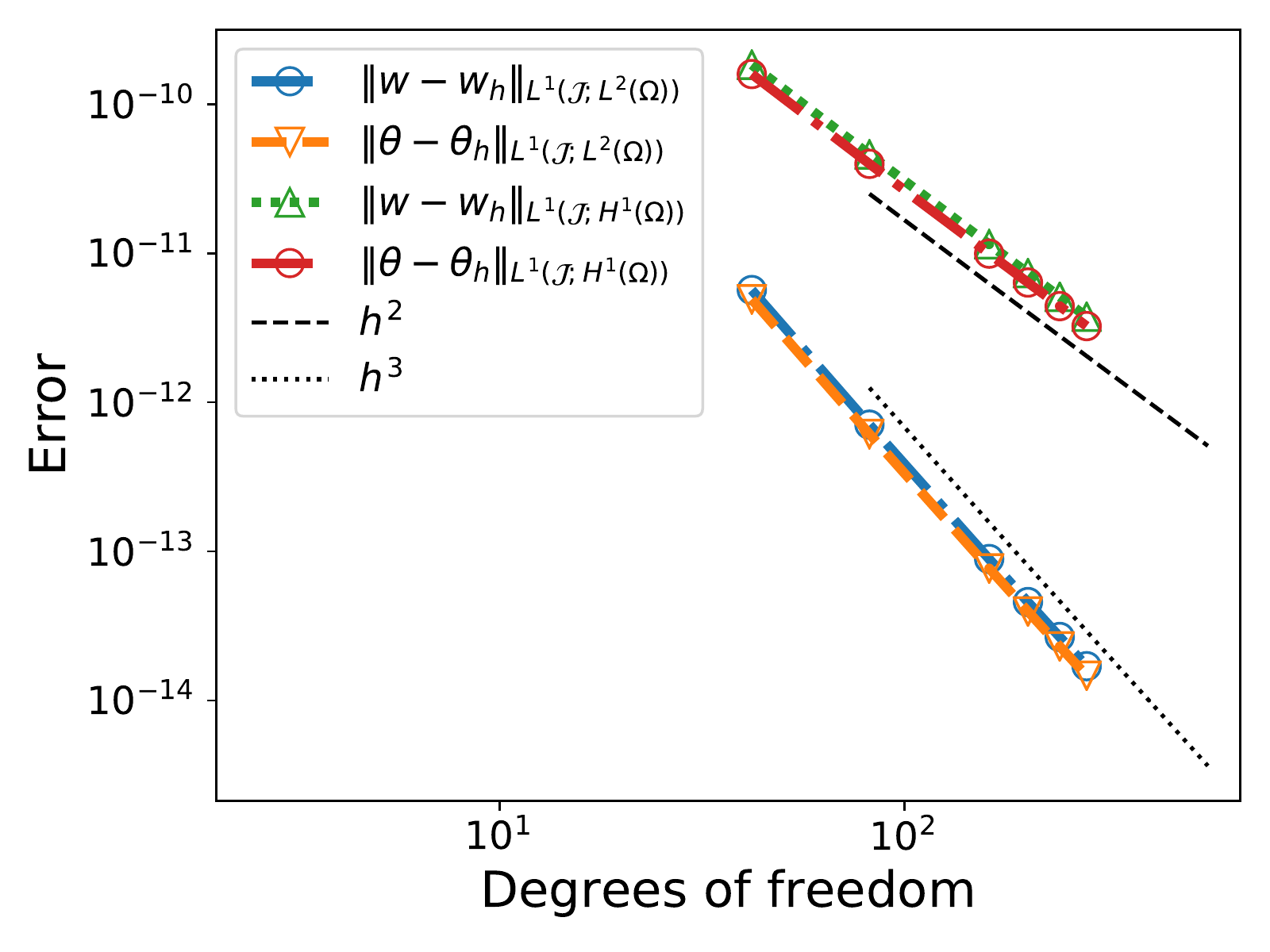}\\
				{\hspace{1cm}\footnotesize (d)}
			\end{center}
		\end{minipage}
		
		\caption{Effect of the thickness on the error  for the simply supported homogeneous viscoelastic beam with thickness $d=0.001m$ when using (a) exact integration with $\mathbb{P}_1$ elements,(b) reduced integration and $\mathbb{P}_1$ elements, (c) exact integration and $\mathbb{P}_2$ elements and (d) reduced integration and $\mathbb{P}_2$ elements.  The observation time is $T=10\,s$ with $5000$ and $32000$ time steps when using linear and quadratic elements, respectively. The values of $h, h^2$ and $h^3$ has been properly scaled to fit the graph. }
		\label{fig:qs_ex_0_exacta vs reducida}
	\end{figure}

	\section{Conclusions}
	\label{section:conclusions}
	
	We have presented an abstract functional framework to deal with mixed formulations for viscoelastic problems. We have shown the solvability of mixed viscoelastic formulations, by means of the well known mixed theory. The relevance is focused in the independence of the perturbation parameter in every estimate, since in the applications, numerical methods can be affected due this parameter. For this reason, we have made a rigorous analysis of the constants in each estimate.  With the well established theory of Volterra equations, we have proved convergence of mixed conforming numerical methods for the mixed viscoelastic problem, where the convergence is independent of the perturbation parameter. In order to apply the developed theory, we have considered a linear viscoelastic Timoshenko beam, which is well known for being a parameter dependent problem respect to the thickness, and consider its viscoelastic mixed formulation, which fits perfectly in our framework. The reported numerical tests show that the mixed numerical method is locking-free, as it happens in the elastic case. 

\bibliographystyle{amsplain}
\bibliography{bibliofond}  

\end{document}